\documentclass[a4paper,11pt]{article}
\usepackage[reqno]{amsmath}
\usepackage{amssymb}
\usepackage{wasysym}
\usepackage{mathrsfs}
\usepackage{accents}
\usepackage{verbatim}
\usepackage{color}
\usepackage{amsrefs}
\usepackage{amsthm}
\usepackage{dsfont}
\usepackage{graphicx}
\usepackage[all]{xy}
\usepackage[colorlinks=false]{hyperref}
\usepackage{mathtools}

\newcommand{\de}{\partial}

\newcommand{\take}{\,\backslash\,}

\newcommand{\wt}{\widetilde}
\newcommand{\Ideal}[1]{\left\langle#1\right\rangle}

\newcommand{\R}{\mathbb{R}}

\DeclareMathOperator{\range}{range}

\DeclareMathOperator{\End}{End}

\DeclareMathOperator{\Dom}{Dom}

\DeclareMathOperator{\vol}{vol}

\newcommand{\ol}{\overline}

\newcommand{\wh}{\widehat}

\newcommand{\Cl}{\mathbb{C}\ell}
\newcommand{\C}{\mathbb{C}}

\newcommand{\D}{\mathcal{D}}
\newcommand{\A}{\mathcal{A}}
\newcommand{\h}{\mathcal{H}}
\newcommand{\J}{\mathcal{J}}

\newcommand{\F}{\mathcal{F}}

\newcommand{\Z}{\mathbb{Z}}
\DeclareMathOperator{\Ext}{Ext}

\newtheorem{thm}{Theorem}[section]

\newtheorem{lemma}[thm]{Lemma}
\newtheorem{prop}[thm]{Proposition}
\newtheorem{cor}[thm]{Corollary}
\theoremstyle{definition}
\newtheorem{defn}[thm]{Definition}
\newtheorem*{notn}{Notation}
\newtheorem{eg}[thm]{Example}
\newtheorem{rem}[thm]{Remark}

\newtheorem{question}{Question}

\DeclareFontFamily{OT1}{pzc}{}
\DeclareFontShape{OT1}{pzc}{m}{it}{<-> s * [1.1] pzcmi7t}{}
\DeclareMathAlphabet{\mathpzc}{OT1}{pzc}{m}{it}

\usepackage{slashed}

\newcommand{\Lip}{\textnormal{Lip}}
\newcommand{\K}{\mathbb{K}}
\newcommand{\B}{\mathbb{B}}

\newcommand{\Dsla}{\slashed{\mathcal{D}}}

\newtheorem{ass}{Assumption}

\newtheorem*{thm*}{Theorem} 

\author{Iain Forsyth$^{**}$, Magnus Goffeng$^*$, Bram Mesland$^{**}$, Adam Rennie$^\dag$ \\[4pt]
${}^*$Department of Mathematical Sciences,\\ Chalmers University of Technology and the University of Gothenburg,\\ SE-412 96 Gothenburg, Sweden\\[3pt]
${}^{**}$Institut f\"ur Analysis, Leibniz Universit\"at Hannover, Welfengarten 1,\\ 30167 Hannover, Germany
\\[3pt]
${}^\dag$	School of Mathematics and Applied Statistics,
University of Wollongong,\\
Northfields Ave 2522, Australia}
\title{
\huge\textsc{Boundaries, spectral triples and $K$-homology}}
\begin{document}
\maketitle

\begin{abstract}  
\parindent=0.0pt
\parskip=2pt
\noindent This paper extends the notion of a spectral triple 
to a relative spectral triple, an unbounded
analogue of a relative Fredholm module for an ideal $J\triangleleft A$. 
Examples include manifolds with boundary,
manifolds with conical singularities, dimension drop algebras, 
$\theta$-deformations and Cuntz-Pimsner algebras of vector bundles.

\noindent The bounded transform of a relative spectral triple 
is a relative Fredholm module, 
making the image of a relative spectral triple under the boundary mapping
in $K$-homology easy to compute. 
We introduce an additional operator called a Clifford normal
with which a relative spectral triple can be doubled into a spectral triple. 

The Clifford normal also provides a boundary Hilbert space, 
a representation of the quotient algebra, a boundary Dirac operator and
an analogue of the Calderon projection. In the examples this data does assemble
to give a boundary spectral triple, though we can not prove this in general.

When we do obtain a boundary spectral triple, we provide sufficient conditions 
for the boundary triple to represent the $K$-homological boundary.
Thus we abstract the proof of Baum-Douglas-Taylor's 
``boundary of Dirac is Dirac on the boundary'' 
theorem into the realm of 
non-commutative geometry.
\end{abstract} 

\parindent=0.0pt
\parskip=4pt

\section{Introduction}
\label{sec:intro}

This paper puts to the test the folklore idea that manifolds with boundary can be
modelled in non-commutative geometry using symmetric (Dirac-type) operators. 
Numerous constructions work as expected from the classical case, but some
results one would expect to be true turn out to require substantial additional data or hypotheses.

The motivations for such an investigation come from at least three sources.
Firstly, in recent years, the importance of secondary 
invariants in geometry has been realised 
(see for instance \cites{hilsumskand,lotthighereta,PSrhoInd,wahlsurgery,xieyu}), 
and the need to consider Dirac-type operators 
on manifolds with boundary in their study \cite{DG3}.
Secondly, having a relative spectral triple representing a $K$-homology class $x$
should facilitate the computation of the image $\partial x$ of $x$ under the
$K$-homology boundary map $\partial$: a difficult problem in general, \cite{BDT}. 
Finally, internal to non-commutative geometry there is the 
more philosophical question of what 
should be the non-commutative analogue of a manifold with boundary.

Historically, attempts to incorporate manifolds with boundary 
into non-commutative geometry have either imposed self-adjoint 
boundary conditions on a Dirac operator (like APS), which
generically collapse the boundary to a point, for example \cites{battistiseiler,levyiochum},
or they have made the Dirac operator self-adjoint by pushing the boundary
to infinity. In fact APS boundary conditions are closely related to the latter framework. 
More general constructions are possible, for instance \cite{kaadrecent}.

The main idea in this paper is to loosen the 
self-adjointness condition appearing in a spectral triple, 
requiring only a symmetric operator with additional
analytic properties relative to an ideal. Sections \ref{sec:rel-mods} and \ref{sect:normal} and the appendices are based on results from the Ph.D thesis of the first named author, \cite{ForsythThesis}.

\subsection{The main results}
\label{subsec:main}

Our approach is inspired by the work of Baum-Douglas-Taylor \cite{BDT},  
and the main players in this paper are relative spectral triples and Kasparov modules,
introduced in Section \ref{sec:rel-mods}. 
We work with a $\Z/2$-graded $C^*$-algebra $A$ and a closed graded 
$*$-ideal $J\triangleleft A$. Loosely speaking, a relative spectral triple for $J\lhd A$ 
is a triple $(\J\lhd \A,\h,\D)$ satisfying the usual axioms of a spectral triple for the 
dense $*$-subalgebra $\A\subseteq A$ 
save the fact that the odd operator $\D$ is symmetric and 
$j\Dom(\D^*)\subseteq \Dom(\D)$ for all $j\in \J$, 
where $\J$ is a dense $\ast$-subalgebra of $J$ which is an ideal in $\A$. 
For the precise definition, see Definition \ref{defn:alternate} 
on page \pageref{defn:alternate}, and we follow the definition with numerous examples.
The first important result is that relative spectral triples give relative Fredholm modules,
and so relative $K$-homology classes.

\begin{thm*}[Bounded transforms of relative spectral triples]
Let $(\J\lhd \A,\h,\D)$ be a relative spectral triple for $J\lhd A$ and define the odd operator 
$F_\D:=\D(1+\D^*\D)^{-1/2}$. For any $a\in A$ and $j\in J$, the operators 
$[F_\D,\rho(a)]_\pm$, $\rho(j)(F_\D-F^*_\D)$ and $\rho(j)(1-F^2_\D)$ are 
compact for all $a\in A$, $j\in J$. In particular, $(\h,F_\D)$ is a relative Fredholm 
module for $J\triangleleft A$, as defined in Definition \ref{relativefrehdolm}.
\end{thm*}

This theorem can be found as Theorem \ref{thm:alternate} below. 
While it is the expected result, the proof is significantly 
more subtle than the case when $\D$ is self-adjoint, so we relegate the main proofs of Section \ref{sect:boundedtransform} to Appendix 
\ref{sec:technicalproof}.  In addition, the proof is given for the more general
`relative Kasparov modules' which we describe in Appendix \ref{subsec:Khom}.

Relative Fredholm modules contain more information than the cycle information
for the ideal $J$. In particular, if $A\to A/J$ is semisplit, 
an even relative Fredholm module contains further information 
with which one can compute the image in $K^1(A/J)$ under the boundary map
(see Proposition \ref{eq:bdrymap}). More precisely, we prove in 
Subsection \ref{subsec:bdry-map}, 
that $\de[(\J\lhd \A,\h,\D)]\in K^1(A/J)$ coincides with the 
$K$-homology class that the extension 
$\ker((\D^*)^+)$ defines under the injection $ K^1(A/J)\to \Ext(A/J)$.

Classically, the geometry of
the boundary is given, along with a compatibility with the total geometry.
Our aim is to incorporate additional information into the relative spectral triple
in order to be able to construct all the elements of the boundary geometry.

In Section \ref{sect:normal} we attempt to encode this
information  in a ``Clifford normal", an additional
operator which plays a subtle analytic role 
(cf. Equation \eqref{ndomdsforman} below).
 A bounded odd operator $n$ is called a Clifford normal 
 for a relative spectral triple $(\J\lhd\A,\h,\D)$ 
if it preserves $\Dom(\D)$ and a core for $\Dom(\D^*)$ 
plus additional technical properties guaranteeing 
that $n$ behaves like the Clifford multiplication 
of a vector field having length $1$ and being normal 
near the boundary of a 
manifold. 

An important point, which we expand on below, is that 
the Clifford normal is a bounded operator 
on $\h$ but an unbounded operator on $\Dom(\D^*)$. This
means that in general we must control the analytic behaviour of
$n$. This is done in the  assumptions in Definition \ref{defn:normal}  
on page \pageref{defn:normal}, and these must be checked in examples.
The existence of a Clifford normal allows for a novel 
construction of a double, analogous 
to the doubling of a manifold.

\begin{thm*}[Doubling a relative spectral triple with a Clifford normal]
Let $(\J\lhd\A,\h,\D,n)$ be an even relative spectral triple with Clifford normal for $J\lhd A$. 
Define $\wt{\A}:=\{(a,b)\in \A\oplus \A: a-b\in \J\}$ 
and $\wt{\D}$ as the restriction of $\D^*\oplus \D^*$ to 
$$
\Dom(\wt{\D})=\{(\xi,\eta)\in \Dom(n)\oplus\Dom(n): \eta-n\gamma\xi\in \Dom(\D)\}.
$$
Then $(\wt{\A},\h\oplus\h^{\text{op}},\wt{\D})$ is a spectral triple.
\end{thm*}

In Section \ref{sect:tripleonbdry}, 
with modest assumptions on the even relative spectral triple with Clifford normal 
$(\J\lhd\A,\h,\D,n)$, we show how to define a  Calderon 
projector, Poisson operator and  candidates for a boundary representation of
$A/J$ on the boundary Hilbert space $\mathfrak{H}_n$ and boundary Dirac operator. 
The problematic
inner product on $\mathfrak{H}_n$ prevents us from proving that {\em any} of these
objects  behaves as in the classical case. Nevertheless, in all our
examples, including the non-commutative ones (see below), 
we can check that these constructions do indeed provide
a boundary spectral triple.

Having natural definitions for all the boundary objects means that it is relatively
straightforward to check in examples whether they satisfy all the hoped-for properties,
namely that they assemble to yield a boundary spectral triple.
In the abstract, it appears that our framework does not suffice to prove
that we can construct a boundary spectral triple.

When we assume that our definitions do indeed yield a boundary spectral triple, 
see Assumption \ref{sixthass} on page \pageref{sixthass}, 
then we can consider additional (much milder) assumptions 
on $(\J\lhd\A,\h,\D,n)$ 
that guarantee that the 
even spectral triple $(\A/\J\otimes \Cl_1,\mathfrak{H}_n, \D_n)$ represents
the class $\partial[(\J\triangleleft \A,\h,\D)]$ in $K^1(A/J)\cong K^0(A/J\otimes \Cl_1)$.
We state these final results in  Theorem \ref{thm:boundarycomparisonnew} 
and Proposition \ref{prop:boundarycomparisonnew} 
(see pages \pageref{thm:boundarycomparisonnew} 
and \pageref{prop:boundarycomparisonnew} respectively).

Throughout the text we examine what all our constructions mean for 
manifolds with boundary, manifolds with conical singularities and dimension
drop algebras. In Section \ref{sec:onexamples} we complete our discussion
of these examples, and also present examples including crossed products
arising from group actions on
manifolds with boundary, $\theta$-deformations of manifolds with boundary and
Cuntz-Pimsner algebras of vector bundles on manifolds with boundary.

\subsection{The classical setup}
\label{subsec:manifold-bad}

To explain the source of the technical difficulties we encounter later, we review the
analytic subtleties of our principal examples: manifolds with boundary. 
The issues we discuss here are well-known to people working on 
boundary value problems, and are presented in more detail in \cite{BaerBall}.
Importantly, the difficulties are present even classically. 

Suppose that $\Dsla$ is a Dirac-type operator on a 
Clifford bundle $S$ over a manifold $M$ with boundary. 
For simplicity, we assume that $M$ is compact, but for the 
purposes of this subsection it suffices to assume 
that the boundary of $M$ is compact. We denote the operator on the boundary 
by $\Dsla_{\partial M}$. 

When $\Dsla$ is a Dirac 
operator on a manifold with boundary we distinguish its defining 
differential expression, which we denote by 
$\Dsla$, and its various realizations that we denote by $\D$ 
with a subscript indicating the boundary conditions. 
We take the minimal closed extension $\D_{\textnormal{min}}$, 
the closure of $C^\infty_c(M^\circ,S)$ 
in the graph norm defined from $\Dsla$, and let 
$\D_{\textnormal{max}}=(\D_{\textnormal{min}})^*$ 
denote the $L^2$-adjoint. 

The domain of $\D_{\textnormal{min}}$ 
is $H^1_0(M,S)$. The trace mapping
$$
R:\Gamma^\infty(\ol{M},S)\to\Gamma^\infty(\de M,S|_{\de M})
$$
extends to a continuous mapping on $\Dom(\D_{\textnormal{max}})$ 
that fits into a short exact sequence
$$
0\to \Dom(\D_{\textnormal{min}})
\to \Dom(\D_{\textnormal{max}})\xrightarrow{R} \check{H}(\Dsla_{\partial M})\to 0,
$$
where $\check{H}(\Dsla_{\partial M})$ is defined 
as follows. We choose a real number $\Lambda$:
zero would do. For $s\in \R$, we consider the 
subspaces of $H^s(\partial M,S|_{\partial M})$ defined by
\begin{align*}
H^{s}_{[\Lambda,\infty)}(\Dsla_{\partial M})&:=\chi_{[\Lambda,\infty)}(\Dsla_{\partial M})
H^s(\partial M,S|_{\partial M})
\quad\mbox{and}\\
H^{s}_{(-\infty,\Lambda)}(\Dsla_{\partial M})&:=\chi_{(-\infty,\Lambda)}(\Dsla_{\partial M})
H^s(\partial M,S|_{\partial M}).
\end{align*}
We define the Hilbert space
\begin{equation}
\label{staronpage6}
\check{H}(\Dsla_{\partial M})
=H^{-1/2}_{[\Lambda,\infty)}(\Dsla_{\partial M})+H^{1/2}_{(-\infty,\Lambda)}(\Dsla_{\partial M}).
\end{equation}
The two spaces $H^{-1/2}_{[\Lambda,\infty)}(\Dsla_{\partial M})$ and 
$H^{1/2}_{(-\infty,\Lambda)}(\Dsla_{\partial M})$ are 
orthogonal. For further details on the maximal domain of a Dirac operator on a 
manifold with boundary, see \cite{BaerBall}*{Section 6}.

The space $\check{H}(\Dsla_{\partial M})$ is rather complicated. 
It is not contained in the space of $L^2$-sections on the boundary $\partial M$.
Worse still, we find that Clifford multiplication by the unit normal $n$
does not preserve $\check{H}(\Dsla_{\partial M})$. 
Hence, any smooth prolongation of the normal
to $M$ does not preserve $\Dom(\D_{\textnormal{max}})$. 
In fact if $\sigma\in \Dom(\D_{\textnormal{max}})$ so that 
$R(\sigma)\in \check{H}(\Dsla_{\partial M})$, 
we find that $R(n\sigma)\in \hat{H}(\Dsla_{\partial M})$
where 
$$
\hat{H}(\Dsla_{\partial M})=\check{H}(-\Dsla_{\partial M})
=H_{[\Lambda,\infty)}^{1/2}(\Dsla_{\partial M})+H_{(-\infty,\Lambda)}^{-1/2}(\Dsla_{\partial M}).
$$
The domain of $n$, as a densely defined operator on $\Dom(\D_{\textnormal{max}})$ is 
the space 
\begin{align}
\nonumber
\Dom(\D_{\textnormal{max}})\cap n\Dom(\D_{\textnormal{max}})&=H^1(M,S)\\
\label{ndomdsforman}
&=\{\sigma\in \Dom(\D_{\textnormal{max}}): 
R(\sigma)\in  \hat{H}(\Dsla_{\partial M})\cap \check{H}(\Dsla_{\partial M})\}.
\end{align}
The fact that $\Dom(\D_{\textnormal{max}})$ is not 
preserved by $n$ leads to several subtleties 
in our development. The operator that plays the 
role of the Clifford normal for a relative spectral triple $(\J\lhd\A,H,\D)$
will need to be considered as an 
unbounded operator on $\Dom(\D^*)$, despite defining 
a bounded operator on the ambient Hilbert 
space. 

The technical way around these problems is found in the construction of the
self-adjoint double. When constructing the double, 
one glues two copies of the manifold with boundary along the
boundary. Since the orientation is changed 
on one copy, using multiplication by the Clifford normal, 
the domain of the Dirac operator on the double 
manifold is the intersection of spaces where the 
boundary value belongs to $\hat{H}(\Dsla_{\partial M})$ 
and $\check{H}(\Dsla_{\partial M})$, respectively. 
Therefore, the doubling of manifolds only uses 
$H^1(M,S)$ and the problems with $\Dom(\D_{\textnormal{max}})$ 
disappear. This approach works also in the general case of relative spectral triples with Clifford normals. 

{\bf Acknowledgements}
The first and fourth listed authors received support from the Australian Research Council. 
The second listed author was supported by the Swedish Research Council Grant 2015-00137 
and Marie Sklodowska Curie Actions, Cofund, Project INCA 600398 and would also like to thank the 
Knut and Alice Wallenberg foundation for their support. All authors thank the Institute for Analysis at the
Leibniz University of Hannover, The University of Wollongong,
the Centre for Symmetry and Deformation (Copenhagen), Chalmers University of 
Technology and the University of Gothenburg for facilitating the collaboration. 
The authors also express gratitude to 
Mathematisches Forschungsinstitut Oberwolfach 
and to 
MATRIX (University of Melbourne/Monash University/ACEMS). 

\section{Relative Kasparov modules}
\label{sec:rel-mods}

In this section we introduce the basic tools for treating boundaries: 
relative unbounded Kasparov modules, e.g. relative spectral triples.  
We work bivariantly 
with a group action by a compact group $G$, in order 
to widen the range of possible constructions. Additional details concerning relative $KK$-theory
have been collected in Appendix \ref{subsec:Khom}.
 
The main technical result of the section states that the bounded transform of relative unbounded 
Kasparov modules are relative Kasparov modules. In particular, relative spectral triples provide 
unbounded representatives of relative $K$-homology classes. 
Finally, we show that the image of the class of a relative unbounded Kasparov module in 
$KK$-theory under the boundary mapping is computable by means of extensions.

\begin{notn}
We let $A$ and $B$ denote $\Z/2$-graded $C^*$-algebras. For $a,b\in A$, 
$[a,b]_\pm=ab-(-1)^{|a||b|} ba$ denotes the graded commutator 
for homogeneous elements $a$ and $b$ of degree $|a|$ and $|b|$, respectively.
The adjointable operators on a countably generated $C^*$-module $X_B$ will
be denoted by $\End^*_B(X_B)$ and the compact endomorphisms by $\End^0_B(X_B)$. 
Elements of $\End^0_B(X_B)$ are also called $B$-compact.
\end{notn}

\subsection{Relative Kasparov modules}
\label{sect:relativetriples}
We introduce relative unbounded Kasparov modules, defined using symmetric operators.
This notion is related to the ``half-closed cycles'' studied in \cite{Hilsum}, 
and also in \cite{DGM}\footnote{The relative unbounded Kasparov modules define 
half-closed cycles in non-relative unbounded $KK$-theory: so if $(\J\lhd\A,X_B,\D)$ 
is a relative unbounded Kasparov module for $(J\triangleleft A,B)$, 
then $(\J,X_B,D)$ is a half-closed cycle for $(J,B)$.}.

\begin{defn}
\label{defn:alternate}
Let $G$ be a compact group, $A$ and $B$ be 
$\Z/2$-graded $G$-$C^*$-algebras with $A$ separable and $B$ $\sigma$-unital, 
and let $J\lhd A$ be a $G$-invariant graded ideal. 
A $G$-equivariant \textbf{even relative unbounded Kasparov module} 
$(\J\lhd \A,X_B,\D)$ for $(J\lhd A,B)$ consists 
of an even $G$-equivariant representation 
$\rho:A\rightarrow \End^*_B(X_B)$ on a $\Z/2$-graded countably generated 
$G$-equivariant $B$-Hilbert $C^*$-module $X_B$, 
an odd $G$-equivariant closed symmetric regular operator 
$\D:\Dom(\D)\subset\h\rightarrow X_B$, and dense 
sub-$\ast$-algebras $\J\subseteq J$ and $\A\subseteq A$, with $\J$ an ideal in $\A$, 
such that:
\begin{enumerate}
\item $\rho(a)\cdot\Dom(\D)\subset\Dom(\D)$ and the 
graded commutator $[\D,\rho(a)]_\pm$ is bounded for all $a\in\A$, 
and hence $[\D,\rho(a)]_\pm$ extends to an adjointable operator;
\item $\rho(j)\cdot\Dom(\D^*)\subset\Dom(\D)$ for all $j\in\J$;
\item $\rho(a)(1+\D^*\D)^{-1/2}$ is $B$-compact for all $a\in\A$;
\item[4.] $\ker(\D^*)$ is a complemented submodule of $X_B$, 
and $\rho(a)(1-P_{\ker(\D^*)})(1+\D\D^*)^{-1/2}$ is $B$-compact for all $a\in\A$.
\end{enumerate}
If $A$ and $B$ are trivially $\Z/2$-graded, 
a $G$-equivariant \textbf{odd relative unbounded Kasparov module} 
$(\J\lhd\A,X_B,\D)$ for $(J\lhd A,B)$ has the same definition except that $X_B$ 
is trivially $\Z/2$-graded and $\D$ need not be odd.
A \textbf{relative spectral triple} for $J\lhd A$ is a relative unbounded Kasparov module for $(J\lhd A,\C)$.
We will usually omit the notation $\rho$.
\end{defn}

In order to address various subtleties in the definition of relative spectral triple,
we now make a series of remarks highlighting some key points.

\begin{rem}
\label{rem:unitalandrwongcompact}
If $A$ is unital and represented non-degenerately on $\h$ and $\ker(\D)$ and $\ker(\D^*)$ are complemented submodules of $X_B$, then Condition 3. implies Condition 4. by the following argument. Let $V$ be the phase of $\D$, which is the partial isometry with initial space $\ker(\D)^\perp$ and final space $\ker(\D^*)^\perp$ defined by $\D=V|\D|$, \cite{ReedSimon1}*{Thm. VIII.32}. Then $\D\D^*=V|\D|^2V^*=V\D^*\D V^*$, which implies that 
$$(1-P_{\ker(\D^*)})(1+\D\D^*)^{-1/2}=VV^*(1+\D\D^*)^{-1/2}=V(1+\D^*\D)^{-1/2}V^*,$$
and hence $(1-P_{\ker(\D^*)})(1+\D\D^*)^{-1/2}$ is $B$-compact.
\end{rem}

\begin{rem}
\label{stabforclosed}
Note that $\ker(\D^*)$ being a complemented submodule of $X_B$ is equivalent to the inclusion 
$\ker(\D^*)\to X_B$ being adjointable. This is automatic in the case that $B=\C$. 
The modules $\ker(\D^*)$ and $\ker(\D)$ are complemented if $\D$ or $\D^*$ 
has closed range, see \cite{Lance}*{Theorem 3.2}. 
In fact, if $A$ is unital and is represented non-degenerately, 
we can, after stabilizing $X_B$ by a finitely generated 
projective $B$-module on which $A$ acts trivially 
and modifying $\D^*$ by a finite rank operator, 
always obtain that the range of $\D^*$ is closed. 
The proof of this fact goes as in \cite{DGM}*{Lemma 3.6}.
\end{rem}

\begin{rem}
\label{rem:alternatecondition}
The key result that the bounded transform of
a relative spectral triple yields a relative Fredholm module 
and so a $K$-homology class, proved in Theorem \ref{thm:alternate},
can be proved starting from the definition above,
or another variant. The alternative definition replaces
Condition 4. of Definition \ref{defn:alternate}  by 
the following more checkable condition.
\begin{itemize}
\item[4'.]  there exists a sequence $(\phi_k)_{k=1}^\infty\subset\rho(\A)$ 
such that $\phi_k[\D,\rho(a)]_\pm\rightarrow[\D,\rho(a)]_\pm$ 
and $[\D,\rho(a)]_\pm\phi_k\rightarrow[\D,\rho(a)]_\pm$ in operator norm for all $a\in\A$.
\end{itemize}
If $\A$ is unital and the representation $\rho$ is non-degenerate, 
then clearly Condition 4'. is satisfied. Compare 4'. to \cite{HigsonRoe}*{Exercise 10.9.18}.
While 4' is more checkable in examples, it  does not suffice for some of our
key constructions, except in the unital case as noted in 
Remark \ref{rem:unitalandrwongcompact}.

Either Condition 4. or Condition 4'. can be used together with Conditions 1.- 3. 
to show that $[\D(1+\D^*\D)^{-1/2},a]_\pm$ 
is $B$-compact for all $a\in\A$. Neither Condition 4. nor 
Condition 4' are needed to prove this compact commutator condition if $a(1+\D\D^*)^{-1/2}$ 
is $B$-compact for all $a\in\A$ (which is not generally the case), 
which is why such conditions do not appear in the definition of an unbounded Kasparov module. 
\end{rem}

\begin{rem}
\label{functorrem}
Relative Kasparov modules depend contravariantly on $J\lhd A$ in the following sense. 
Consider ideals $J_i\lhd A_i$, for $i=1,2$, and fix dense $*$-subalgebras 
$\J_i\subseteq J_i$, $\A_i\subseteq A_i$, with $\J_i\lhd \A_i$.
Let $\varphi:A_1\to A_2$ be an equivariant $*$-homomorphism with $\varphi(\A_1)\subseteq \A_2$ and 
$\varphi(\J_1)\subseteq \J_2$. If $(\J_2\lhd \A_2,X_B,\D)$ is a relative Kasparov module for 
$(J_2\lhd A_2,B)$, then $\varphi^*(\J_2\lhd \A_2,X_B,\D):=(\J_1\lhd\A_1,{}_\varphi X_B,\D)$ is 
a relative Kasparov module for $(J_1\lhd A_1,B)$. Here ${}_\varphi X_B$ denotes $X_B$ 
with the left action of $A_1$ induced by $\rho\circ \varphi$ (where $\rho:A_2\to \End_B^*(X_B)$ 
denotes the left action of $A_2$). Compare to Remark \ref{functorrembounded} on page 
\pageref{functorrembounded}.
\end{rem}

\begin{rem}
\label{rem:j-compact}
While $a(1+\D^*\D)^{-1/2}$ is $B$-compact for all $a\in\A$, typically $j(1+\D\D^*)^{-1/2}$ 
is $B$-compact only for $j\in J$. This follows since for $j_1,j_2\in\J$,
$$
\h\xrightarrow{(1+\D\D^*)^{-1/2}}\Dom(\D^*)\stackrel{j_1}{\rightarrow}\Dom(\D)\stackrel{j_2}{\to} \h
$$
and the map $j_2:\Dom(\D)\rightarrow\h$ is compact since $j_2(1+\D^*\D)^{-1/2}$ is compact. 
This argument shows that $j(1+\D\D^*)^{-1/2}$ is compact for $j\in \J^2$ and by a density argument 
the same holds for $j\in J$.
\end{rem}

\begin{rem}
\label{almostequi}
We can loosen the assumption on the $G$-equivariance of $\D$ 
to almost $G$-equivariant: namely $G$ acts 
strongly continuously on $\Dom(\D)$ and $B_g:=g\D g^{-1}-\D$ is adjointable on $X_B$ 
for all $g\in G$. If $\D$ satisfies the latter assumptions, 
$\D$ is a locally bounded perturbation of the $G$-equivariant 
operator $\int_G g\D g^{-1} \mathrm{d}g$, 
where the integral is interpreted pointwise as an operator 
$\Dom(\D)\to X_B$.
\end{rem}

\begin{rem}
\label{lipremark}
We topologise $\A$ using the Lip-topology defined from the norm
$$
\|a\|_\Lip:=\|a\|_A+\|[\D,\rho(a)]\|_{\End_B^*(X_B)}.
$$
It is immediate from the construction that if $(\J\lhd\A,X_B,\D)$ is a 
relative unbounded Kasparov module, then so is 
$$
(\overline{\J}^{\Lip}\lhd \overline{\A}^{\Lip},X_B,\D).
$$
It is unclear if it holds that $(J\cap \overline{\A}^{\Lip} \lhd \overline{\A}^{\Lip},X_B,\D)$ 
is a relative unbounded Kasparov module in general. 
We note that $\ker(\A\to A/J)=J\cap \A$, so if $\J\neq J\cap \A$ 
then $\A/\J\to A/J$ is not injective. 
It follows from \cite{BlackadarCuntz} that both 
$\overline{\J}^{\Lip}$ and $\overline{\A}^{\Lip}$ are 
closed under the holomorphic functional calculus inside $J$ and $A$, respectively.
\end{rem}

\subsection{Examples}
\label{subsec:examples}

\subsubsection{Manifolds with boundary}
\label{eg:diracboundary}
Let $M$ be an open submanifold of a complete Riemannian manifold 
$\wt{M}$, and let $S$ be a (possibly $\Z/2$-graded) 
Clifford module over $M$ with Clifford connection $\nabla$, 
which we assume extend to $\wt{M}$. To emphasize when 
we consider the open manifold and its closure, 
we write $M^\circ$ and $\ol{M}$, respectively. Let $\D_{\textnormal{min}}$ be the 
closure of the Dirac operator $\Dsla$ acting on smooth sections of 
$S$ with compact support in $M$. We use the notation 
$C^\infty_0(M^\circ):=C^\infty_c(\ol{M})\cap C_0(M^\circ)$. 
Here $C^\infty_c(\ol{M})$ is the space of restrictions of 
elements from $C^\infty_c(\wt{M})$ to $\ol{M}$. 

Then
$(C^\infty_0(M^\circ)\lhd C^\infty_c(\ol{M}),L^2(S),\D_{\textnormal{min}})$ 
is a relative spectral triple for $C_0(M^\circ)\lhd C_0(\ol{M})$, 
which is even if and only if $S$ is $\Z/2$-graded. 
It is not hard to see that Conditions 1. and 2. of Definition \ref{defn:alternate} 
are satisfied, and $f(1+\D^*_{\textnormal{min}}\D_{\textnormal{min}})^{-1/2}$ is compact for all 
$f\in C_0(\ol{M})$ by elliptic operator theory, 
in particular the Rellich Lemma and the identification of 
$\Dom(\D_{\textnormal{min}})$ with the closure of $\Gamma^\infty_c(M;S)$ 
in the first Sobolev space, \cite{BDT}*{Proposition 3.1}, \cite{HigsonRoe}*{10.4.3}.
Condition 4'. of Remark \ref{rem:alternatecondition} is always satisfied. 
Condition 4. of Definition \ref{defn:alternate} is satisfied when $M$ is a 
compact manifold with boundary, by the discussion in Remark \ref{rem:unitalandrwongcompact}.

In particular we obtain a relative spectral triple when 
$\ol{M}$ is a complete Riemannian manifold with 
boundary, although the case when $M$ is an open 
submanifold of a complete manifold is much more 
general. For concrete examples of relative 
spectral triples for  manifolds with boundary, see \cite{FMR}.
We note that by \cite[Theorem 2.12 and 2.17]{hewettmoiola}, 
$\D_{\textnormal{min}}$ is 
self-adjoint if $\wt{M}\take M$ is a submanifold of 
codimension greater than or equal to $2$.

The importance of the next example will be seen when we discuss doubles of relative
spectral triples in Subsection \ref{sec:double}.

\subsubsection{Relative spectral triples and extensions}
\label{technicallytrashtalking}
Relative unbounded Kasparov modules behave well 
under closed symmetric extensions of the operator $\D$. 
Classically, this corresponds to choosing symmetric 
boundary conditions. Let $(\J\lhd\A,X_B,\D)$ 
be a relative unbounded Kasparov module for $(J\lhd A,B)$. 
If $\hat{\D}$ is a regular closed symmetric extension of $\D$ 
with $\ker(\hat{\D}^*)$ complemented such that 
$a(1+\hat{\D}^*\hat{\D})^{-1/2}$ and $a(1-P_{\ker(\hat{\D}^*)})(1+\hat{\D}\hat{\D}^*)^{-1/2}$ are 
compact for all $a\in\A$, we can consider the subalgebra 
\begin{align*}
\hat{\A}&:=\{a\in \mathcal{A}: \, a\Dom(\hat{\D})\subseteq \Dom(\hat{\D})\}\quad\mbox{and the ideal}\\
\hat{\J}&:= \{j\in \hat{\A}: \,j\Dom(\hat{\D}^*)\subseteq \Dom(\hat{\D})\}.
\end{align*}
Let $\hat{A}$ and $\hat{J}$ denote the 
$C^*$-closures of $\hat{\J}$ and $\hat{\A}$, respectively. 
Then $(\hat{\J}\lhd \hat{\A},X_B,\hat{\D})$ is a relative 
unbounded Kasparov module for $(\hat{J}\lhd \hat{A},B)$. Note that if $(\J\lhd\A,X_B,\D)$ is even, then $\hat{\D}$ must be chosen to be odd.

Let us discuss the minimal closed extension $\D_{\textnormal{min}}$ of a Dirac operator on 
a compact manifold $\overline{M}$ with boundary in this context, so we have $B=\C$. In the  
general context of a $C^*$-algebra $B$, the discussion applies through the following construction:
given a smooth Hermitian $B$-bundle $\mathcal{E}_B\to \overline{M}$ with a connection  
and a closed symmetric extension $\hat{\D}$ of $\D_{\textnormal{min}}$, the twisted 
Dirac operator $\hat{\D}_\mathcal{E}$ is a closed regular symmetric extension of the twisted 
Dirac operator $\D_{\textnormal{min},\mathcal{E}}$ (see \cite{DGM}*{Subsection 1.5}).

For $\mathcal{A}=C^\infty(\overline{M})$ and $\hat{\D}$ 
the APS-extension of $\D_{\textnormal{min}}$, 
it follows from \cite[Proposition 4.7]{battistiseiler} that $\hat{\A}$ is the minimal unitisation of 
$C^\infty_0(M^\circ)$ (when $\partial M$ is connected: more generally, $\hat{\A}$ is the algebra of 
smooth functions which are locally constant when restricted to $\partial M$). Self-adjointness 
of the APS-extension implies that $\hat{\J}=\hat{\A}$. 
In this case, the subalgebra $\hat{\A}\subseteq \A$ 
corresponds to the quotient mapping $\overline{M}\to \overline{M}/\partial M\cong M\cup\{\infty\}$ 
that collapses the boundary of $M$ to a point.

In \cite{levyiochum}, chiral boundary conditions on a 
Dirac type operator were considered for an even-dimensional manifold $M$. 
The details of chiral boundary conditions are discussed in 
\cite{levyiochum}*{Subsection 4.2}. Let $\hat{\D}$ denote said extension, which is self-adjoint. 
By \cite{levyiochum}*{Theorem 4.5}, 
$(C^\infty(\overline{M}),L^2(M,S),\hat{\D})$ is an \emph{odd} spectral triple. 
The reason that the spectral triple is odd despite $M$ being even-dimensional is 
that chiral boundary conditions are not graded. 
In fact, the existence of a self-adjoint extension $\D_e$ defining an 
even spectral triple $(C^\infty(\overline{M}),L^2(M,S),\D_e)$ is obstructed 
by the class $\de[\D]\in K^1(C(\de M))$, where $[\D]$ is the class in $K^0(C_0(M^\circ))$ 
defined by the Dirac operator, which is independent of the choice of extension
(see Remark \ref{remarkonextensions}). An extensive discussion of spectral flow for 
(generalised) chiral boundary conditions appears in \cite{GLesch}.
In particular, \cite[Theorem 3.3]{GLesch} shows that $[(C^\infty(\overline{M}),L^2(M,S),\hat{\D})]$ 
is a torsion element in $K^1(C(\ol{M}))$. 

\subsubsection{Dimension drop algebras}
\label{dimdropsubex}
Let $\ol{M}$ be a manifold with boundary with Clifford module $S$ and Dirac operator $\Dsla$, 
as in subsection \ref{eg:diracboundary}. For $B\subset M_N(\C)$ a $C^*$-subalgebra, define
$$
\A=\{f\in C^\infty_c(\ol{M},M_N(\C)):\, f(x)\in B\,\,\,\forall x\in\de M\}.
$$
A typical choice for $B$ is the diagonal matrices. Let $J:=C_0(M^\circ,M_N(\C))\lhd A$ and $\J:=J\cap\A$. 
The data $(\J\lhd\A,L^2(S)\otimes\C^N,\D_{\text{min}}\otimes 1_N)$ determines a relative spectral triple 
for $J\lhd A:=\{f\in C_0(\ol{M},M_N(\C)):f(x)\in B\,\,\forall x\in\de M\}$.
All subsequent constructions for manifolds with boundary that we consider will automatically
work for these examples. This statement follows since the relative spectral triple is the pullback
of the relative spectral triple 
$$
(C^\infty_0(M^\circ,M_N(\C))\lhd C^\infty_c(\overline{M},M_N(\C)),L^2(S)\otimes\C^N,\D_{\text{min}}\otimes 1_N)
$$
along the obvious 
inclusion $A\hookrightarrow C_0(\ol{M},M_N(\C))$. 
For details on functoriality, see Remark \ref{functorrem}. While the spectrum of $\J$ is Hausdorff, 
the spectrum of $\A$ is not in general.

\subsubsection{Relative spectral triples on conical manifolds}
\label{eg:cones-on-the-brain}

Dirac operators on stratified pseudo-manifolds form a rich source of examples for relative spectral triples. 
A manifold with boundary can be viewed as a stratified pseudo-manifold with a codimension one strata: 
its boundary. We will consider the other extreme case of a stratified pseudo-manifold with stratas of 
maximal codimension, i.e. conical manifolds. The index theory for general stratified pseudo-manifolds 
was studied in \cite{alrp,alrp2}. The analysis on conical manifolds is better understood. The analytic 
results upon which this section rests can be found in Lesch \cite{leschhab}; 
we refer the reader there for precise details. 
Spectral triples on conical manifolds have been considered in the literature, see \cite{lescurethesis}. 
Nevertheless, revisiting conical manifolds within the framework of this paper serve to conceptualise 
the non-commutative geometry of singular manifolds.

Let $M$ be a conical manifold, whose cross section at the conical points is a 
closed $(d-1)$-dimensional manifold $N$. That is, we have a set of conical points 
$\mathfrak{c}=\{c_1,\dots,c_l\}\subseteq M$ and a decomposition $N=N_1\dot{\cup}\cdots \dot{\cup}N_l$ 
such that near any $c_i$ we have ``polar coordinates" $(r,x)$ where $0\leq r<1$, with $r=0$ 
corresponding to $c_i$, and $x$ denotes coordinates on $N$. The fact that $M$ is conical is encoded 
in a non-complete metric $g$ on $M\setminus \mathfrak{c}$. We assume that the metric $g$ is a 
straight-cone-metric, i.e. it will near $c_i$ take the form
$$
g=\mathrm{d} r^2+r^2h_i,
$$
for a metric $h_i$ on $N_i$. It is possible to have a smooth $r$-dependence in the metrics $(h_i)_{i=1}^l$, 
but it complicates the analysis so we assume that the cone is straight for simplicity. 
The manifolds $N_i$ are not necessarily connected. 
We let $M^{\textnormal{reg}}:=M\setminus \mathfrak{c}$.

We can construct a Dirac operator $\Dsla$ on $M$ acting on some Clifford bundle $S\to M$ 
of product type near the conical points. We can describe $\Dsla$ as follows. 
The change of metric from $r^2h_i$ to $h_i$ on 
$N_i\times \{r\}$ induces a bundle automorphism on $S|_{N_i}$ which we denote by $U_i(r)$. 
Under $U_i(r)$, the Dirac operator $\D$ takes the form  
$n(i\partial_r+r^{-1}\Dsla_{N_i})$ near the 
conical point $c_i$. 
Here $\Dsla_{N_i}$ is a bounded perturbation of the Dirac 
operator  on $N_i$ (in the metric $h_i$), and  $n$ is
Clifford multiplication by the normal vector $\overrightarrow{n}_N=\mathrm{d} r$. 
The Dirac operator $\Dsla_{N_i}$ act on $S|_{N_i}$ and anticommutes with 
$n$. 
We often identify $N_i$ with the submanifold 
$N_i\times \{1\}\subseteq N_i\times [0,1]/N_i\times \{0\}\subseteq M$.
The following theorem characterises the domains of Dirac operators on 
conical manifolds; details of the proof can be found in \cite{leschhab}.

\begin{thm}
\label{alltheconeprop}
Let $\D_{\textnormal{min}}$ denote the graph closure of $\Dsla|_{C^\infty(M\setminus \mathfrak{c},S)}$ 
and $\D_{\textnormal{max}}$ the operator $\Dsla$ with the domain
$$
\Dom(\D_{\textnormal{max}}):=\{f\in L^2(M,S): \, \Dsla f\in L^2(M,S)\}.
$$
The following properties hold:

1) The operator $\D_{\textnormal{min}}$ is a closed symmetric operator with 
$\D_{\textnormal{max}}=\D_{\textnormal{min}}^*$;

2) Any closed extension of $\D_{\textnormal{min}}$ contained in $\D_{\textnormal{max}}$ 
has compact resolvent;

3) The vector space $V:=\Dom(\D_{\textnormal{max}})/\Dom(\D_{\textnormal{min}})$ is finite dimensional 
and isomorphic to the subspace 
$$
V\cong W= \bigoplus_{i=1}^lW_i\subseteq 
\bigoplus_{i=1}^l C^\infty(N_i,S|_{N_i}).
$$
The vector subspaces $W_i$ are given by
$$
W_i:= \chi_{\left(-\frac{1}{2},\frac{1}{2}\right)}(\Dsla_{N_i})L^2(N_i,S|_{N_i})=
\bigoplus_{\lambda\in \sigma(\Dsla_{N_i}),|\lambda|<1/2}\bigoplus_{j=1}^{m_i(\lambda)} 
\C f_{i,\lambda,j}\subseteq C^\infty(N_i,S|_{N_i}),
$$
where $m_i$ is the multiplicity function of $\Dsla_{N_i}$ 
and $(f_{i,\lambda,j})_{j=1}^{m_i(\lambda)}$ 
is a basis for the eigenspace $\ker(\Dsla_{N_i}-\lambda)$. 
For suitable choices of bases 
$(f_{i,\lambda,j})_{j=1}^{m_i(\lambda)}$, the quotient mapping 
$\Dom(\D_{\text{min}})\to V$ can 
be split by a mapping $T=\oplus_{i=1}^l T_i:\oplus_{i=1}^lW_i\to \Dom(\D_{\text{max}})$ 
that takes the form
$$
f_{i,\lambda,j}\mapsto r^{-\lambda}\chi_i(r)[U_i(r)(f_{i,\lambda,j})].
$$
Here we extend $f_{i,\lambda,j}$ to a constant function on the cone 
$N_i\times [0,1]/N_i\times \{0\}$ and $\chi_i$ denotes a cutoff function.
\end{thm}
Let us turn to our setup of relative spectral triples. We take $J=C_0(M\take\mathfrak{c})$, $A=J+\sum_{i=1}^\ell\C \chi_j$, $\J=C_c^\infty(M\take\mathfrak{c})$, and $\A=\J+\sum_{i=1}^\ell\C\chi_j$, where $\chi_j$ is a cutoff near the 
conical point $c_j$. It clearly holds that
$$
A/J\cong \C^l\cong C(\mathfrak{c}) \quad\mbox{-- the continuous functions on $\mathfrak{c}$}.
$$
We note that $\A\subset A$ and $\J\subseteq J$ are both dense 
holomorphically closed sub-$\ast$-algebras, and that $\A$ preserves $\Dom(\D_{\text{max}})$.
With this Theorem \ref{alltheconeprop}, it is straightforward to show that 
$(\J\lhd \A,L^2(M,S),\D_{\text{min}})$ is a relative spectral triple for $J\lhd A$. 

There are more relative spectral triples associated to this example, and we will describe
them later when we discuss the Clifford normal and boundary Hilbert space.

For more general stratified manifolds, for which we refer to \cite[Section 2 and 3]{alrp}, 
with an iterated edge metric, it seems likely that 
Dirac operators associated to iterated edge metrics yield relative spectral triples. 
On a stratified pseudo-manifold $X$ with an iterated edge metric $g$ and a Clifford 
bundle $S\to X$, one can construct a Dirac operator $\Dsla_g$ acting on 
$C^\infty_c(X^{\textnormal{reg}},S)$, where $X^{\textnormal{reg}}\subseteq X$ denotes the regular 
part. We let $\D_{\textnormal{min}}$ denote its closure. The candidate for a relative 
spectral triple is $(C^\infty_c(X^{\textnormal{reg}})\lhd C^\infty_c(X), L^2(X,S),\D_{\textnormal{min}})$. 
Dirac operators on such singular manifolds were discussed in 
\cite{alrp,alrp2}. Sufficient conditions for essential self-adjointness were given in \cite{alrp}.

\subsection{The bounded transform of relative unbounded Kasparov modules}
\label{sect:boundedtransform}

The main result of this section is that the 
bounded transform $(X_B,\D(1+\D^*\D)^{-1/2})$ of a $G$-equivariant
relative unbounded Kasparov module $(\J\lhd \A,X_B,\D)$ is a 
$G$-equivariant relative Kasparov module for $(J\lhd A,B)$. 
Hence, a $G$-equivariant relative unbounded Kasparov module defines a class in relative 
$KK$-theory. The proof follows the same ideas as the 
proof that the bounded transform of an unbounded 
Kasparov module is a bounded Kasparov module, \cite{BaajJulg}, though the use of symmetric
operators necessitates additional technicalities. 
A similar method is also used in \cite{Hilsum}*{\S3} to show 
that the bounded transform of a ``half-closed operator'' 
\cite{Hilsum}*{p. 77} defines a bounded Kasparov module.

We begin with some results concerning non-self-adjoint operators.
The notation $\ol{T}$ denotes the operator closure of a 
closeable operator $T$ (i.e. $\ol{T}$ is the operator whose graph is the closure of the graph of $T$).
First we summarise a range of elementary results about 
closed regular operators that are required for the proof, 
and for our subsequent developments of the theory.
\begin{lemma}
\label{lem:domains}
Let $\D:\Dom(\D)\subset X_B\to X_B$ be a closed (unbounded) regular
operator on a Hilbert module $X_B$. Then:

1) $(1+\D^*\D)^{-1/2}:X_B\to \Dom(\D)$ 
is a unitary isomorphism for the graph inner product on $\Dom(\D)$;

2) $(\D(1+\D^*\D)^{-1/2})^*=\overline{(1+\D^*\D)^{-1/2}\D^*}=\D^*(1+\D\D^*)^{-1/2}$;

3) $\|\D(1+\lambda+\D^*\D)^{-1/2}\|_{\End^*_B(X_B)}\leq1$; 

4) $\|(1+\lambda+\D^*\D)^{-1/2}\|_{\End^*_B(X_B)}\leq\frac{1}{\sqrt{1+\lambda}},$
for all $\lambda\in[0,\infty)$.

Moreover, if $\A\subset \End^*_B(X_B)$ is a $\ast$-algebra such that $(1+\D^*\D)^{-1/2}a$ 
is $B$-compact for all $a\in\A$, then $(1+\lambda+\D^*\D)^{-1/2}a$ is $B$-compact for all $a\in\A$.
\end{lemma}

Part 1) is well-known; for a proof see \cite{ForsythThesis}*{Lemma C.2}. 
Part 2) follows from \cite{Lance}*{Theorem 10.4}. 
If $\D$ is regular then so is $\D^*$, \cite{KaadLesch2}*{Lemma 2.2}, 
so parts 1) and 2) of Lemma \ref{lem:domains} also 
hold when $\D$ is replaced by $\D^*$. The proof of 
parts 3) and 4) follows the ideas of \cite[Appendix A]{CP1}, and the 
full argument can be found in \cite{ForsythThesis}*{Lemma 8.7}. The last
statement is a resolvent computation, 
\cite{ForsythThesis}*{Lemma 8.6}.

Next we recall some of the subtleties that arise when we want to
take commutators with a symmetric operator, as opposed to a self-adjoint operator.

\begin{lemma}
\label{lem:commutator}
Let $\D:\Dom(\D)\subset X_B\to X_B$ be an odd closed symmetric regular 
operator on a $\Z/2$-graded Hilbert module $X_B$, and let 
$\A\subset\End^*_B(X_B)$ be a sub-$\ast$-algebra such that for all 
$a\in\A$,  $a\cdot\Dom(\D)\subset\Dom(\D)$ 
and $[\D,a]_\pm$ has an adjointable extension to $X_B$. Then

1) $a\cdot\Dom(\D^*)\subset\Dom(\D^*)$, 
so that $[\D^*,a]_\pm$ is defined on $\Dom(\D^*)$ for all $a\in\A$,

2) $[\D^*,a]_\pm$ is bounded and extends to $\ol{[\D,a]_\pm}$ for all $a\in\A$, and

3) for all $a\in\A$ of homogeneous degree, 
\begin{align*}
[(1+\lambda+\D^*\D)^{-1},a]&=-\D^*(1+\lambda+\D\D^*)^{-1}[\D,a]_\pm(1+\lambda+\D^*\D)^{-1}\\
&\quad-(-1)^{\deg a}(1+\lambda+\D^*\D)^{-1}[\D^*,a]_\pm\D(1+\lambda+\D^*\D)^{-1}.
\end{align*}
\end{lemma}
Parts 1) and 2) of this result can be found in \cite[Lemma 2.1]{Hilsum}. A version of 3) appears
in \cite[Equation 3.2]{Hilsum}, and a detailed proof of this statement can be found in 
\cite{ForsythThesis}*{Lemma 8.5}. The importance of part 3) can be seen from \cite{FMR}
where it is shown that not taking proper care of domains in this 
commutator expression can lead to the bounded transform {\em not} defining
a (relative) Fredholm module. See also \cite[Lemma 2.3]{CP1}.

For elements of the ideal $\J$, part 3) of the above lemma admits a significant refinement.
\begin{lemma}
\label{lem:adjointbound}
Let $A$ and $B$ be $\Z/2$-graded $C^*$-algebras, 
with $A$ separable and $B$ $\sigma$-unital, 
and let $J\lhd A$ be an ideal. Let $(\J\lhd\A,X_B,\D)$ be a 
relative unbounded Kasparov module for $(J\lhd A,B)$, and let 
$ \D_e\subset \D^*$ be a closed, (odd if $(\J\lhd\A,X_B,\D)$ is even),  regular extension of $\D$. Then
\begin{align*}
j\D(1+\lambda+\D^*\D)^{-1}-(1+\lambda+\D_e^*\D_e)^{-1}&j\D
=\D_e^*(1+\lambda+\D_e\D_e^*)^{-1}[\D^*,j]_\pm\D(1+\lambda+\D^*\D)^{-1}\\
&+(-1)^{\deg j}(1+\lambda+\D_e^*\D_e)^{-1}[\D,j]_\pm\D^*\D(1+\lambda+\D^*\D)^{-1}
\end{align*}
for all  $j\in\J$ of homogeneous degree and $\lambda\in[0,\infty)$, where both sides of the equation are defined on $\Dom(\D)$,
and hence
$$
\|\ol{j\D(1+\lambda+\D^*\D)^{-1}-(1+\lambda+\D_e^*\D_e)^{-1}j\D}\|_{\End_B^*(X_B)}\leq\frac{2\|\ol{[\D,j]_\pm}\|}{1+\lambda}.
$$
\end{lemma}
This lemma can be used when proving that
the bounded transform 
of closed extensions $\D_e$ yield
Kasparov modules for $J$ (see Theorem \ref{thm:extensions}). 
Since its proof is rather technical, it can be found in Appendix \ref{sec:technicalproof}.

The following result is the main theorem of this section. Again we defer
the proof to the Appendix. Conceptually it follows the same scheme as \cite{BaajJulg}, but
the use of symmetric operators complicates the details significantly.

\begin{thm}
\label{thm:alternate}
Let  $G$ be a compact group, $A$ and $B$ be $\Z/2$-graded $G-C^*$-algebras, 
with $A$ separable and $B$ $\sigma$-unital 
and let $J\lhd A$ be a $G$-invariant graded ideal. Let $(\J\lhd\A,X_B,\D)$ be a 
$G$-equivariant relative unbounded Kasparov module for $(J\lhd A,B)$, and let 
$F=\D(1+\D^*\D)^{-1/2}$ be the bounded transform 
of $\D$. Then $(X_B,F)$ is a $G$-equivariant relative Kasparov 
module for $(J\lhd A,B)$. The same holds true when replacing Condition 4. in 
Definition \ref{defn:alternate} by Condition 4'. in Remark \ref{rem:alternatecondition}.
\end{thm}

That there are two  definitions of relative Kasparov module
that yield this important result is a positive for the flexibility of (some of) the theory.
It is not clear, however, in what way the two different hypotheses 4. and 4'. are related
in general.

\begin{rem}
Suppose that $(\J\lhd\A,X_B,\D)$ is a $G$-equivariant relative unbounded Kasparov module for a 
$G$-invariant graded ideal 
$J$ in a separable, $\Z/2$-graded $G$-$C^*$-algebra 
$A$. Then the bounded transform $F_{\D^*}=\D^*(1+\D\D^*)^{-1/2}$ 
of $\D^*$ also defines a $G$-equivariant relative Kasparov module with the 
same class as $F_\D=\D(1+\D^*\D)^{-1/2}$ in relative 
$KK$-theory, even though $(\J\lhd \A,X_B,\D^*)$ is not 
 a relative unbounded Kasparov module (unless $\D$ is self-adjoint). 
 This is because the path $[0,1]\ni t\mapsto tF_\D+(1-t)F_{\D^*}$ is an operator homotopy 
of $G$-equivariant relative Kasparov modules, 
using the fact that $F_\D^*=F_{\D^*}$, \cite{Lance}*{Theorem 10.4}.
\end{rem}

The next result shows that if $(\J\lhd\A,X_B,\D)$ is a 
$G$-equivariant relative unbounded Kasparov module for an ideal $J$ in a separable 
$\Z/2$-graded $G$-$C^*$-algebra $A$, then the relative 
$KK$-theory class can also be represented by the 
phase of $\D$ (whenever it is well-defined), and hence by a partial isometry.

\begin{prop}
\label{prop:phasetriple}
Let $(\J\lhd\A,X_B,\D)$ be a $G$-equivariant relative unbounded Kasparov module for a 
$G$-equivariant graded ideal 
$J$ in a separable $\Z/2$-graded $G$-$C^*$-algebra $A$. Suppose that $\ker(\D)$ is a complemented submodule of $X_B$. 
Let $V$ be the phase of $\D$, which is the 
partial isometry with initial space $\ker(\D)^\perp$ 
and final space $\ker(\D^*)^\perp$ defined by 
$\D=V|\D|$, \cite{ReedSimon1}*{Theorem VIII.32}. 
Then $(X_B,V)$ is a $G$-equivariant relative (bounded) Kasparov module 
with the same class as the bounded transform 
$(X_B,\D(1+\D^*\D)^{-1/2})$.
\end{prop}

\begin{proof}
We note that since $\ker(\D)$ and 
and $\ker(\D^*)$ are complemented, the phase $V$ is well-defined.
We claim that $a(V-\D(1+\D^*\D)^{-1/2})$ is $B$-compact for all $a\in A$, 
from which it follows that $[0,1]\ni t\mapsto tV+(1-t)\D(1+\D^*\D)^{-1/2}$ 
is an operator homotopy of $G$-equivariant 
relative Kasparov modules. 
Since
$a(1-P_{\ker(\D^*)})(1+\D\D^*)^{-1/2}$ is $B$-compact for all $a\in\A$, then 
\begin{align*}
a(\D(1+\D^*\D)^{-1/2}&-V)=aV\left((\D^*\D)^{1/2}(1+\D^*\D)^{-1/2}-1\right)\\
&=a(1-P_{\ker(\D^*)})(1+\D\D^*)^{-1/2}(1+\D\D^*)^{1/2}\left((\D\D^*)^{1/2}(1+\D\D^*)^{-1/2}-1\right)
\end{align*}
is $B$-compact for all $a\in\A$, proving the claim.
\end{proof}

The following result is a specialisation of \cite{Hilsum}*{Theorem 3.2}. 
It can be proved by  using Lemma \ref{lem:adjointbound} and the 
integral formula for fractional powers (see Equation \eqref{eq:iffp} in 
Appendix \ref{sec:technicalproof}) to show that 
$j(F_{\D}-F_{\D_e})$ is compact for all $j\in\J$.

\begin{thm}
\label{thm:extensions}
Let $(\J\lhd\A,X_B,\D)$ be a $G$-equivariant 
relative unbounded Kasparov module for a 
$G$-invariant graded ideal $J$ in a 
separable $\Z/2$-graded $G$-$C^*$-algebra $A$, and let 
$\D\subset \D_e\subset \D^*$ be a closed regular 
$G$-equivariant extension of $\D$. Then:

1) The data $(X_B,F_{\D_e}=\D_e(1+\D_e^*\D_e)^{-1/2})$ 
defines a $G$-equivariant
Kasparov module for $(J,B)$, and so $\D_e$ defines 
a class $[\D_e]\in KK_G^*(J,B)$; and

2) $[\D_e]=[\D]\in KK^*_G(J,B)$; i.e. the $KK$-theory 
class is independent of the choice of extension.
\end{thm}

\begin{rem}
\label{remarkonextensions}
It should be emphasised that $F_{\D_e}$ does not 
generally define a Kasparov module for $A$, 
even if $\D_e$ is self-adjoint with compact resolvent. 
The case $B=\C$ was studied in \cite{FMR}.
If $\D_e$ is self-adjoint with compact resolvent, 
the triple $(\A,\h,\D_e)$ appears to satisfy the 
conditions of a spectral triple since $[\D_e,a]_\pm$ 
is well-defined and bounded on $\Dom(\D)$, which is 
dense in $\h$. However, if $[\D_e,a]_\pm$ is not 
well-defined on $\Dom(\D_e)$ the bounded transform 
need not define a Fredholm module. 

If $\D$ does admit a $G$-equivariant self-adjoint regular extension $\D_e$ 
such that $a\cdot \Dom(\D_e)\subset\Dom(\D_e)$ 
and $a(1+\D_e^2)^{-1/2}$ is compact for all $a\in\A$, 
then $(\A,X_B,\D_e)$ is a $G$-equivariant unbounded Kasparov module for $A$ and 
hence defines a class in $KK^*_G(A,B)$. It follows from 
the exactness of the six-term exact sequence in 
$KK$-theory that $\de([\D])=0\in KK^{*+1}_G(A/J,B)$ 
(for $A$ trivially $\Z/2$-graded and $A\to A/J$ semisplit), 
since $[\D]=\iota^*[(\A,X_B,\D_e)]\in KK^*_G(J,B)$, where 
$\iota:J\rightarrow A$ is the inclusion map. So the 
non-vanishing of $\de([\D])$ is an obstruction to the 
existence of such extensions. For a Dirac-type operator 
on a compact manifold with boundary,  this obstruction is expressed in \cite{BDT}*{Corollary 4.2}.
\end{rem}

\subsection{The $K$-homology boundary map for relative spectral triples}
\label{subsec:bdry-map}

We will now turn to study the image of the class of a relative unbounded Kasparov 
module under the boundary mapping $\partial:KK^0_G(J\triangleleft A,B)\to KK^{1}_G(A/J,B)$. 
The computation uses the isomorphism $KK^{1}_G(A/J,B)\cong \Ext_G^{-1}(A/J,B)$, 
see \cite{thomsenequi} and further discussion in Appendix \ref{subsec:Khom}.

Suppose that $A$ and $B$ are trivially $\Z/2$-graded $G$-$C^*$-algebras, with $A$ separable and $B$ $\sigma$-unital, 
 and that $J\lhd A$ is a $G$-invariant ideal such that $A\to A/J$ is semisplit. 
Let $(\J\lhd\A,X_B,\D)$ be a $G$-equivariant even relative unbounded Kasparov module for $(J\lhd A,B)$. 
With respect to the $\Z/2$-grading $X_B=X_B^+\oplus X_B^-$, we write
$$
\D=\begin{pmatrix}0&\D^-\\\D^+&0\end{pmatrix}
\quad\mbox{and}\quad 
F=\begin{pmatrix}0&F^-\\F^+&0\end{pmatrix}=\D(1+\D^*\D)^{-1/2}.
$$
The operator $F$ is the bounded transform of $\D$. The hypotheses of Proposition \ref{prop:extbdrymap} are satisfied for $(X_B,F)$, and so the boundary class in 
$KK^1_G(A/J,B)$ of $[\D]\in KK^0_B(J\lhd A,B)$ has a simple description in terms of extensions. 

\begin{prop}
\label{representingyo}
Let $(\J\lhd\A,X_B,\D)$ be a $G$-equivariant even relative unbounded Kasparov module 
for $(J\lhd A,B)$ such that  $A\to A/J$ is semisplit and assume that $\ker(\D)$ is complemented. Then 
$$
\partial [(\J\lhd \A,X_B,\D)]=[\alpha] \in KK^1_G(A/J,B),
$$
where the invertible extension $\alpha$ of $A/J$ is defined by the Busby invariant
\begin{align*}
&\alpha:A/J\rightarrow \mathcal{Q}_B(\ker (\D^*)^+),\quad \alpha(a)
:=\pi(P_{\ker((\D^*)^+)}\wt{a}P_{\ker((\D^*)^+)}),
\end{align*}
where $\wt{a}\in A$ is any preimage of $a\in A/J$.
\end{prop}

\begin{proof}
By a slight  abuse of notation, given a complemented submodule $W\subset X_B$ such that 
$P_WjP_W$ and $[P_W,b]$ are compact for all $j\in J$ and $b\in A$, 
we also denote by $W$ the extension $A/J\rightarrow \mathcal{Q}_B(W)$ given by 
$a\mapsto\pi(P_W\wt{a}P_W)$, where $\wt{a}\in A$ is any lift of 
$a\in A/J$ and $\pi:\End^*_B(W)\rightarrow \mathcal{Q}_B(W)$ is the quotient map.

Since $(X_B,\D^*(1+\D\D^*)^{-1/2})$ is a relative 
Kasparov module for $(J\lhd A,B)$ with the same class as $(X_B,\D(1+\D^*\D)^{-1/2})$ 
(an operator homotopy is $[0,1]\ni t\mapsto t\D(1+\D^*\D)^{-1/2}+(1-t)\D^*(1+\D\D^*)^{-1/2}$), 
we can use Proposition \ref{prop:extbdrymap} to express the boundary map as
\begin{equation*}
\de[(\J\lhd\A,\h,\D)]=[\ker((\D^*)^+)]-[\ker(\D^-)]\in \Ext_G^{-1}(A/J,B)\cong  KK^1_G(A/J,B).
\end{equation*}
Since $a(1+\D^*\D)^{-1/2}$ is compact for all
$a\in A$, $aP_{\ker(\D)}$ is compact for all $a\in A$. Therefore $\ker(\D^-)$ is a 
trivial extension, and so $\de[(\J\lhd\A,\h,\D)]=[\ker((\D^*)^+)]$.
\end{proof}

\section{The Clifford normal and the double of a relative spectral triple}
\label{sect:normal}

In this section we will discuss further geometric constructions for a relative spectral triple 
$(\J\lhd\A,\h,\D)$ for $J\lhd A$. To simplify the discussion, we restrict to $B=\C$ and 
assume that $G$ is the trivial group.

We will show that an auxiliary operator we call a \emph{Clifford normal}, denoted by $n$, 
can be used to encode the additional information needed for geometric constructions.
Motivated by the doubling construction on a 
manifold with boundary, \cite{BBW}*{Ch. 9}, 
we use the Clifford normal to construct a 
spectral triple for the pullback algebra 
$\wt{A}=\{(a,b)\in A\oplus A:a-b\in J\}$. 

The Clifford normal $n$ can also be used to 
construct a ``boundary'' Hilbert space $\mathfrak{H}_n$. 
The Hilbert space $\mathfrak{H}_n$ carries a densely defined action of 
$\A/\J\wh{\otimes}\Cl_1$ which extends to an action of $A/J\wh{\otimes}\Cl_1$ under 
additional assumptions. Under additional assumptions on $\D$ and the Clifford normal $n$
we can construct a symmetric operator $\D_n$ on $\mathfrak{H}_n$. 
In Section \ref{sect:tripleonbdry}, we will consider what happens when 
$(\A/\J\wh{\otimes}\Cl_1,\mathfrak{H}_n,\D_n)$ defines a spectral triple.

\subsection{The Clifford normal and the boundary Hilbert space}
\label{subsectionnormal}
 
The motivation for the Clifford normal comes from the classical example of a manifold with boundary. 
Let $\Dsla$ be a Dirac operator on a Clifford module $S$ over a compact Riemannian manifold with 
boundary $\ol{M}$. We emphasize that $\Dsla$ is the differential expression defining the Dirac operator, 
and not an unbounded operator with a prescribed domain. Let $(\cdot|\cdot)$ denote the pointwise 
inner product on $S$. For sections $\xi,\eta\in C^\infty(\ol{M},S)$, 
we have Green's formula \cite{BBW}*{Proposition 3.4}
\begin{align}
\label{eq:greens}
\Ideal{\Dsla\xi,\eta}_{L^2(S)}-\Ideal{\xi,\Dsla\eta}_{L^2(S)}
=\int_{\de M}(\xi|n\eta)\,\vol_{\de M}
&=\Ideal{\xi,n\eta}_{L^2(S|_{\de M})},
\end{align}
where $n$ denotes Clifford multiplication by the inward unit normal.  
If, abusing notation, $n$ is also Clifford multiplication by some smooth extension of 
the inward unit normal to the whole manifold,
then the boundary inner product  can be expressed as
$$
\Ideal{\xi,\eta}_{L^2(S|_{\de M})}
=\Ideal{\xi,\Dsla n\eta}_{L^2(S)}-\Ideal{\Dsla\xi,n\eta}_{L^2(S)}.
$$
The operator $n$ is the model for the Clifford normal.

\begin{defn}
\label{defn:normal}
Let $A$ be a separable $\Z/2$-graded $C^*$-algebra and $J\lhd A$ a graded ideal. 
We assume that $(\J\lhd\A,\h,\D)$ 
is a relative spectral triple for $J\lhd A$. A \textbf{Clifford normal} for $(\J\lhd\A,\h,\D)$ is an odd 
(in the case that $(\J\lhd\A,\h,\D)$ is even) operator $n\in \B(\h)$ such that:

1) $n\cdot\Dom(\D)\subset\Dom(\D)$ and $\Dom(n):=\Dom(\D^*)\cap n\Dom(\D^*)$ is a core for $\D^*$;

2)
$n^*=-n$;

3)
$[\D^*,n]$ is a densely defined symmetric operator  on $\h$;

4) $[n,a]_\pm\cdot\Dom(n)\subset\Dom(\D)$ for all $a\in\A$;

5) $(n^2+1)\cdot\Dom(n)\subset\Dom(\D)$;

6)
$\Ideal{\xi,\D^*n\xi}-\Ideal{\D^*\xi,n\xi}\geq0$ for all $\xi\in\Dom(n)$;

7) For $w,\,z\in\Dom(\D^*)$, if  
$\Ideal{w,\D^*n\xi}-\Ideal{\D^*w,n\xi}=-\Ideal{z,\D^*\xi}+\Ideal{\D^*z,\xi}$ for all $\xi\in\Dom(n)$ then
$w+nz\in \Dom(\D)$.

If $n$ is a Clifford normal for $(\J\lhd\A,\h,\D)$, we say that $(\J\lhd\A,\h,\D,n)$ 
is a relative spectral triple with Clifford normal.
\end{defn}

Condition 6) in Definition \ref{defn:normal} will be 
necessary for our purposes. Certainly Condition 6) is 
something we would prefer to prove from more 
conceptually elementary assumptions, but it is unclear 
whether this is possible. 

The opaque non-degeneracy 
assumption in Condition 7) of Definition \ref{defn:normal} 
will be necessary for self-adjointness in the 
construction of the ``double" (see Subsection \ref{sec:double}) 
as well as for a non-degeneracy condition of the 
quadratic form in Condition 6). An equivalent form of the 
non-degeneracy condition 7) is given in 
Remark \ref{remark59} (see page \pageref{remark59}).

\begin{rem}
Condition 2) of Definition \ref{defn:normal} can be weakened to $(n+n^*)\cdot\Dom(n)\subset\Dom(\D)$. 
In the case that $(\J\lhd\A,\h,\D)$ is even, the condition that $n$ is odd can be weakened 
to $n\gamma+\gamma n$  extending by continuity in the graph norm to an operator on 
$\Dom(\D^*)$ such that $(n\gamma+\gamma n)\cdot\Dom(\D^*)\subset\Dom(\D)$. 
Here $\gamma$ is the grading operator on $\h$.
In practice we do not need this level of generality.
\end{rem}

\begin{rem}
The space $\Dom(n)\subseteq \Dom(\D^*)$ is the domain of $n$ as a densely defined operator on $\Dom(\D^*)$.
Note that conditions 1) and 5) together imply that $n$ preserves $\Dom(n)$, so that $\Dom(n^2)=\Dom(n)$, 
viewing $n$ and $n^2$ as densely defined operators on $\Dom(\D^*)$.
\end{rem}

To put the conditions 6) and 7) of Definition \ref{defn:normal} in context, we recall the following well-known fact for 
symmetric operators. We say that a sesquilinear form $\omega$ is anti-Hermitian if
$\omega(\xi,\eta)=-\overline{\omega(\eta,\xi)}$ for all $\xi$ and $\eta$ in its domain. 

\begin{lemma}
\label{nondegeprop}
Let $\D$ be a closed symmetric operator. The anti-Hermitian form 
\begin{align*}
\omega_{\D^*}:\Dom(\D^*)/ &\Dom(\D)\times\Dom(\D^*)/\Dom(\D)\to \C\\
& \omega_{\D^*}([\xi],[\eta]):=\Ideal{\xi,\D^*\eta}-\Ideal{\D^*\xi,\eta},
\end{align*}
is well-defined and non-degenerate,
where for $\xi\in\Dom(\D^*)$, $[\xi]$ denotes the class in $\Dom(\D^*)/\!\Dom(\D)$.
\end{lemma}

\begin{proof}
We first establish that $\omega_{\D^*}$ is well-defined. 
If $\xi\in\Dom(\D)$ and $\eta\in\Dom(\D^*)$, then
\begin{align*}
\Ideal{\xi,\D^*\eta}-\Ideal{\D^*\xi,\eta}
&=\Ideal{\xi,\D^*\eta}-\Ideal{\D\xi,\eta}=\Ideal{\xi,\D^*\eta}-\Ideal{\xi,\D^*\eta}=0,
\end{align*}
On the other hand, if $\eta\in\Dom(\D)$ and $\xi\in\Dom(\D^*)$, then
\begin{align*}
\Ideal{\xi,\D^*\eta}-\Ideal{\D^*\xi,\eta}
&=\Ideal{\xi,\D \eta}-\Ideal{\D^*\xi,\eta}=\Ideal{\D^*\xi,\eta}-\Ideal{\D^*\xi,\eta}=0.
\end{align*}
These calculations show that $\omega_{\D^*}([\xi],[\eta])$ 
does not depend on the choice of representatives 
$\xi,\eta\in\Dom(\D^*)$ of $[\xi],[\eta]\in\Dom(\D^*)/\Dom(\D)$, 
and hence that $\omega_{\D^*}$ is well-defined.

It is clear that $\omega_{\D^*}$ is anti-Hermitian. 
To show that $\omega_{\D^*}$ is non-degenerate, it therefore 
suffices to prove that $\omega_{\D^*}([\xi],[\eta])=0$ for all $\xi\in \Dom(\D^*)$
implies that $\eta\in \Dom(\D)$. If $\omega_{\D^*}([\xi],[\eta])=0$ for all $\xi\in \Dom(\D^*)$, 
it follows from the definition of $\omega_{\D^*}$
that $\Ideal{\D^*\xi,\eta}=\Ideal{\xi,\D^*\eta}$ for any $\xi\in \Dom(\D^*)$. 
Hence $\xi\in \Dom(\D^{**})=\Dom(\D)$ and the lemma follows.
\end{proof}

\begin{defn}
We say that two Clifford normals $n$ and $n'$ for a relative spectral triple 
$(\J\lhd \A,\h,\D)$ are equivalent if $n-n'$ extends by continuity in the graph norm to a continuous operator 
$n-n':\Dom(\D^*)\to \Dom(\D)$. An equivalence class of Clifford normals for $(\J\lhd \A,\h,\D)$ is 
called a normal structure.
\end{defn}

\begin{defn}
\label{def:haitch}
We introduce the notations (to be justified by the classical example below)
\begin{align*}
\mathfrak{H}^{1/2}_n&:=\Dom(n)/\Dom(\D) 
\quad\mbox{and} \quad
\check{\mathfrak{H}}:=\Dom(\D^*)/\Dom(\D).
\end{align*}
Given $\xi\in\Dom(\D^*)$, we let $[\xi]$ denote the class in $\check{\mathfrak{H}}$. 
Similarly, given $\xi\in\Dom(n)$, $[\xi]$ denotes the class in $\mathfrak{H}_n^{1/2}$. 
We topologise the spaces $\Dom(n)$, $\mathfrak{H}^{1/2}_n$ and 
$\check{\mathfrak{H}}$ as Hilbert spaces using the respective graph inner products.
\end{defn}

\begin{notn}
We reserve the font $\h$ to refer to Hilbert spaces classically associated to the total space and 
the font $\mathfrak{H}$ for Hilbert spaces classically associated to the boundary.
\end{notn}

We remark that the space $\check{\mathfrak{H}}$ 
is independent of the choice of Clifford normal. 
The spaces $\Dom(n)$ and $\mathfrak{H}^{1/2}_n$ 
only depend on the normal structure, i.e. 
if $n\sim n'$ then the identity map defines 
continuous isomorphisms $\Dom(n)\cong\Dom(n')$ and 
$\mathfrak{H}^{1/2}_n\cong \mathfrak{H}^{1/2}_{n'}$.

\begin{eg}
\label{classicalexampleonmfd}
The reader is encouraged to revisit the discussion for manifolds with boundary in the introduction.
Let $\D_{\textnormal{min}}$ be the minimal closed extension of a Dirac operator $\Dsla$ on a Clifford module $S$ 
over a compact Riemannian manifold with boundary $\ol{M}$. Then as in Subsubsection \ref{eg:diracboundary}, 
$(C^\infty_0(M^\circ)\lhd C^\infty(\ol{M}),L^2(M,S),\D_{\textnormal{min}})$ is a relative spectral triple for $C_0( M^\circ)\lhd C(\ol{M})$. 
We can extend the inward unit normal on the boundary to a unitary endomorphism defined 
on a collar neighbourhood of the boundary (for instance using parallel transport).
By multiplying by a cut-off function only depending on the normal coordinate 
we can define an anti-self-adjoint endomorphism $n$ over 
the manifold $\ol{M}$. The normal structure is independent of choice of extension of the 
normal vector to the interior. The operator $n$ immediately satisfies all the conditions of 
Definition \ref{defn:normal}, except perhaps Conditions 1) and 3), which we now verify.

By \cite{BaerBall}*{Theorem 6.7}, $\Dom(n)=H^1(\ol{M},S)$ is a core for $\D_{\textnormal{max}}=\D^*_{\textnormal{min}}$.
Condition 1) is satisfied because $nH^1_0(M^\circ,S)\subseteq H^1_0(M^\circ,S)$ and $nH^1(\ol{M},S)\subseteq H^1(\ol{M},S)$.

To address Condition 3), we examine the behaviour of the 
Dirac operator near the boundary. In a collar neighbourhood of the boundary, $\Dsla$ has the form
$$
\Dsla=n\left(\frac{\de}{\de u}+B_u\right)
$$
where $u$ is the inward normal coordinate and $B_u$ is a family of 
Dirac operators over the boundary, \cite{BBW}*{p. 50}. Near the boundary,
\begin{equation}
\label{commwithdira}
[\Dsla,n]=n\left(\frac{\de}{\de u}+B_u\right)n-n^2\left(\frac{\de}{\de u}+B_u\right)
=n\frac{\de n}{\de u}+nB_un+B_u.
\end{equation}
The second and third terms are symmetric, and since $n$ commutes with 
$\frac{\partial n}{\partial u}$, it is straightforward to check that $n\frac{\partial n}{\partial u}$
is self-adjoint. Thus, $[\D^*_{\textnormal{min}},n]$ is a perturbation of a symmetric operator 
by a bounded self-adjoint operator, which is then symmetric.

If $M$ is merely an open submanifold of a complete manifold, 
then we still obtain a relative spectral triple, as in 
Subsubsection \ref{eg:diracboundary}. However, in this case 
$\ol{M}$ need not be a manifold with boundary and 
there need not be a Clifford normal. So the Clifford normal 
$n$ is additional structure that is imposed on the 
geometry in order to obtain a reasonable boundary.
\end{eg}

\begin{rem}
For manifolds with boundary, $\Dom(n)=H^1(\ol{M},S)$ 
by Equation \eqref{ndomdsforman} 
(on page \pageref{ndomdsforman}).
Using Equations  \eqref{staronpage6} 
and \eqref{ndomdsforman} on page \pageref{staronpage6},
it can be checked that $\mathfrak{H}^{1/2}_n=H^{1/2}(\partial M,S|_{\partial M})$ and 
$\check{\mathfrak{H}}=\check{H}(\D_{\partial M})$. In this case 
$$
\omega_{\D^*}(\xi,\eta)=\int_{\de M}(\xi|n\eta)\,\vol_{\de M}.
$$
The anti-Hermitian form $\omega_{\D^*}$ is a 
well-defined pairing on $\check{H}(\D_{\partial M})$ 
because $n$ defines a unitary isomorphism 
$$
\check{H}(\D_{\partial M})\to \hat{H}(\D_{\partial M})\cong \check{H}(\D_{\partial M})^*.
$$
The last identification is via the $L^2$-pairing on $\partial M$.
\end{rem}

\begin{lemma}
\label{innerproductlemma}
Let $n$ be a Clifford normal for the relative spectral triple $(\J\lhd \A,\h,\D)$. 
The form $\langle [\xi],[\eta]\rangle_n :=\omega_{\D^*}([\xi],[n\eta])$ defines a 
Hermitian inner product on $\mathfrak{H}^{1/2}_n$ only 
depending on the normal structure that $n$ defines.
\end{lemma}

\begin{proof}
To show that the form is Hermitian, we compute
\begin{align*}
&\ol{\Ideal{[\xi],[\eta]}_n}=\ol{\Ideal{\xi,\D^*n\eta}}-\ol{\Ideal{\D^*\xi,n\eta}}
=\Ideal{\D^*n\eta,\xi}-\Ideal{n\eta,\D^*\xi}\\
&=\Ideal{n\D^*\eta,\xi}+\Ideal{[\D^*,n]\eta,\xi}+\Ideal{\eta,n\D^*\xi}
=-\Ideal{\D^*\eta,n\xi}+\Ideal{\eta,[\D^*,n]\xi}+\Ideal{\eta,n\D^*\xi}\\
&=\Ideal{\eta,\D^*n\xi}-\Ideal{\D^*\eta,n\xi}=\Ideal{[\eta],[\xi]}_n.
\end{align*}
By Lemma \ref{nondegeprop} and density of $\Dom(n)\subseteq \Dom(\D^*)$ 
in the graph norm, $\Ideal{\cdot,\cdot}_n$ is non-degenerate. 
Condition 6) of Definition \ref{defn:normal} ensures that $\Ideal{\cdot,\cdot}_n$ 
is positive-definite, since it is non-degenerate.
\end{proof}

\begin{defn}
\label{defn:bdry-Hilbert}
The completion of $\mathfrak{H}^{1/2}_n$ with 
respect to the norm coming from $\Ideal{\cdot,\cdot}_n$ 
is a Hilbert space, which we call the \textbf{boundary Hilbert space} and denote by $\mathfrak{H}_n$.
\end{defn}

\begin{defn}
For a relative spectral triple $(\J\lhd\A,\h,\D,n)$ with Clifford normal, we define an operator 
$n_\de:\mathfrak{H}^{1/2}_n\rightarrow\mathfrak{H}^{1/2}_n$ by $n_\de[\xi]:=[n\xi]$.
\end{defn}

\begin{lemma}\label{lem:n1}
Let $n$ be a Clifford normal for the relative spectral triple $(\J\lhd \A,\h,\D)$. 
The operator $n_\de$ extends to a bounded operator on $\mathfrak{H}_n$ that only depends on 
the normal structure that $n$ defines. The operator $n_\de$ satisfies the properties $n_\de^2=-1$, 
$$\Ideal{n_\de[\xi],n_\de[\eta]}_n=\Ideal{[\xi],[\eta]}_n,\quad\mbox{for all $[\xi],[\eta]\in\mathfrak{H}_n$},$$
and $n_\de$ restricts to a continuous operator on $\mathfrak{H}^{1/2}_n$.
\end{lemma}

\begin{proof}
The first claim follows from $(n^2+1)\cdot\Dom(n)\subset\Dom(\D)$. For the second claim, we have
\begin{align*}
&\Ideal{n_\de[\xi],n_\de[\eta]}_n=\Ideal{n\xi,\D^*n^2\eta}-\Ideal{\D^*n\xi,n^2\eta}\\
&=-\Ideal{n\xi,\D^*\eta}+\Ideal{\D^*n\xi,\eta}+\Ideal{n\xi,\D(n^2+1)\xi}-\Ideal{\D^*n\xi,(n^2+1)\xi}\\
&=\Ideal{\xi,n\D^*\eta}-\Ideal{\D^*\xi,n\eta}+\Ideal{[\D^*,n]\xi,\eta}\\
&=\Ideal{\xi,\D^*n\eta}-\Ideal{\xi,[\D^*,n]\eta}-\Ideal{\D^*\xi,n\eta}+\Ideal{[\D^*,n]\xi,\eta}
=\Ideal{[\xi],[\eta]}_n.&\qedhere
\end{align*}
\end{proof}

\begin{prop}
\label{bdrysympform}
Suppose that $(\J\lhd \A,\h,\D,n)$ is a relative spectral triple with Clifford normal. 
The densely defined operator $n:\Dom(n)\subset\Dom(\D^*)\rightarrow\Dom(\D^*)$ is a 
closed operator (for the graph norm on $\Dom(\D^*)$). 
Moreover, the anti-Hermitian form $\omega_{\D^*}$ is tamed by the complex structure 
$n$ in the sense that 
$$
\langle \cdot,\cdot \rangle_n:
\check{\mathfrak{H}}\times \mathfrak{H}^{1/2}_n\to \C, \quad 
\langle [\xi],[\eta]\rangle_n :=\omega_{\D^*}([\xi],n_\de [\eta]),
$$
is a non-degenerate pairing. Moreover, this pairing only depends on the 
normal structure that $n$ defines.
\end{prop}

\begin{proof}
Let us first prove that $n$ is closed in its graph norm on $\Dom(\D^*)$. 
Assume that $(\xi_j)_j\subseteq \Dom(n)$ is a Cauchy sequence in the graph norm of $\Dom(n)$. 
That is, the sequences $(\xi_j)_j$, $(\D^*\xi_j)_j$ and $(\D^*n\xi_j)_j$ 
are Cauchy sequences in $\h$. We write $\xi:=\lim \xi_j$. Since $\D^*$ is closed, 
$\D^*\xi_j\to \D^*\xi$ and $\xi\in \Dom(\D^*)$. 
Moreover, $n\xi_j$ converges to $n\xi$, because $n$ is continuous 
in the Hilbert space $\h$. If $n\xi_j\to n\xi$ and $\D^*n\xi_j$ converges, the closedness of $\D^*$ 
implies $\D^*n\xi_j\to \D^*n\xi$ and $n\xi\in \Dom(\D^*)$. Therefore $\xi\in\Dom(n)$.

We prove non-degeneracy of $\langle \cdot,\cdot \rangle_n$ one variable at a time. 
If $\langle \xi,\eta\rangle_n =0$ for all $\eta\in  \mathfrak{H}^{1/2}_n$, 
Condition 7) implies that $\xi=0$ in $\check{\mathfrak{H}}$. 
The non-degeneracy in the second variable follows from the non-degeneracy of the Hermitian 
form $\Ideal{\cdot,\cdot}_n$ on $\mathfrak{H}^{1/2}_n$ and the natural inclusion 
$\mathfrak{H}^{1/2}_n\subset\check{\mathfrak{H}}$.
\end{proof}

\begin{question}
\label{q:abnormal-normal?}
It remains open if a relative spectral triple $(\J\lhd\A,\h,\D)$ for $J\lhd A$ admits several Clifford normals 
defining ``boundary'' Hilbert spaces such that the identity mapping on $\mathfrak{H}^{1/2}_n$ 
does not extend to an isomorphism on the different completions. Such Clifford normals must clearly 
define different normal structures.
\end{question}

\subsection{The doubled spectral triple}
\label{sec:double}

We continue our study of relative spectral triples with 
Clifford normals for a graded ideal $J\lhd A$. 
We now additionally require that $(\J\lhd\A,\h,\D,n)$ 
is \emph{even}. The odd case can be treated by associating to the odd relative spectral triple 
an even relative spectral triple for $J\wh{\otimes}\Cl_1\lhd A\wh{\otimes}\Cl_1$ analogously to 
\cite{Connes}*{Proposition IV.A.14}, where $\Cl_1$ is the complex Clifford algebra with one generator.
In this section, we will use the Clifford normal $n$ 
to construct the ``doubled'' spectral triple 
$(\wt{\A},\wt{\h},\wt{\D})$, which is a spectral triple 
for the $C^*$-algebra constructed as the restricted 
direct sum
$$
\wt{A}:=A\oplus_JA=\{(a,b)\in A\oplus A:a-b\in J\}.
$$
For a manifold with boundary $\ol{M}$, $C(\ol{M})\oplus_{C_0(M)}C(\ol{M})=C(2M)$ and 
the construction in this subsection mimics the doubling construction of a Dirac operator on 
$\ol{M}$, \cite{BBW}*{Ch. 9}. 

Let $\wt{\h}=\h\oplus\h$. We equip $\wt{\h}$ with the $\mathbb{Z}/2$-grading 
$\wt{\h}^\pm=\h^\pm\oplus\h^{\mp}$. The pullback algebra $\wt{A}$ is represented 
on $\wt{\h}$ by $(a,b)\cdot(\xi,\eta)=(a\xi,b\eta)$.

\begin{defn}
\label{defofdtwiddle}
Let $(\J\lhd \A,\h,\D,n)$ be a relative spectral triple 
with Clifford normal and denote the grading operator on $\h$ by $\gamma$. 
Define an operator $\wt{\D}$ on $\wt{\h}$ on the domain
$$
\Dom(\wt{\D})=\{(\xi,\eta)\in\Dom(n)\oplus\Dom(n):\eta-n\gamma\xi\in\Dom(\D)\}
$$
by $\wt{\D}(\xi,\eta)=(\D^*\xi,\D^*\eta)$.
\end{defn}

\begin{rem}
The operator $\wt{\D}$ only depends on the normal structure that $n$ defines.
\end{rem}

\begin{prop}
\label{prop:doubleadjoint}
The operator $\wt{\D}$ is self-adjoint on $\wt{\h}$.
\end{prop}

\begin{proof}
We first show that $\wt{\D}$ is symmetric, and then show that $\Dom(\wt{\D}^*)\subset\Dom(\wt{\D})$ 
which establishes that $\wt{\D}$ is self-adjoint. Let $(\xi,n\gamma\xi+\varphi),(\xi',n\gamma\xi'+\varphi')\in\Dom(\wt{\D})$, 
where $\xi,\xi'\in\Dom(n)$ and $\varphi,\varphi'\in\Dom(\D)$. Then
\begin{align*}
&\Ideal{\wt{\D}(\xi,n\gamma\xi+\varphi),(\xi',n\gamma\xi'+\varphi')}\\
&=\Ideal{\D^*\xi,\xi'}+\Ideal{\D^*n\gamma\xi,n\gamma\xi'}+
\Ideal{\D\varphi,n\gamma\xi'}+\Ideal{\D^*n\gamma\xi,\varphi'}\\
&=\Ideal{\D^*\xi,\xi'}+\Ideal{[\D^*,n]\gamma\xi,n\gamma\xi'}+
\Ideal{n\D^*\gamma\xi,n\gamma\xi'}+
\Ideal{\varphi,\D^*n\gamma\xi'}+\Ideal{n\gamma\xi,\D\varphi'}\\
&=\Ideal{\D^*\xi,\xi'}+\Ideal{\gamma\xi,[\D^*,n]n\gamma\xi'}+
\Ideal{\D^*\xi,n^2\xi'}+\Ideal{\varphi,\D^*n\gamma\xi'}+\Ideal{n\gamma\xi,\D\varphi'}\\
&\qquad\text{(since $[\D^*,n]$ is symmetric)}\\
&=\Ideal{\xi,\D^*\xi'}+
\Ideal{\gamma\xi,[\D^*,n]n\gamma\xi'}+\Ideal{\xi,\D^*n^2\xi'}+
\Ideal{\varphi,\D^*n\gamma\xi'}+\Ideal{n\gamma\xi,\D\varphi'}\\
&\qquad\text{(since $(n^2+1)\cdot\Dom(n)\subset\Dom(\D)$)}\\
&=\Ideal{(\xi,n\gamma\xi+\varphi),\wt{\D}(\xi',n\gamma\xi'+\varphi')},
\end{align*}
after some rearranging, which shows that $\wt{\D}$ is symmetric.

We now show that $\Dom(\wt{\D}^*)\subset\Dom(\wt{\D})$.  Let $(\eta,\zeta)\in\Dom(\wt{\D}^*)$, 
which means that there exists $(\rho,\sigma)\in\wt{\h}$ such that for all 
$(\xi,n\gamma\xi+\varphi)\in\Dom(\wt{\D})$, with $\varphi\in\Dom(\D)$, we have
\begin{align}\label{eq:adjointequationfortildeD}
\Ideal{\wt{\D}(\xi,n\gamma\xi+\varphi),(\eta,\zeta)}=\Ideal{(\xi,n\gamma\xi+\varphi),(\rho,\sigma)}.
\end{align}
Since $\wt{\D}$ is an extension of $\D\oplus\D$, the adjoint $\wt{\D}^*$ is a restriction of $\D^*\oplus\D^*$, 
and so $(\rho,\sigma)=(\D^*\eta,\D^*\zeta)$. Rearranging Equation \eqref{eq:adjointequationfortildeD}, we have
\begin{align*}
-\Ideal{\D^*\xi,\eta}+\Ideal{\xi,\D^*\eta}=\Ideal{\D^*n\xi,\gamma\zeta}-\Ideal{n\xi,\D^*\gamma\zeta}
\end{align*}
for all $\xi\in\Dom(n)$, which by Condition 7) of Definition \ref{defn:normal} implies that $\gamma\zeta+n\eta\in\Dom(\D)$. 
Applying the grading operator $\gamma$ yields $\zeta-n\gamma\eta\in\Dom(\D)$ and hence $(\eta,\zeta)\in\Dom(\wt{\D})$, 
and thus we have established that $\wt{\D}$ is self-adjoint.
\end{proof}

\begin{prop}
\label{prop:doublecommutators}
Let $(\J\lhd \A,\h,\D,n)$ be an even relative spectral triple for $J\lhd A$ with Clifford normal. We set
$$
\wt{\A}:=\{(a,b)\in\A\oplus \A:a-b\in \J\},
$$
and define $\wt{\D}$ as in Definition \ref{defofdtwiddle}. The $\Z/2$-graded commutators 
$[\wt{\D},\wt{a}]_\pm$ are defined and bounded on $\Dom(\wt{\D})$ for all $\wt{a}\in\wt{\A}$.
\end{prop}
\begin{proof}
Let $(\xi,n\gamma\xi+\varphi)\in\Dom(\wt{\D})$, 
where $\varphi\in\Dom(\D)$, and let $(a,a+j)\in\wt{\A}$, where $j\in\J$. Then
$$
(a\xi,an\gamma\xi+jn\gamma\xi+(a+j)\varphi)
=(a\xi,n\gamma a\xi+[a,n]_\pm\gamma\xi+jn\gamma\xi+(a+j)\varphi)
$$
which is in $\Dom(\wt{\D})$ since $[a,n]_\pm\cdot\Dom(n)\subset\Dom(\D)$, 
which is Condition 4) of Definition \ref{defn:normal}. The boundedness of the 
commutators follows from the fact that $[\D^*,a]_\pm$ is bounded for all $a\in\A$.
\end{proof}

The following result will be used to show that $\wt{\D}$ 
has locally compact resolvent.

\begin{prop}
\label{prop:extresolvent}
Let $T$ be a closed symmetric operator on a separable 
Hilbert space $\h$, let $a\in \B(\h)$ be an operator such that $a(1+T^*T)^{-1/2}$ and 
$a(1-P_{\ker(T^*)})(1+TT^*)^{-1/2}$ are compact, 
and let $T\subset T_e\subset T^*$ be a closed extension of $T$. 
Then $a(1+T_e^*T_e)^{-1/2}$ is compact if and only if $aP_{\ker(T_e)}$ is compact.
\end{prop}

\begin{proof}

For a closed operator $S$ on $\h$, let $a_S:\Dom(S)\rightarrow\h$ 
be the composition of $a$ with the inclusion $\Dom(S)\hookrightarrow\h$.
Since $(1+S^*S)^{-1/2}:\h\rightarrow\Dom(S)$ is unitary, 
Lemma \ref{lem:domains}, (where $\Dom(S)$ is 
equipped with the graph inner product), $a(1+S^*S)^{-1/2}$ 
is compact as an operator on $\h$ if and only if $a_S$ is compact.
 
We can write 
$$
a_{T^*}=a_{T^*}P_{\ker(T^*)}+a_{T^*}(1-P_{\ker(T^*)}),
$$
where the second term is compact since $a(1-P_{\ker(T^*)})(1+TT^*)^{-1/2}$ is compact.
 
Let $T_e$ be a closed operator with $T\subset T_e\subset T^*$. We write $\iota$ for 
the domain inclusion $\Dom(T_e)\to \Dom(T^*)$. Then
\begin{align*}
a_{T_e}&=a_{T^*}\iota
=a_{T^*}P_{\ker(T^*)}\iota+a_{T^*}(1-P_{\ker(T^*)})\iota\\
&=a_{T^*}P_{\ker(T_e)}+\text{ compact operator.}
\end{align*}
Since the graph inner product and 
$\h$-inner product agree on $\ker(T^*)$, $a_{T^*}P_{\ker(T_e)}$ 
is compact if and only if $a P_{\ker(T_e)}$ is compact.  
Using the fact that  $a_{T_e}$ is compact if and only if 
$a(1+T_e^*T_e)^{-1/2}\in \K(\h)$ completes the proof.
\end{proof}

\begin{lemma}
\label{lem:doublekernel}
With $\wt{\D}$ as above, $\ker(\wt{\D})=\ker(\D)\oplus\ker(\D)$, and hence $\wt{a}P_{\ker(\wt{D})}$ 
is compact for all $\wt{a}\in\wt{A}$.
\end{lemma}

\begin{proof}
Let $(\xi,n\gamma\xi+\varphi)\in\ker(\wt{\D})$, 
where $\varphi\in\Dom(\D)$, so $\xi,n\gamma\xi+\varphi\in\ker(\D^*)$. 
Since $\{\D^*,\gamma\}=0$,  $\gamma$ preserves $\ker(\D^*)$, 
and so $-\gamma(n\gamma\xi+\varphi)=n\xi-\gamma\varphi\in\ker(\D^*)$. Hence
\begin{align*}
0&=\Ideal{\xi,\D^*(n\xi-\gamma\varphi)}-\Ideal{\D^*\xi,n\xi-\gamma\varphi}
=\Ideal{\xi,\D^*n\xi}-\Ideal{\D^*\xi,n\xi}=\Ideal{[\xi],[\xi]}_n
\end{align*}
since $\gamma\varphi\in\Dom(\D)$ and so 
$\Ideal{\xi,\D\gamma\varphi}=\Ideal{\D^*\xi,\gamma\varphi}$. 
The definiteness of $\Ideal{\cdot,\cdot}_n$ implies that 
$[\xi]=0$; i.e. $\xi\in\Dom(\D)$. This in turn implies that 
$n\gamma\xi+\varphi\in\Dom(\D)$, and hence $(\xi,n\gamma\xi+\varphi)\in\ker(\D)\oplus\ker(\D)$. 
If $(a,b)\in\wt{A}$, then $(a,b)P_{\ker(\wt{D})}=aP_{\ker(\D)}\oplus b P_{\ker(\D)}$ is compact.
\end{proof}

\begin{rem}
Let $\D_{\text{min}}$ be the minimal closed extension of a 
Dirac operator $\Dsla$ on a Clifford module over a 
compact manifold with boundary, as in Subsubsection \ref{eg:diracboundary}. 
Then $\ker(\D_{\text{min}})=\{0\}$, \cite{BBW}*{Corollary 8.3}. 
The above result corresponds to the doubled operator $\wt{\D}$ 
being invertible in this case, \cite{BBW}*{Theorem 9.1}.
\end{rem}

By combining Propositions \ref{prop:doubleadjoint}, 
\ref{prop:doublecommutators}, \ref{prop:extresolvent} and 
Lemma \ref{lem:doublekernel}, we arrive at the main result of this section.

\begin{thm}
\label{thm:doubletriple}
Let $(\J\lhd \A,\h,\D,n)$ be a relative spectral triple with Clifford normal for a graded ideal 
$J$ in a $\Z/2$-graded, separable,  $C^*$-algebra $A$. 
The triple $(\wt{\A},\wt{\h},\wt{\D})$ is a spectral triple for the pullback algebra 
$\wt{A}=\{(a,b)\in A\oplus A:a-b\in J\}$.
\end{thm}

\begin{rem}
\label{rem:fu}
The double construction for algebras is functorial 
for maps  $\phi:(\mathcal{I},\mathcal{B})\rightarrow(\J,\A)$, meaning
that
we can define $\wt{\phi}:\wt{\mathcal{B}}\rightarrow\wt{\A}$.
Moreover the double construction for relative spectral triples with Clifford normals is also functorial, so
we also find that if 
$(\mathcal{I}\lhd \mathcal{B},\h,\D,n)=\phi^*(\J\lhd \A,\h,\D,n)$ 
then $(\wt{\mathcal{B}},\wt{\h},\wt{\D})=\wt{\phi}^*(\wt{\A},\wt{\h},\wt{\D})$.
\end{rem}

\begin{eg}
In a limited and very non-unique way, the doubling construction is invertible.
Let $(\wt{\A},\wt{\h},\wt{\D})$ be an even 
spectral triple for a unital $C^*$-algebra $\wt{A}$ represented non-degenerately on $\wt{\h}$. Suppose there is an odd involution $Z$ on $\wt{\h}$
that preserves the domain of $\D$ and commutes with $\D$. 
The involution $Z$ implements a non-graded $\Z/2$-action on 
$(\wt{\A},\wt{\h},\wt{\D})$. Assume that there is a graded decomposition 
\begin{equation}
\label{decomposofw}
\wt{\h}=\h^1\oplus \h^2
\quad\mbox{in which}\quad 
Z=\begin{pmatrix} 0 & U\\ U^* & 0\end{pmatrix},
\end{equation}
for an odd unitary $U:\h^2\to \h^1$. 
The possible decompositions as in Equation \eqref{decomposofw} 
stand in a one-to-one correspondence with closed 
graded subspaces $\h^1\subseteq \wt{\h}$ 
such that $\h^1\perp Z\h^1$ and 
$\wt{\h}=\h^1+Z\h^1$. If  $\h^1\subseteq \wt{\h}$ 
is such a subspace, we obtain a decomposition 
as in Equation \eqref{decomposofw} by 
setting $\h^2:=Z\h^1$ and $U:=Z|_{\h^2}$.
In general, neither existence nor uniqueness of 
such decompositions can be guaranteed. 
Usually they arise from further information available in the example at hand.

We let $\A$ denote the commutant of $Z$ in $\wt{\A}$. 
If $\h^1$ is $\wt{\D}$-invariant, we define $\D$ by restricting 
$\wt{\D}$ to $\Dom(\D):=\Dom(\wt{\D})\cap \h^1$. 
Then $\D$ is easily seen to be symmetric. 
If the domain $\Dom(\D)$ is preserved by $\A$, $\D$ has 
bounded commutators with $\A$. We define 
$\J:=\{j\in\A:\,j\cdot\Dom(\D^*)\subset\Dom(\D)\}$, 
and let $J$ denote its $C^*$-closure. 
Under the assumption that $\h^1$ is chosen such that: 
\begin{enumerate}
\item $\h^1\cap \Dom(\wt{\D})\subseteq \h^1$ is dense and 
$\wt{\D}(\h^1\cap \Dom(\wt{\D}))\subseteq \h^1$;
\item $\h^1$ is preserved by $\mathcal{A}$;
\item $\D$ is closed with $\Dom(\D):=\h^1\cap \Dom(\wt{\D})$;
\end{enumerate}
it is straightforward to check that $(\J\lhd \A,\h^1,\D)$ is a relative 
Kasparov module for $(J\triangleleft A,B)$. This is the ``inverse'' construction of the 
double construction presented above.

The operator $\D$ is highly dependent on the choice of $\h^1$. 
For instance, let $M$ be a compact manifold with 
boundary and $2M$ its double. Let $S$ be a $\Z/2$-graded Clifford module on $M$, 
which extends to a $\Z/2$-graded Clifford module on $2M$, 
also denoted $S$. The flip mapping 
$\sigma:2M\to 2M$ (interchanging the two copies of $M$) lifts to 
an odd involution $Z$ on $\wt{\h}:=L^2(2M,S)$. 

For suitable 
choices of geometric data, we can construct a Dirac 
operator $\wt{\D}$ on $S\to 2M$ commuting with $Z$. 
If we take $\h^1=L^2(M,S)$ (for one copy of $M\subseteq 2M$) 
$\D$ will be the minimal closed extension of $\wt{\D}$ restricted to $M$. 
The conditions $\h^1\perp Z\h^1$ and $\wt{\h}=\h^1+ Z\h^1$ are
satisfied whenever $\h^1=L^2(W,S)$ for a subset $W\subseteq 2M$ 
such that $\sigma(W)\cup W$ has full 
measure and $W\cap \sigma(W)$ has zero measure. 
It suffices for $W$ to be open for 
$(C^\infty_0(M^\circ)\lhd C^\infty(\overline{M}),\h^1,\D)$ to be a 
relative spectral triple. 
Here $C^\infty(\overline{M})$ is the restriction of $C^\infty(2M)$ to $\ol{M}$.
\end{eg}

\begin{eg}(Clifford normals on manifolds with boundary)
We have already shown that for a manifold with boundary, Clifford multiplication by a smooth
extension of the unit normal vector field provides a Clifford normal, 
see Example \ref{classicalexampleonmfd}. The abstract double construction
applied to the relative spectral triple with Clifford normal of a manifold with boundary produces
the spectral triple of the doubled manifold (cf. \cite{BBW}).
\end{eg}

\begin{eg}(Clifford normals and dimension drop algebras)
For a dimension drop algebra on a manifold with boundary, Clifford multiplication
by a smooth
extension of the unit normal vector field provides a 
Clifford normal. The double is determined
from the double of the manifold with boundary using Remark \ref{rem:fu}.
\end{eg}

\subsection{Clifford normals on manifolds with conical singularities}
\label{cliffordoncon}
We consider conical manifolds as in Subsubsection \ref{eg:cones-on-the-brain}.
For conical manifolds, we can characterise all symmetric extensions and 
describe their normal structures. Under the isomorphism $V\cong \oplus_{i=1}^lW_i$ 
of the deficiency space of $\D_{\textnormal{min}}$, the 
symplectic form $\omega_{\D_{\textnormal{max}}}$ 
corresponds to the direct sum of the symplectic forms
$$
\omega_i(f,g)
:=\langle f,c(n_N)g\rangle_{L^2(N_i,S|_{N_i})},\quad f,\,g\in W_i
=\chi_{\left(-\frac{1}{2},\frac{1}{2}\right)}(\Dsla_{N_i})L^2(N_i,S|_{N_i}).
$$
Let us briefly sketch why the symplectic form 
$\omega_{\D_{\textnormal{max}}}$ takes the form it does. 
The ``Clifford normal" $n$ anticommutes with $\Dsla_{N_i}$. 
In particular, if $f\in W_i$ is an eigenvector with 
eigenvalue $\lambda$ then $n f\in W_i$ is an 
eigenvector with eigenvalue $-\lambda$. If we take 
two eigenvectors $f$ and $g$ of $\Dsla_{N_i}$, with 
eigenvalues $\lambda$ and $\mu$ we have by 
partial integration on the open submanifold 
$M_\epsilon:=\{r>\epsilon\}\subseteq M$ that
$$
\omega_{\D_{\textnormal{max}}}(Tf,Tg)=
\lim_{\epsilon\to 0}\int_{N_i}\epsilon^{\lambda+\mu}\langle f,c(n_N)g\rangle.
$$
Since $\overrightarrow{n}_N$ interchanges the eigenspaces of 
opposite signs, this can be non-zero if and only if 
$\mu+\lambda=0$ in which case the limit is finite. 
We also note that non-degeneracy of 
$\omega_{\D_{\textnormal{max}}}$ follows from Lemma \ref{nondegeprop}.

We will view $W$ as a $\C^l$-module 
under the decomposition $W= \oplus_{i=1}^lW_i$. 
The $\C^l$-action is compatible with the 
symplectic form $\omega_{\D_{\textnormal{max}}}$. 
For a subspace $L\subseteq W$ we denote its annihilator by
$$
L^\perp:=\{x\in W: \omega_{\D_{\textnormal{max}}}(x,y)=0\;\forall \,y\in L\}.
$$
Recall that a subspace $L\subseteq W$ of a 
symplectic vector space is called isotropic if 
$L\subseteq L^\perp$ and Lagrangian if $L=L^\perp$. 
The following Proposition is an immediate 
consequence of Theorem \ref{alltheconeprop} and the definition
of the map $T:\,W\to \Dom(\D_{\textnormal{max}})$.

\begin{prop}
\label{closedectofdir}
Extensions $\hat{\D}$ of $\D_{\textnormal{min}}$ contained in 
$\D_{\textnormal{max}}$ stand in a 
bijective correspondence to subspaces $L\subseteq W$ via 
$$
\Dom(\hat{\D})=\Dom(\D_{\textnormal{min}})+T(L).
$$
Denoting the extension associated with 
$L$ by $\D_L$, we have the following properties:

1) $\D_L^*=\D_{L^\perp}$;

2) $\D_L$ is symmetric if and only if $L$ is isotropic;

3) $\D_L$ is self-adjoint if and only if $L$ is Lagrangian.
\end{prop}

Recall the notation $\J=C_c^\infty(M\take\mathfrak{c})$, $\A=\J+\sum_{i=1}^\ell\C\chi_j$, 
where $\chi_j$ is a cutoff near the conical point $c_j$. The action of $\A$ on 
$\Dom(\D_{\textnormal{max}})$ induces the $\C^l$-action on $W$ via $\A/\J\cong \C^l$.
Although $\J$ maps $\Dom(\D_{\textnormal{max}})$ to 
$\Dom(\D_{\textnormal{min}})$, beware that $\A$ 
preserves $\Dom(\D_L)$, for a subspace $L\subseteq W$, 
if and only if $L$ is a $\C^l$-submodule of $W$. We 
will see that a choice of normal structure in the sense of 
Definition \ref{defn:normal} is a choice of complex 
structure $I$ on $V$ compatible with $\omega$ and such that $I$ is 
$\omega_{\D_{\textnormal{max}}}$-compatible (i.e. 
$\omega_{\D_{\textnormal{max}}}(I[\xi],I[\eta])=\omega_{\D_{\textnormal{max}}}([\xi],[\eta])$)
and tames $\omega_{\D_{\textnormal{max}}}$ (i.e. $\omega_{\D_{\textnormal{max}}}(\cdot, I\cdot)$ 
is positive semi-definite).

\begin{thm}
\label{conicalandnormal}
Let $M$ be a conical manifold with a Dirac operator 
$\Dsla$ acting on a Clifford bundle $S$ and 
$\A$ is as above. 

1) The relative spectral triples $(\J\lhd \A,L^2(M,S),\hat{\D})$ for $J=C_0(M\take\mathfrak{c})\lhd A=J+\sum_{i=1}^\ell\C \chi_j$,
with $\D_{\textnormal{min}}\subseteq \hat{\D}\subseteq \D_{\textnormal{max}}$, 
stand in a one-to-one correspondence with 
$\C^l$-invariant isotropic graded subspaces 
$L\subseteq W$ via $\hat{\D}=\D_L$. The triple
$(\A,L^2(M,S),\D_L)$ is a spectral triple for $\A$ if and 
only if $L$ is Lagrangian.

2) The normal structures $n$ for $(\J\lhd \A,L^2(M,S),\D_L)$ 
stand in a one-to-one correspondence with $\omega_{\D_L^*}$-compatible odd 
$\C^l$-linear complex structures $I=\bigoplus_{i=1}^lI_i$ on 
$L^\perp/L=\bigoplus_{i=1}^l (L^\perp\cap W_i)/(L\cap W_i)$ that 
tame $\omega_{\D_L^*}$, and the normal structure 
associated with $I$ is determined by the identity $nT=TI$.

3) Whenever $( \J\lhd\A,L^2(M,S),\D_L)$ admits a normal structure, 
there exists a $\C^l$-invariant Lagrangian graded subspace
$L\subseteq \wh{L}\subseteq L^\perp$ such that the spectral triple $(\A,L^2(M,S),\D_{\wh{L}})$ 
lifts the class of the relative spectral triple $(\J\lhd \A,L^2(M,S),\D_L)$ 
in $K^*(J\lhd A)$ under the mapping $K^*(A)\to K^*(J\lhd A)$.
\end{thm}

\begin{proof}
Part 1 is immediate from Proposition \ref{closedectofdir} and the discussion after it. We 
fix a $\C^l$-invariant isotropic graded subspace $L\subseteq W$ for the remainder of the 
proof. Assume that $n$ is a Clifford normal for $(\J\lhd \A,L^2(M,S),\D_L)$. Since 
$\check{\mathfrak{H}}$ is finite-dimensional by Theorem \ref{alltheconeprop}, 
the Clifford normal $n$ must preserve $\Dom(\D_L^*)=\Dom(\D_{L^\perp})$.
Therefore $n$ induces an odd $\C^l$-linear complex structure on $L^\perp/L$.

Conversely, if $I$ is an odd $\C^l$-linear complex structure on $L^\perp/L$ we can 
define an odd operator $n_0$ by declaring $n_0T=TI$ on $T(L^\perp/L)$ and extending by 
$0$ on the orthogonal complement in $L^2(M,S)$. Since 
$(n_0+n_0)^*\Dom(\D_L^*)\subseteq \Dom(\D_L)$ 
and $(n_0\gamma+\gamma n_0)\Dom(\D_L^*)\subseteq \Dom(\D_L)$, $n_0$ defines 
a normal structure. Part 3 follows easily from symplectic considerations. 
If there exists a $\omega_{\D_L^*}$-compatible odd $\C^l$-linear complex structure 
$I=\bigoplus_{i=1}^lI_i$ taming $\omega_{\D_L^*}$, there are graded $\C^l$-invariant Lagrangian 
subspaces $L_0\subseteq L^\perp/L$. This follows because we can find a graded $\C^l$-invariant 
subspace $L_0$ such that $L_0=(JL_0)^\perp=L^\perp_0$ where the first $\perp$ is taken relative 
to the inner product $\omega_{\D_{\textnormal{max}}}(\cdot, I\cdot)$. 
We can take $\wh{L}\subseteq L^\perp$ as the preimage of such an $L_0$. 
That $\wh{L}$ is Lagrangian 
follows because $x\in \wh{L}^\perp$ if and only if $x\mod L\in L_0^\perp=L_0$.

\end{proof}

\begin{rem}
\label{geoemtricnormal}
The geometric normal $\overrightarrow{n}_N$ 
to the cross-section $N$ defines an odd anti-selfadjoint 
operator $n$ via Clifford multiplication. 
It follows from Theorem \ref{alltheconeprop} that $n$ preserves 
$\Dom(\D_{L^\perp})$ (and in turn also $\Dom(\D_L)$) if and only if 
$$
L^\perp\cap W\subseteq \bigoplus_{i=1}^l \chi_{\{0\}}(\Dsla_{N_i})L^2(N_i,S|_{N_i}).
$$
For instance, if $L=0$ then $n$ is a Clifford normal for 
$(\J\lhd \A,L^2(M,S),\D_{\textnormal{min}})$ 
if and only if $\sigma(\Dsla_{N_i})\cap (-\frac{1}{2},\frac{1}{2})\subseteq \{0\}$ 
for all $i$. We do however note that $n$ always acts on 
$\Dom(\D_{\textnormal{max}})/\Dom(\D_{\textnormal{min}})$ as an odd 
$\C^l$-linear complex structure and as such \emph{induces} a Clifford normal 
for the minimal closed extension.
\end{rem}

\begin{rem}
\label{khomcomp}
Part 3 of Theorem \ref{conicalandnormal} and the above remark imply that 
$\de [(\J\lhd \A,L^2(M,S),\D_L)]=0\in K^{*+1}(\C^l)$ for any $l$ (cf. Theorem \ref{thm:extensions}). 
The vanishing of $\de [(\J\lhd \A,L^2(M,S),\D_L)]$ is implemented by the fact that $[(\J\lhd \A,L^2(M,S),\D_L)]$ 
is in the image of $K^*(A)\to K^*(J\lhd A)$. The spectral triples considered in \cite{lescurethesis} 
will in our language correspond to $L=V^+$ (the even part of $V$).
\end{rem}

\section{Constructing a spectral triple for the ``boundary'' $A/J$}
\label{sect:tripleonbdry}

Using the choice of a Clifford normal $n$ of a relative spectral triple 
$(\J\lhd\A,\h,\D)$ for $J\lhd A$, we constructed a ``boundary'' 
Hilbert space $\mathfrak{H}_n$ in Subsection \ref{subsectionnormal}. In this section  we will construct a 
densely defined $*$-action of $\A/\J\hat{\otimes}\Cl_1$ on $\mathfrak{H}_n$ and a densely defined operator $\D_n$.
For a manifold with boundary, these objects  correspond to pointwise multiplication and
the boundary Dirac operator respectively. The goal is for $(\A/\J\wh{\otimes}\Cl_1,\mathfrak{H}_n,\D_n)$ to be a geometrically constructed spectral triple which represents the class of the boundary $\de[(\J\lhd\A,\h,\D)]\in KK^1(A/J,\C)$.

In general, the existence of a Clifford normal allows the definition of these various objects, but need  
not guarantee their good behaviour.
In particular, the assumptions on a Clifford normal need not guarantee 
that the object $(\A/\J\wh{\otimes}\Cl_1, \mathfrak{H}_n,\D_n)$ 
is a spectral triple. In Subsection \ref{sect:calderon} we consider even relative spectral triples for which the algebra is unital and represented non-degenerately. For such relative spectral triples we construct the 
``Calderon projector'' $P_C$ on the Hilbert space dual of $\mathfrak{H}_n^{1/2}$. 
The operator $P_C$ plays the role of the Calderon projector on a manifold with boundary. 
In Subsection \ref{representingsubsec} we impose the additional constraint that $A$ is trivially $\Z/2$-graded. We use the 
Calderon projector to show that $\de[(\J\lhd\A,\h,\D)]=[(\A/\J\wh{\otimes}\Cl_1,\mathfrak{H}_n,\D_n)]$ under various assumptions.

\subsection{The action of the boundary algebra}

We will see in this subsection how the action of $\A$ on 
$\Dom(\D^*)$ and the Clifford normal $n$ induces an action 
of $\A/\J\hat{\otimes}\Cl_1$ on $\mathfrak{H}^{1/2}_n$. 

\begin{lemma}
\label{lem:starhom}
Let $n$ be a Clifford normal for the relative spectral triple 
$(\J\lhd \A,\h,\D)$ for $J\lhd A$. Given $[a]\in\A/\J$, define 
$$
\rho_\de([a]):\mathfrak{H}^{1/2}_n\rightarrow\mathfrak{H}^{1/2}_n,
\quad\rho_\de([a])[\xi]=[a\xi].
$$
The linear map $[a]\mapsto\rho_\de([a])$ is multiplicative, 
and respects the $*$-operation defined from 
the inner product $\Ideal{\cdot,\cdot}_n$ in the sense that
$$
\Ideal{\rho_\de([a])[\xi],[\eta]}_n
=\Ideal{[\xi],\rho_\de([a]^*)[\eta]}_n,\qquad[\xi],[\eta]\in\mathfrak{H}^{1/2}_n.
$$
The mapping $\A\to \B(\mathfrak{H}^{1/2}_n)$, $a\mapsto \rho_\de([a])$ is continuous in the $\Lip$-topology 
(see Remark \ref{lipremark} on page \pageref{lipremark}).
\end{lemma}

\begin{proof}
That $[a]\mapsto\rho_\de([a])$ is multiplicative is immediate. 
Let $\xi,\eta\in\Dom(n)$, and let $a\in\A$ be of homogeneous degree. Then
\begin{align*}
&\Ideal{\rho_\de([a])[\xi],[\eta]}_n-\Ideal{[\xi],\rho_\de([a]^*)[\eta]}_n\\
&=\Ideal{a\xi,\D^*n\eta}-\Ideal{\D^*a\xi,n\eta}-\Ideal{\xi,\D^*na^*\eta}+\Ideal{\D^*\xi,na^*\eta}\\
&=\Ideal{\xi,a^*\D^*n\eta}-\Ideal{\D^*a\xi,n\eta}
-(-1)^{\deg a}\Ideal{\xi,\D^*a^*n\eta}-\Ideal{\xi,\D[n,a^*]_\pm\eta}\\
&\quad+(-1)^{\deg a}\Ideal{\D^*\xi,a^*n\eta}+\Ideal{\D^*\xi,[n,a^*]_\pm\eta}\\
&=-(-1)^{\deg a}\Ideal{\xi,[\D^*,a^*]_\pm n\eta}-\Ideal{[\D^*,a]_\pm \xi,n\eta}=0,
\end{align*}
since $([\D^*,a]_\pm)^*\supset-(-1)^{\deg a}[\D^*,a^*]_\pm$.
\end{proof}

\begin{rem}
It is unclear when $\rho_\de(\A/\J)\subseteq \B(\mathfrak{H}_n)$. The intricate definition of the inner 
product $\Ideal{\cdot,\cdot}_n$ in terms of the symplectic form $\omega_{\D^*}$ makes 
norm estimates difficult. 
\end{rem}

\begin{defn}[Graded structure of the boundary]
\label{defn:gradingonboundary}
Let $n$ be a Clifford normal for the relative spectral triple 
$(\J\lhd \A,\h,\D)$. We consider the cases that $(\J\lhd\A,\h,\D)$ is even and odd separately.

Suppose first $(\J\lhd \A,\h,\D)$ is even. The boundary Hilbert space 
$\mathfrak{H}_n$ inherits the $\Z/2$-grading of $\h$.
Define a representation of the Clifford algebra $\Cl_1$ on $\mathfrak{H}_n$ by $c\mapsto -in_\de$, 
where $c$ is the self-adjoint unitary generator of $\Cl_1$. 
This is a well-defined representation by Lemma \ref{lem:n1}. 
Condition 4) of Definition \ref{defn:normal} ensures that 
$[n_\de,\rho_\de[a]]_\pm=0$ for all $a\in \A/\J$. The mapping 
$a\wh{\otimes}z\mapsto\rho_\de(a)z$ defines a $*$-action of 
$\A/\J\wh{\otimes}\Cl_1$ on the dense subspace space $\mathfrak{H}^{1/2}_n \subseteq \mathfrak{H}_n$.

Suppose now that $(\J\lhd\A,\h,\D)$ is odd. Lemma \ref{lem:n1} 
shows that $\gamma:=-in_\de$ defines a grading operator on 
$\mathfrak{H}_n$; i.e. $\gamma=\gamma^*$ and $\gamma^2=1$.
\end{defn}

\subsection{The boundary Dirac operator}

Let $n$ be a Clifford normal for the relative spectral triple $(\J\lhd \A,\h,\D)$ for $J\lhd A$. 
We construct an (unbounded) operator $\D_n$ on the boundary Hilbert space $\mathfrak{H}_n$, 
so that $(\A/\J\wh{\otimes}\Cl_1,\mathfrak{H}_n,\D_n)$ is a candidate for a spectral triple. 
We need an additional assumption in order to construct $\D_n$, however.

\begin{defn}
We define the spaces
\begin{align*}
\h^2_n:=\{\xi\in\Dom(n):\D^*(n\xi),\ \D^*\xi\in\Dom(n)\},\qquad\h^2_{n,0}:=\h^2_n\cap\Dom(\D).
\end{align*}
\end{defn}

\begin{rem}
For a manifold with boundary, $\h^2_n=H^2(\ol{M},S)$ 
and $\h^2_{n,0}=H^2(\ol{M},S)\cap H^1_0(M,S)$. 
\end{rem}

\begin{defn}
Let $n$ be a Clifford normal for the relative spectral triple with normal
$(\J\lhd\A,\h,\D)$ such that $[\D^*,n]\cdot\h^2_{n,0}\subset\Dom(\D)$. 
We define $\D_n:\Dom(\D_n)\subset\mathfrak{H}_n\rightarrow\mathfrak{H}_n$ by
$$
\Dom(\D_n)=\{[\xi]:\xi\in\h^2_n\}\subset\mathfrak{H}_n\quad\mbox{by}\quad
\D_n[\xi]:=\left[\frac{1}{2}n[\D^*,n]\xi\right].
$$
We call the operator $\D_n$ the \textbf{boundary operator}. 
If $(n^2+1)\cdot\h^2_n\subset\Dom(\D^2)$ then $n_\de$ preserves $\Dom(\D_n)$ and
$\D_n$ anticommutes with $n_\de$, in which case $\D_n$ is an 
odd operator whether $(\J\lhd\A,\h,\D)$ is even or odd.
\end{defn}

\begin{rem}
Not every representative of $[\xi]$ necessarily belongs to 
$\h^2_n$. Such a representative 
must be chosen when defining $\D_n[\xi]$. However, the assumption
that $[\D^*,n]\cdot\h^2_{n,0}\subset\Dom(\D)$
implies that $\D_n$ does not depend on the particular 
choice of representative.
\end{rem}

The next result shows that there are several equivalent statements to the assumption that 
$[\D^*,n]\cdot\h^2_{n,0}\subset\Dom(\D)$. 
Note that $\{\cdot,\cdot\}$ denotes the anticommutator $\{a,b\}=ab+ba$.

\begin{prop}
\label{prop:equivalences}
We assume that $n$ is a Clifford normal for the relative spectral triple $(\J\lhd\A,\h,\D)$. 
The following are equivalent:

1) $[\D^*,n]\cdot\h^2_{n,0}\subset\Dom(\D)$;

2) For all $\xi\in\h^2_{n,0}$, $\eta\in\h^2_n$, \ 
$
\Ideal{\{\D^*,[\D^*,n]\}\xi,\eta}=\Ideal{\xi,\{\D^*,[\D^*,n]\}\eta};
$

3) For all $\xi\in\h^2_{n,0}$, $\eta\in\h^2_n$, \ 
$
\Ideal{[\D^*,[\D^*,n]]\xi,\eta}=-\Ideal{\xi,[\D^*,[\D^*,n]]\eta};
$

4) For all $\xi\in\h^2_{n,0}$, $\eta\in\h^2_n$, \ 
$
\Ideal{\D^*[\D^*,n]\xi,\eta}=\Ideal{\xi,[\D^*,n]\D^*\eta}.
$

Moreover, the boundary operator $\D_n$ is well-defined 
and symmetric with respect to the boundary inner product 
$\Ideal{\cdot,\cdot}_n$ if and only if $\{\D^*,[\D^*,n]\}$ is  
symmetric as an operator on $\h$ with domain $\h^2_n$.
\end{prop}

The proof of Proposition \ref{prop:equivalences} can be found in \cite{ForsythThesis}*{Propositions 9.10 and 9.11}.

\begin{defn}
Let $(\J\lhd\A,\h,\D)$ be a relative spectral triple with two Clifford normals $n$ and $n'$. 
We say that $n$ and $n'$ are second order equivalent if $n$ and $n'$ are equivalent 
and $(n-n')\h^2_n\subseteq \Dom(\D^2)$. We call a second order 
equivalence class of a Clifford normal a second order normal structure.
\end{defn}

The operator $\D_n$ 
only depends on the second order normal structure 
that $n$ defines.

\begin{question}
\label{q:sa-bdd-comms-resolvent}
It is unknown whether the existing assumptions on the 
Clifford normal suffice to ensure that 
$\D_n$ is self-adjoint and has bounded commutators 
with $\A/\J$, or whether $\D_n$ has $\A/\J$-compact resolvent.
\end{question}

We can not prove that the various ingredients $\mathfrak{H}_n$, 
the representation $\rho$ of $\A/\J$ on 
$\mathfrak{H}_n$ and $\D_n$ form a spectral triple. 
On the other hand, if the relative spectral triple is even and the algebra $A$ is 
unital and represented non-degenerately, we can define an analogue of the Calderon projector, 
which facilitates the identification of the image of $[(\J\triangleleft \A,\h,\D)]$ under the 
$K$-homology boundary map.

\subsection{The Calderon projector}
\label{sect:calderon}

Let $(\J\lhd \A,\h,\D)$ be an even relative spectral triple with Clifford normal $n$, 
such that the algebra $A$ is unital and represented non-degenerately. 
The odd case can be treated by constructing an associated even relative spectral triple similarly 
to \cite{Connes}*{Proposition IV.A.13}.
Recall the construction of the doubled spectral triple $(\wt{\A},\wt{\h},\wt{\D})$  from Subsection \ref{sec:double}.
We let $e_1:\h\to \wt{\h}$ denote inclusion into the first factor and $r_1:\wt{\h}\to \h$ the 
projection onto the first factor. The notations $e_1$ and $r_1$ come from the classical 
case of manifolds with boundary where they correspond to ``extension by zero" and 
restriction. Note that $e_1(\Dom(\D))=e_1(\h)\cap \Dom(\wt{\D})$ and 
$r_1(\Dom(\wt{\D}))= \Dom(n)$.

\begin{defn}
We define  $\mathfrak{H}^{-1/2}_n$ as the dual space of $\mathfrak{H}^{1/2}_n$. 
\end{defn}

Recall the pairing from Proposition \ref{bdrysympform} 
$\check{\mathfrak{H}}\times\mathfrak{H}^{1/2}_n\to\C$ given by $\langle [\xi],[\eta]\rangle_n :=\omega_{\D^*}([\xi],n_\de [\eta])$.
The pairing allows us to define a map $\check\iota:\check{\mathfrak{H}}\to \mathfrak{H}^{-1/2}_n$ via 
$\check{\iota}(\psi)(\zeta):=\langle \psi,\zeta\rangle_n$, for $\psi\in \check{\mathfrak{H}}$, $\zeta\in \mathfrak{H}^{1/2}_n$.
We can embed $\mathfrak{H}^{1/2}_n$ continuously into $\check{\mathfrak{H}}$ using the inclusion $\Dom(n)\subseteq \Dom(\D^*)$.

\begin{lemma}
\label{remarkondualrho}  
The following hold:
\newline 1) the map  $\check{\iota}:\check{\mathfrak{H}}\to \mathfrak{H}^{-1/2}_n$
is a continuous inclusion 
and $\check{\iota}(\mathfrak{H}_{n}^{1/2})$ 
is dense in $\mathfrak{H}^{-1/2}_{n}$;
\newline
2) the restriction of $\check{\iota}$ to $\mathfrak{H}_n^{1/2}$ extends to 
a continuous embedding 
$i_\de:\mathfrak{H}_n \hookrightarrow \mathfrak{H}^{-1/2}_n$;
\newline
3) the composition $R:\Dom \D^{*}\to \check{\mathfrak{H}}\to \mathfrak{H}^{-1/2}_{n}$ 
defines a continuous mapping for the graph norm on $\Dom \D^{*}$;\newline
4) there is a $\Lip$-continuous homomorphism $\A\to\A/\J\xrightarrow{\rho_\de} \B(\mathfrak{H}^{-1/2}_n)$. 
\vspace{1mm}
\newline
Moreover, the maps $\check{\iota}$ and $R$ are $\A$-linear.
\end{lemma}

\begin{proof}  
For $\xi\in\Dom(\D^*)$, $\eta\in\Dom(n)$, it is straightforward to establish the estimate
\[
|\omega_{\D^*}([\xi],n_\de [\eta])|\leq | \Ideal{\xi,\D^{*}n\eta}_\h|+
|\Ideal{\D^{*}\xi, n\eta}_\h|.
\]
By taking the infimum over all representatives $\xi$ and $\eta$, 
we obtain that  $\check{\iota}:\check{\mathfrak{H}}\to \mathfrak{H}^{-1/2}_n$ 
is continuous. We make the identifications $\mathfrak{H}_{n}^{1/2}=\left(\mathfrak{H}_{n}^{-1/2}\right)^{*}$ 
and $\mathfrak{H}_{n}^{-1/2}=\left(\mathfrak{H}_{n}^{1/2}\right)^{*}$. Let $\check{\iota}^{*}$ denote the Banach adjoint map
$(\check{\iota}_{\mathfrak{H}_n^{1/2}})^*$, 
which under these identifications takes the form 
$$ 
\check{\iota}^{*}:\mathfrak{H}_{n}^{1/2}=(\mathfrak{H}_{n}^{-1/2})^{*}\to \mathfrak{H}_{n}^{-1/2}=(\mathfrak{H}_{n}^{1/2 })^{*} ,
$$
The map $\check{\iota}^{*}$ coincides with the restriction
$\check{\iota}|_{\mathfrak{H}_n^{1/2}}:\mathfrak{H}^{1/2}_{n}\to \mathfrak{H}_{n}^{-1/2}$, 
since for $\psi,\zeta\in\mathfrak{H}^{1/2}_{n}$:
\[
(\check{\iota}^{*}(\psi),\zeta)=(\psi,\check{\iota}(\zeta))=\check{\iota}(\zeta)(\psi)
=\omega_{\D^{*}}(\zeta, n_{\partial}\psi)=\omega_{\D^{*}}(\psi, n_{\partial}\zeta)
=\check{\iota}(\psi)(\zeta). 
\]
Therefore $\check{\iota}^{*}$ is injective, and so $\check{\iota}$ 
has dense range,  proving 1).

The argument proving 2), is similar to the argument proving 1) upon noting that the dual 
of the continuous inclusion $\mathfrak{H}^{1/2}_n\hookrightarrow \mathfrak{H}_n$ is 
$i_\de:\mathfrak{H}_n \hookrightarrow \mathfrak{H}^{-1/2}_n$.

Statement 3) follows from 1) since $\Dom\D^{*}\to\check{\mathfrak{H}}$ 
is a contraction. 
For 4), we use duality and the $\Lip$-continuous map 
$\A\to\A/\J\xrightarrow{\rho_\de} \B(\mathfrak{H}^{1/2}_n)$ 
from Lemma \ref{lem:starhom}.
\end{proof}

On a manifold with boundary, $\mathfrak{H}^{-1/2}_n\cong H^{-1/2}(\de M,S|_{\de M})$ and 
the duality with $H^{1/2}(\de M,S|_{\de M})$ is implemented via the $L^2$-pairing. 

\begin{rem} 
\label{remark59}
Recall the discussion after Definition \ref{defn:normal}.
We can consider $n_\de\check{\mathfrak{H}}\subseteq  \mathfrak{H}^{-1/2}_n$ 
as a subspace by defining 
$n_\de\xi\in \mathfrak{H}^{-1/2}_n$ as the functional 
$\mathfrak{H}^{1/2}_n\ni \eta\mapsto -\langle \xi,n_\de\eta\rangle_n$ 
for $\xi\in \check{\mathfrak{H}}$. Therefore, there is a continuous embedding 
$$
\check{\iota}:\check{\mathfrak{H}}
+n_\de\check{\mathfrak{H}}\hookrightarrow  \mathfrak{H}^{-1/2}_n,
$$ 
defined from the pairing $\Ideal{\cdot,\cdot}_n$. 
In fact, it follows from Lemma \ref{nondegeprop} 
that injectivity of $\check{\iota}$ is 
equivalent to Condition 7) in Definition \ref{defn:normal}. 
Continuity is clear from the definitions of the norms. 
These inclusions are compatible with the partially defined 
$\A/\J$-actions using $\rho_\de$ (see Lemmas \ref{lem:starhom} 
and  \ref{remarkondualrho}). 
\end{rem}

\begin{defn}
\label{defn:cauchy} 
Recall the continuous map 
$R:\Dom\D^{*}\to\mathfrak{H}^{-1/2}_{n}$ from 
Lemma \ref{remarkondualrho} part 3). We define the 
{\bf Cauchy space} $\mathfrak{H}^{-1/2}_C\subseteq \mathfrak{H}^{-1/2}_n$ 
as the closure of the image of $\ker(\D^*)$ in $\mathfrak{H}^{-1/2}_n$.
\end{defn}

We shall see below, in Corollary \ref{closedcor}, 
that the image of $\ker(\D^*)$ in $\mathfrak{H}^{-1/2}_n$ 
is already closed.

Since $A$ is unital 
and represented non-degenerately on $\h$, 
the operator $\wt{\D}:\Dom(\wt{\D})\rightarrow\wt{\D}$ 
is Fredholm and hence has closed range. This allows us 
to define the pseudoinverse $\wt{\D}^{-1}\in \B(\wt{\h})$, 
which is the extension by zero of the inverse of the 
invertible operator $\wt{\D}:\ker(\wt{\D})^\perp\rightarrow\ker(\wt{\D})^\perp$.  

\begin{rem}
If $\D$ is a Dirac operator on a compact manifold with 
boundary, then $\ker(\wt{\D})=\ker(\D)\oplus\ker(\D)=\{0\}$, 
and so $\wt{\D}$ is genuinely invertible, \cite{BBW}*{Theorem 9.1}.
\end{rem}

\begin{prop}
\label{poandcal}
Let $(\J\lhd \A,\h,\D,n)$ be an even relative spectral triple with Clifford normal such that 
the $C^*$-algebra $A$ is unital and represented non-degenerately. The even operator 
$\mathcal{K}:\mathfrak{H}^{-1/2}_n \to\h$ defined by
$$
\langle \mathcal{K}f,v\rangle_\h
=\langle f, [nr_1\wt{\D}^{-1}e_1v]\rangle_{n}\quad \forall v\,\in \h,
$$
is well-defined and continuous, and has range  $\ker(\D)^\perp\subset\ker(\D^*)$. The even continous operator
$$
P_C:\mathfrak{H}^{-1/2}_n\to \mathfrak{H}^{-1/2}_n,\quad P_C :=R\circ\mathcal{K},
$$
is a (not necessarily self-adjoint) idempotent with range $\mathfrak{H}^{-1/2}_C$. 
The continuous operators 
$\mathcal{K}$ and $P_C$ only depend on the normal structure.
\end{prop}

\begin{proof}
By construction, for any $v\in \h$ we have $r_1\wt{\D}^{-1}e_1v\in\Dom(n)$ and 
$\h\ni v\mapsto r_1\wt{\D}^{-1}e_1v\in \Dom(n)$ is continuous. 
Therefore $[r_1\wt{\D}^{-1}e_1v]\in \mathfrak{H}^{1/2}_n$ is well-defined and for a 
suitable constant $c$,
\begin{align*}
\|[r_1\wt{\D}^{-1}e_1v]\|_{\mathfrak{H}^{1/2}_n}\leq c\|v\|_{ \h}.
\end{align*}
It follows that, for $f\in\mathfrak{H}^{-1/2}_n$, the functional 
$\h\ni v\mapsto \langle f, [nr_1\wt{\D}^{-1}e_1v]\rangle_{n}$ 
is continuous. By Riesz' representation theorem, 
there is a $\mathcal{K}f\in\h$ such that 
$$
\langle \mathcal{K}f,v\rangle_{ \h}=\langle f, [nr_1\wt{\D}^{-1}e_1v]\rangle_{n}
$$ 
for all $v\in \h$. It follows from the construction of the double $\wt{\D}$ that 
$$
\|[nr_1\wt{\D}^{-1}e_1v]\|_{\mathfrak{h}^{1/2}_n}
\leq \|nr_1\wt{\D}^{-1}e_1v\|_{\Dom(n)}\lesssim \|v\|_{\h},
$$ 
so $\mathcal{K}:\mathfrak{H}^{-1/2}_n \to\h$ is continuous. By Lemma \ref{remarkondualrho}, 
the map $ \ker \D^{*}\to \mathfrak{H}^{-1/2}_{n}$ is continuous with respect to the graph norm, 
which on $\ker \D^{*}$  coincides with the subspace norm $\ker \D^{*}\subset \mathcal{H}$. 
Thus, the composition $P_{C}=R\circ\mathcal{K}$ is continuous as well. Now suppose that $v\in \Dom(\D)$. 
Then we have
\begin{align*}
\langle \mathcal{K}f,\D v\rangle_\h
=\langle f, [nr_1\wt{\D}^{-1}e_1\D v]\rangle_{n}=\langle f, [nv]\rangle_{n}=0,
\end{align*}
because $nv\in \Dom(\D)$. It follows that 
$\mathcal{K}f\in \range(\D)^\perp=\ker(\D^*)$. 
If $v\in\ker(\D)$, then $\wt{\D}^{-1}e_1v=0$ since 
$\wt{\D}$ is zero on $\ker(\wt{\D})$. Hence $\range\mathcal{K}\subset\ker(\D)^\perp$.

Now assume that $u\in \ker(\D)^\perp\subset\ker(\D^*)$, and set 
$f:=[u]\in \mathfrak{H}_C^{-1/2}$. Then, for any $v\in \h$
\begin{align*}
\langle \mathcal{K}f,v\rangle_\h=\Ideal{u,\D^*r_1\wt{\D}^{-1}e_1v}_\h=\langle u,v\rangle_\h.
\end{align*}
It follows that $u=\mathcal{K}f$, so the operator 
$P_Cf:=[\mathcal{K}f]$ defines an idempotent with range $\mathfrak{H}_C^{-1/2}$.
\end{proof}

\begin{cor}
\label{closedcor}
Let $(\J\lhd \A,\h,\D,n)$ be relative spectral triple with normal. The image of $\ker(\D^*)\to \mathfrak{H}^{-1/2}_n$ is closed.
\end{cor}

\begin{defn}
\label{defn:fish}
The operator $\mathcal{K}:\mathfrak{H}^{-1/2}_n \to \ker(\D)^\perp\subset\ker(\D^*)$ 
is called the Poisson operator 
of the even relative spectral triple with Clifford normal $(\J\lhd\A,\h,\D,n)$, 
and $P_C$ is called its Calderon projector. The Calderon 
projector is an idempotent but is not necessarily an orthogonal projection.
\end{defn}

\begin{rem}
The Calderon projector $P_C$ plays an important role in the classical case of manifolds 
with boundary. In this case, $P_C$ is self-adjoint up to a compact perturbation. 
To be precise, $P_C$ is a pseudo-differential operator of order $0$ with self-adjoint 
principal symbol by \cite{BBW}*{Theorem 12.4}. Moreover, 
the idempotent $P_C$ defines an odd $K$-homology cycle whose class coincides with the image 
of $(\J\lhd\A,\h,\D)$ under the boundary mapping. We will in the next section use $P_C$ to place 
natural conditions on the Clifford normal that makes $(\A/\J\hat{\otimes}\Cl_1, \mathfrak{H}_n,\D_n)$ into a 
spectral triple representing $\de[(\J\lhd\A,\h,\D)]$.
\end{rem}

\subsection{Representing the boundary in $K$-homology by a spectral triple}
\label{representingsubsec}

In the last subsection we saw how to construct a candidate for a 
``boundary spectral triple'' plus an analogue of the Calderon projector.
It is desirous to know whether the $K$-homological boundary of a ``geometrically defined'' 
relative spectral triple can also be represented ``geometrically''. In examples, such
as our running examples and the crossed products and $\theta$-deformations 
to be discussed later, we can check that our constructions 
do in fact produce a boundary spectral
triple.

In the generality that we are working in, we believe that it is impossible to prove that
all the boundary objects that we have defined assemble into a spectral triple,
let alone that this spectral triple represents the $K$-homological boundary.
Classically  the compatible geometry of the boundary is implicit
in the definition of a manifold with boundary, and one need not try to construct 
the boundary as we have done.

We overcome this impasse by making the additional assumption that the geometrically defined triple 
$(\A/\J\wh{\otimes}\Cl_1,\mathfrak{H}_n,\D_n)$  is a spectral triple for $A/J\wh{\otimes}\Cl_1$. 
This assumption can be verified for all of our examples. We then show that the Calderon
projector we have defined plays an important role in showing that the
$K$-homological boundary of $(\J\lhd\A,\h,\D)$ is represented by $(\A/\J\wh{\otimes}\Cl_1,\mathfrak{H}_n,\D_n)$, as it is classically.

\begin{ass}
\label{sixthass}
Let $(\J\lhd\A,\h,\D,n)$ be an (even or odd) relative spectral triple with Clifford normal for $J\lhd A$. 
The operator $\{\D^*,[\D^*,n]\}:\h^2_n\subset\h\rightarrow\h$ is symmetric
(cf. Proposition \ref{prop:equivalences} on page \pageref{prop:equivalences}) and the following hold.

a. $\J=J\cap\A$, and there is a $C>0$ such that for each $a\in \A$, $\rho_\de([a])$ extends to 
a bounded operator on $\mathfrak{H}_n$ satisfying
$$\|\rho_\de([a])\|_{\B(\mathfrak{H}_n)}\leq C\|a\|_{\Lip},$$ 
$\rho_\de([a])$ preserves $\Dom(\D_n)$ and the densely defined operator $[\D_n,\rho_\de([a])]_\pm$ is bounded.

b. $(n^2+1)\cdot\h^2_n\subset\Dom(\D^2)$, so that $n_\de$ preserves $\Dom(\D_n)$ and
$\D_n$ anticommutes with $n_\de$.

c. $\D_n$ is essentially self-adjoint, or equivalently $\range(\D_n\pm i)$ are dense in $\mathfrak{H}_n$.

d. $a(1+\D_n^2)^{-1/2}$ is compact for all $a\in\A/\J$, or equivalently $a:\Dom(\D_n)\rightarrow\mathfrak{H}_n$ 
is compact with respect to the graph norm, where by an abuse of notation $\D_n$ also refers to its self-adjoint closure.
\end{ass}

We will {\bf explicitly state whenever we use Assumption \ref{sixthass}}. In the even case, Assumption
\ref{sixthass} includes the assumption that our boundary data
$$
(\A/\J\wh{\otimes} \Cl_1,\mathfrak{H}_n,\D_n)
$$
assembles to give an even spectral triple for $A/J\wh{\otimes} \Cl_1$. (In the odd case, 
Assumption \ref{sixthass} ensures that $(\A/\J,\mathfrak{H}_n,\D_n)$ is an even spectral triple for $A/J$.) 
Indeed, $\rho_\de$ extends to 
a representation of $A/J\wh{\otimes} \Cl_1$ due to the following argument. 
Part a.\ of Assumption \ref{sixthass} implies that 
$\rho_\de$ defines a $*$-homomorphism 
$\overline{\A}^{\Lip}\to \B(\mathfrak{H}_n)$. 
Since $\overline{\A}^{\Lip}\subseteq A$ is closed 
under holomorphic functional calculus (see Remark \ref{lipremark} 
on page \pageref{lipremark}), we can extend to a $*$-homomorphism 
$A\to \B(\mathfrak{H}_n)$ vanishing on $\J$, and by density also vanishing on $J$. Therefore $\rho_\de$ 
extends to a $*$-homomorphism 
$\rho_\de:A/J\hat{\otimes}\Cl_1\to \B(\mathfrak{H}_n)$.

As an example, we note that on a manifold 
with boundary, the operator $\{\D^*,[\D^*,n]\}$ 
coincides with the symmetric operator $-2n\Dsla_\de^2$ near the boundary.

In this section we examine some further assumptions guaranteeing that 
$$\de[(\J\lhd\A,\h,\D)]=[(\A/\J\wh{\otimes}\Cl_1,\mathfrak{H}_n,\D_n)]\in KK^1(A/J,\C).$$
We will assume that $A$ is trivially $\Z/2$-graded and that the spectral triple $(\J\lhd\A,\h,\D)$ is even, 
since in this case we can (with Assumption \ref{sixthass} holding) construct an odd spectral triple 
$(\A/\J,\mathfrak{H}_n^+,-n_\de^-\D_n^+)$ representing the same class as 
$(\A/\J\hat{\otimes}\Cl_1, \mathfrak{H}_n,\D_n)$ in $KK^1(A/J,\C)$. 
There is an extension which represents the class of this odd spectral triple triple, 
and we can compare this extension to the extension of Proposition \ref{representingyo} 
which represents $\de[(\J\lhd\A,\h,\D)]$. 

\begin{lemma}\label{lem:oddbdrytriple}
Let $(\J\lhd\A,\h,\D,n)$ be an even relative spectral triple with Clifford normal for $J\lhd A$, 
with $A$ trivially graded, and suppose Assumption \ref{sixthass} holds so that 
$(\A/\J\wh{\otimes}\Cl_1,\mathfrak{H}_n,\D_n)$ is an even spectral triple. Write 
$n_\de=\left(\begin{smallmatrix}0&n_\de^-\\n_\de^+&0\end{smallmatrix}\right)$, 
$\D_n=\left(\begin{smallmatrix}0&\D_n^-\\\D_n^+&0\end{smallmatrix}\right)$ with respect to 
the $\Z/2$-grading of $\mathfrak{H}_n$, and let $\D_\de:=-n_\de^- \D_n^+$. Then
$$\big(\A/\J,\mathfrak{H}_n^+,\D_\de\big)$$
is an odd spectral triple representing the same class as $(\A/\J\wh{\otimes}\Cl_1,\mathfrak{H}_n,\D_n)$ 
in $KK^1(A/J,\C)$.
\end{lemma}

\begin{proof}
Since $\D_n$ anticommutes with $n_\de$, $\D_\de$ is self-adjoint. 
It is clear that $\D_\de$ has bounded commutators with $\A/\J$, 
and $a(1+\D_\de^*\D_\de)^{-1/2}=a(1+\D_n^-\D_n^+)^{-1/2}$ 
is compact for all $a\in A/J$. Hence $(\A/\J,\mathfrak{H}_n^+,\D_\de)$ 
is a spectral triple for $A/J$. 

The even spectral triple for $A/J\wh{\otimes}\Cl_1$ 
associated to $(\A/\J,\mathfrak{H}_n^+,\D_\de)$ is 
$$
\left(\A/\J\wh{\otimes}\Cl_1,\mathfrak{H}_n^+\oplus\mathfrak{H}_n^+,
\left(\begin{smallmatrix}0&in^-_\de\D_n^+\\-in^-_\de\D_n^+&0\end{smallmatrix}\right)\right),
$$
where $\mathfrak{H}_n^+\oplus\mathfrak{H}_n^+$ is graded by 
$\gamma=\left(\begin{smallmatrix}1&0\\0&-1\end{smallmatrix}\right)$, 
and the self-adjoint unitary generator of $\Cl_1$ 
acts by $c\mapsto\left(\begin{smallmatrix}0&1\\1&0\end{smallmatrix}\right)$, 
\cite{Connes}*{Proposition IV.A.13}. Take the $\Z/2$-graded unitary 
$u:=\left(\begin{smallmatrix}1&0\\0&-in_\de^+\end{smallmatrix}\right)
:\mathfrak{H}_n^+\oplus\mathfrak{H}_n^+\rightarrow\mathfrak{H}_n$. Then
\begin{align*}
u\begin{pmatrix}0&in_\de^-\D_n^+\\-in_\de^-\D_n^+&0\end{pmatrix}u^*
&=\begin{pmatrix}0&n_\de^-\D_n^+n_\de^-\\-n_\de^+n_\de^-\D_n^+&0\end{pmatrix}=\D_n
\end{align*}
while
\begin{align*}
u\begin{pmatrix}0&1\\1&0\end{pmatrix}u^*&=\begin{pmatrix}0&-in_\de^-\\-in_\de^+&0\end{pmatrix}
\end{align*}
which is the action of the Clifford generator on $\mathfrak{H}_n$. 
This establishes that the even spectral triple associated to 
$(\A/\J,\mathfrak{H}_n^+,\D_\de)$ is unitarily equivalent to 
$(\A/\J\wh{\otimes}\Cl_1,\mathfrak{H}_n,\D_n)$.
\end{proof}

Recall the Calderon projector $P_C\in \B(\mathfrak{H}^{-1/2}_n)$ 
of Proposition \ref{poandcal}. The operator $P_C$ is even and hence has a 
restriction $P_C^+:(\mathfrak{H}_n^{-1/2})^+\rightarrow(\mathfrak{H}_n^{-1/2})^+$.
We will also make use of the $\A/\J$-action $\rho_\de$ on  
$\mathfrak{H}_n^{-1/2}$, see Lemma \ref{remarkondualrho}. For the proof Theorem 
\ref{thm:boundarycomparisonnew} below, we recall the following notion.
\begin{defn}Let $B$ be a $\ast$-algebra, and let $\h_i$ be a Hilbert space, 
$\mathcal{Q}(\h_i)=\mathbb{B}(\h_i)/\mathbb{K}(\h_i)$ be the Calkin algebra, 
and $\beta_i:B\to \mathcal{Q}(\h_i)$ be a $\ast$-homomorphism for $i=1,2$. We say
that $\beta_1$ and $\beta_2$ are \textbf{similar} if there is a 
Fredholm operator $S:\h_1\to \h_2$ such that $\beta_1(a)=\pi(S)\beta_2(a)\pi(S)^{-1}$, 
where $\pi:\mathbb{B}(\h_2)\to\mathcal{Q}(\h_2)$ denotes the quotient map onto the Calkin algebra. 
\end{defn}

\begin{lemma} 
\label{lem:similaritylemma}
Let $B$ be a $C^*$-algebra and $\beta_i:B\to \mathcal{Q}(\h_i)$, $i=1,2$ be similar. 
Then $[\beta_{1}]=[\beta_{2}]\in \Ext(B,\C)$.
\end{lemma}

\begin{proof} 
Suppose that $\beta_1(a)=\pi(S)\beta_2(a)\pi(S)^{-1}$  with $S$ Fredholm. 
We may assume that $S$ is invertible by the following argument.
Replace $\h_{1}$ by $\h_{1}\oplus \ker S^{*}$, $\h_{2}$ by $\h_{2}\oplus \ker S$, 
and $S$ by $S|_{\ker{S}^\perp}\oplus v$ 
where the finite dimensional spaces $\ker S$, $\ker S^{*}$ 
carry the trivial $B$ representation and 
$v:\ker S\oplus\ker S^{*}\to \ker S^{*}\oplus \ker S$ 
is a unitary. It then follows that
\[
\pi(S)\beta_{2}(a)=\beta_{1}(a)\pi(S),\quad \pi(S^{*})\beta_{1}(a)
=\beta_{2}(a)\pi(S^{*}),\quad \pi(S^{*}S)\beta_{2}(a)=\beta_{2}(a)\pi(S^{*}S), 
\] 
and thus by polar decomposition $S=u(S^{*}S)^{1/2}$ 
and $\pi(S)\beta_{2}(a)\pi(S)^{-1}=\pi(u)\beta_{2}(a)\pi(u)^{*}$. 
It is now straightforward to show that the $\beta_{i}$ 
define $*$-isomorphic extensions and thus $[\beta_{1}]=[\beta_{2}]$.
\end{proof}

The following theorem gives sufficient conditions for when the ``geometric'' boundary 
$(\A/\J,\mathfrak{H}_n^+,\D_\de)$ represents the $K$-homological boundary of $(\J\lhd\A,\h,\D)$. 
The motivation for the theorem comes from the classical manifold setting, where there exists an 
elliptic Fredholm pseudodifferential operator $S$ between the Sobolev spaces 
$\mathfrak{H}_n^{-1/2}$ and $\mathfrak{H}_n$ intertwining $P_\geq$ and $P_C$ up to compacts. 
The proof is motivated by the proof of \cite[Proposition 4.3]{BDT}. We revisit this result below 
under conditions closer resembling the situation on a manifold.

\begin{thm}
\label{thm:boundarycomparisonnew}
Let $(\J\lhd\A,\h,\D,n)$ be an even relative spectral triple with Clifford normal for $J\lhd A$, 
where $A$ is trivially $\Z/2$-graded, unital and represented non-degenerately on $\h$. 
Suppose the relative spectral triple satisfies Assumption \ref{sixthass}, so that $(\A/\J,\mathfrak{H}_n^+,\D_\de)$ 
is an odd spectral triple for $A/J$. Let $P_\geq\in\B(\mathfrak{H}_n^+ )$ be the non-negative spectral projection 
associated to $\D_\de$. If there exists a Fredholm operator $S:(\mathfrak{H}_n^{-1/2})^+\rightarrow\mathfrak{H}_n^+$ 
such that

1. $P_{\geq}[S,\rho_\de(a)]P_C^+:(\mathfrak{H}_n^{-1/2})^+\rightarrow\mathfrak{H}_n^+$ is compact for all $a\in \A/\J$, and

2. $P_{\geq}S-SP_C^+:(\mathfrak{H}_n^{-1/2})^+\rightarrow\mathfrak{H}_n^+$ is compact,

then $\de[(\J\lhd\A,\h,\D)]=[(\A/\J,\mathfrak{H}_n^+,\D_\de)]\in KK^1(A/J,\C)$.
\end{thm}

\begin{proof}
By Lemma \ref{lem:oddbdrytriple}, the class of $(\A/\J\wh{\otimes}\Cl_1,\mathfrak{H}_n,\D_n)$ 
is represented by the odd spectral triple $(\A/\J,\mathfrak{H}_n^+,\D_\de)$. 
An extension corresponding to this  odd spectral triple under the isomorphism between 
$KK^1(A/J,\C)$ and $\Ext^{-1}(A/J,\C)$ is 
$$\tau:A/J\rightarrow\mathcal{Q}(P_{\geq}\mathfrak{H}_n^+),\quad  
\tau(a):= \pi(P_{\geq}\rho_\de(a)P_{\geq}), $$ 
where $\pi:\mathbb{B}(\mathfrak{H}_n^+)\to\mathcal{Q}(\mathfrak{H}_n^+)$ is the quotient map. 
By Proposition  \ref{representingyo}, $\de[(\J\lhd\A,\h,\D)]$ is represented by the extension 
\begin{align*}
&\alpha:A/J\rightarrow \mathcal{Q}(\ker ((\D^*)^+)),\quad \alpha(a)
:=\pi(P_{\ker((\D^*)^+)}\wt{a}P_{\ker((\D^*)^+)}),
\end{align*}
where $\wt{a}\in A$ is any pre-image of $a\in A/J$. We show that $\alpha$ is equivalent to $\tau$. 
First, we claim that the extension $\tau$ restricted to 
$\A/\J$ is similar to
$$\beta:\A/\J\rightarrow\mathcal{Q}((\mathfrak{H}_C^{-1/2})^+),\quad  
\beta(a):=\pi(P_C^+\rho_\de(a)P_C^+),$$ where $\mathfrak{H}_C^{-1/2}$ is the Cauchy space
from Definition \ref{defn:cauchy}. The extension defined by $\beta$ coincides with that defined from 
the map $a\mapsto \pi(P_C^+\rho_\de(a)P_C^+)\in\mathcal{Q}((\mathfrak{H}_n^{-1/2})^+)$. 
Let $S:(\mathfrak{H}_n^{-1/2})^+\rightarrow\mathfrak{H}_n^+$ be a Fredholm operator satisfying 1. and 2. 
This implies that 
$$\pi(P_C^+\rho_\de(a)P_C^+)=\pi(S)^{-1}\pi(P_{\geq}\rho_\de(a)P_{\geq})\pi(S)=\pi(S)^{-1}\tau(a)\pi(S) $$ for 
$a\in \A/\J$, and the claim follows.

Let $R$ be the continuous map 
$R:\Dom \D^*\rightarrow\check{\mathfrak{H}}\hookrightarrow\mathfrak{H}_n^{-1/2}$ 
from Lemma \ref{remarkondualrho}, part 3), 
and  $\mathcal{K}:\mathfrak{H}_n^{-1/2}\rightarrow\ker(\D^*)$ 
the Poisson operator from Definition \ref{defn:fish}. 
Let $R^+$, $\mathcal{K}^+$ denote the restrictions to the even subspaces. 
Since 
$$
P_C=R\mathcal{K}\in\B(\mathfrak{H}_n^{-1/2}),\quad\mathcal{K}R|_{\ker(\D^*)}=P_{\ker(\D^*)}-P_{\ker(\D)},
$$ 
and $P_{\ker(\D)}$ is compact, it follows that  
$$
\mathcal{K}^+P_C^+:(\mathfrak{H}_C^{-1/2})^+\rightarrow\ker((\D^*)^+)\quad\mbox{and}
\quad R^+P_{\ker((\D^*)^+)}:\ker((\D^*)^+)\rightarrow(\mathfrak{H}_C^{-1/2})^+
$$ 
are Fredholm and mutually inverse to each other modulo compact operators. 
Hence the extension $\alpha$ is similar to 
$$
\wt{\alpha}:\A/\J\rightarrow\mathcal{Q}((\mathfrak{H}_C^{-1/2})^+),
\quad \wt{\alpha}(a):=\pi(P_C^+R^+P_{\ker((\D^*)^+)}\wt{a}
P_{\ker((\D^*)^+)}\mathcal{K}^+P_C^+),
$$ 
where $\wt{a}\in \A$ denotes any pre-image of $a\in \A/\J$. Since 
$$
P_CRP_{\ker(\D^*)}=P_CR,\quad P_{\ker((\D^*)^+)}\mathcal{K}^+P_C^+
=\mathcal{K}^+P_C^+,\quad R\wt{a}=\rho_\de(a)R, \quad a\in \A/\J,
$$ 
it follows that $\wt{\alpha}(a)=\pi(P_C^+\rho_\de(a)P_C^+)=\beta(a)$ for all $a\in \A/\J$. 
We deduce that $\alpha$ and $\tau$ restricted to $\A/\J$ 
are similar. By density, $\alpha$ and $\tau$ are similar. 
The proof of the theorem follows from Lemma \ref{lem:similaritylemma}.
\end{proof}

Recall that if $M$ is a compact Riemannian manifold 
with boundary and $\Dsla$ is a Dirac operator on $M$, 
then $(C^\infty_0(M^\circ)\lhd C^\infty(\ol{M}),L^2(E),\D_{\text{min}})$ 
is a relative spectral triple for $C_0(M^\circ)\lhd C(\ol{M})$. 
In this case, 
$$
(\mathfrak{H}^{1/2}_n)^+=H^{1/2}(E^+|_{\de M})=\Dom((1+\D_\de^2)^{1/4}),
$$ 
and the dual space  $(\mathfrak{H}^{-1/2}_n)^+$ 
is $H^{-1/2}(E^+|_{\de M})$, which is the completion 
of $L^2(E^+|_{\de M})$ in the norm 
$\|\xi\|_{H^{-1/2}}=\|(1+\D_\de^2)^{-1/4}\xi\|$.

In the general case $(\mathfrak{H}^{1/2}_n)^+$ and 
$(\mathfrak{H}^{-1/2}_n)^+$ are not defined in terms 
of the boundary operator $\D_\de$. 
However, if we assume that $(\mathfrak{H}_n^{-1/2})^+$ 
is related to $\D_\de$ in a similar way, then $(1+\D_\de^2)^{-1/4}$  
is a Fredholm operator $(\mathfrak{H}_n^{-1/2})^+\rightarrow\mathfrak{H}_n^+$ and 
hence a candidate for an operator satisfying the conditions of Theorem \ref{thm:boundarycomparisonnew}.

\begin{prop}
\label{prop:boundarycomparisonnew}
Let $(\J\lhd\A,\h,\D,n)$ be an even relative 
spectral triple with Clifford normal for $J\lhd A$, 
where $A$ is trivially $\Z/2$-graded, unital and 
represented non-degenerately on $\h$. 
We suppose that the relative spectral triple 
satisfies Assumption \ref{sixthass}. Viewing 
$\mathfrak{H}_{n}$ as a dense subspace 
of $\mathfrak{H}_{n}^{-1/2}$  
cf. Lemma \ref{remarkondualrho}, we additionally suppose

1)  the $\mathfrak{H}_n^{-1/2}$ norm and 
$\|(1+\D_\de^2)^{-1/4}\cdot\|_{\mathfrak{H}_{n}^+}$, 
  are equivalent on $\mathfrak{H}_{n}^+$;

2) the operator $P_{\geq}-P_C^+$, which is densely 
defined on $(\mathfrak{H}_{n}^{-1/2})^+$ with domain $\mathfrak{H}_{n}^+$,  
extends to a compact operator on $(\mathfrak{H}_n^{-1/2})^+$.

Then $\de[(\J\lhd\A,\h,\D)]=[(\A/\J,\mathfrak{H}_n^+,\D_\de)]\in KK^1(A/J,\C)$.
\end{prop}
\begin{proof}
As in the proof of Theorem \ref{thm:boundarycomparisonnew}, $\de[(\J\lhd\A,\h,\D)]$ is 
represented by the extension $$\alpha:A/J\rightarrow\mathcal{Q}(\ker((\D^*)^+)),\quad  
\alpha(a)=\pi(P_{\ker((\D^*)^+)}\wt{a}P_{\ker((\D^*)^+)}),$$ and restricted to $\A/\J,$
the homomorphism $\alpha$ is similar to 
$$
\wt{\alpha}:\A/\J\rightarrow\mathcal{Q}((\mathfrak{H}_C^{-1/2})^+),\quad 
\wt{\alpha}(a)=\pi(P_C^+\rho_\de(a)P_C^+).
$$ 
The extension representing $[(\A/\J,\mathfrak{H}_n^+,\D_\de)]\in KK^1(A/J,\C)\cong \Ext^{-1}(A/J)$ 
is given by the Busby invariant 
$$
\tau:A/J\rightarrow\mathcal{Q}(P_\geq\mathfrak{H}_n^+),
\quad  \tau(a)=\pi(P_{\geq}\rho_\de(a)P_\geq).
$$
If $a\in\A/\J$, then after applying \cite[Lemma 1.5]{ciprianiguidoscarlatti} 
to the invertible operator 
$\begin{pmatrix}1&\D_\de\\\D_\de&-1\end{pmatrix}$ we find that 
$$
\rho_\de(a)\cdot\Dom((1+\D_\de^2)^{1/4})\subset\Dom((1+\D_\de^2)^{1/4}),
$$ 
and the operator $[(1+\D_\de^2)^{1/4},\rho_\de(a)]$ is bounded, 
so on $\Dom((1+\D_\de^2)^{1/4})\subset\mathfrak{H}_n^+$ and for $a\in\A/\J$,
\begin{align*}
(1+\D_\de^2)^{-1/4}P_{\geq}\rho_\de(a)P_{\geq}(1+\D_\de^2)^{1/4}=
-P_{\geq}(1+\D_\de^2)^{-1/4}[(1+\D_\de)^{1/4},\rho_\de(a)]P_\geq+P_{\geq}\rho_\de(a)P_{\geq}.
\end{align*}
The first term extends to a compact operator on $\mathfrak{H}_n^+$ since 
$(1+\D_\de^2)^{-1/4}\in\K(\mathfrak{H}_n^+)$. Hence the restriction of $\tau$ to $\A/\J$ is similar to 
$$\sigma:\A/\J\rightarrow\mathcal{Q}(\mathfrak{H}_n^+),\quad \sigma(a)=
\pi(\ol{(1+\D_\de^2)^{-1/4}P_{\geq}\rho_\de(a)P_{\geq}(1+\D_\de^2)^{1/4}}).$$
By 1) we may assume that the operator $(1+\D_\de^2)^{-1/4}$ is a unitary between 
$(\mathfrak{H}_n^{-1/2})^+$ and $\mathfrak{H}_n^+$, so $\sigma$ is similar to 
$$\wt{\sigma}:\A/\J\rightarrow\mathcal{Q}((\mathfrak{H}_n^{-1/2})^+),\quad \wt{\sigma}(a)=\pi(P_{\geq}\rho_\de(a) P_{\geq}),$$ 
where by 2) $P_\geq$ extends to a bounded operator on $(\mathfrak{H}_n^{-1/2})^+$. 
Property 2) also implies that 
$\wt{\sigma}(a)=\pi(P_C^+\rho_\de(a)P_C^+)=\wt{\alpha}(a).$ 
Hence the restrictions of $\alpha$ and $\tau$ to $\A/\J$ are similar, which by density implies that $\alpha$ and $\tau$ are similar. 
\end{proof}

\begin{rem}
In \cite{BDT}, the analogue of Proposition \ref{prop:boundarycomparisonnew} is proven for 
even-dimensional manifolds with boundaries, where Assumption \ref{sixthass} 
is a well-known statement.  The approach in \cite{BDT} to odd-dimensional manifolds uses a suspension 
by $S^1$ to reduce to the even-dimensional case. In a similar vein, we can suspend relative spectral 
triples with Clifford normals. If $(\J\lhd \A,\h,\D,n)$ is a relative spectral triple with Clifford normal we define 
$$\Sigma(\J\lhd \A,\h,\D,n):=(C^\infty(S^1,\J)\lhd C^\infty(S^1,\A),L^2(S^1,\hat{\h}),\hat{\D},\hat{n}),$$
where smoothness in $C^\infty(S^1,\A)$ and $C^\infty(S^1,\J)$ is defined in the Lipschitz norm topology, $\hat{\h}=\h\oplus\h$ 
as graded by $1\oplus (-1)$ in the odd case, $\hat{\h}=\h$ with a trivial grading in the even case, and
\begin{align*}
\hat{\D}&:=
\begin{pmatrix}
0& -\partial_\theta+\D\\
\de_\theta+\D&0
\end{pmatrix}, 
\quad\quad 
\hat{n}:=
\begin{pmatrix}
0& n\\
n&0
\end{pmatrix} \quad\mbox{in the odd case, and}\\
\hat{\D}&:=\D+i\gamma\de_\theta
\quad\quad 
\hat{n}:=n
 \quad\mbox{in the even case.}
\end{align*}
Here $\partial_\theta$ denotes differentiation in the angular variable on $S^1$. 

A lengthier exercise with matrices, using \cite[Theorem 1.3]{DGM}, shows that the suspension functor
$(\J\lhd \A,\h,\D,n)\mapsto \Sigma(\J\lhd \A,\h,\D,n)$ preserves Assumption \ref{sixthass}. 
It is unclear to the authors if the assumptions appearing in Theorem \ref{thm:boundarycomparisonnew} 
or Proposition \ref{prop:boundarycomparisonnew} are preserved under suspension; the Sobolev spaces are not 
necessarily well behaved under suspension. However, the argument in \cite{BDT} shows that if $(\J\lhd\A,\h,\D,n)$ 
is an odd relative spectral triple with normal such that the even relative spectral triple with normal $\Sigma(\J\lhd\A,\h,\D,n)$ 
satisfies the conditions of Theorem \ref{thm:boundarycomparisonnew} 
(or Proposition \ref{prop:boundarycomparisonnew}), then 
$$\de[(\J\lhd\A,\h,\D)]=[(\A/\J,\mathfrak{H}_n,\D_n)]\in KK^{0}(A/J,\C).$$
Recall here that for $(\J\lhd\A,\h,\D,n)$ odd, $\mathfrak{H}_n$ is $\Z/2$-graded by 
$\gamma=-i n_\de$ as in Definition \ref{defn:gradingonboundary}.
\end{rem}

\section{Examples}
\label{sec:onexamples}

\subsection{Concluding remarks on the main examples}

\subsubsection{Manifolds with boundary}
\label{ass6andmfds}

All our assumptions are satisfied for a manifold with boundary $M$. The assumptions hold because of 
sophisticated machinery involving pseudodifferential calculus. Assumption \ref{sixthass} 
is immediate from the fact that $(C^\infty(\partial M)\wh{\otimes}\Cl_1,L^2(\partial M,S|_{\partial M}),\D_n)$ is a spectral triple. 
The assumptions in Proposition \ref{prop:boundarycomparisonnew}, and therefore also Theorem \ref{thm:boundarycomparisonnew}, 
follows from symbol arguments: $P_{\geq}$ and $P_C$ are pseudo-differential operators of order 
zero with the same principal symbol. This reproves the well-known result in $K$-homology that 
the boundary mapping applied to a Dirac operators correspond to the geometric boundary, see \cite[Proposition 5.1]{BDT}.

\subsubsection{Revisiting the boundary class of a conical manifold}

For a conical manifold $M$, as in the Subsubsections \ref{eg:cones-on-the-brain} and \ref{cliffordoncon}, 
we know that the boundary class $\de[(\J\lhd \A,L^2(M,S),\D_L)]=0\in K^*(\C^l)$ where $\C^l=\A/\J$ is the boundary 
algebra (see Remark \ref{khomcomp}). We will revisit this result in light of Subsection \ref{representingsubsec}. 

An elementary computation with the eigenfunctions on the cross-section $N$ 
gives us 
$$
\Dom(\D_{L^\perp}^2)/\Dom(\D_L^2)=\Dom(\D_{L^\perp})/\Dom(\D_L)=L^\perp/L.
$$ 
Since $\D_L^*f\in \Dom\D_L$ for any $f\in \Dom(\D_{L^\perp}^2)$, 
we find that
$\D_n=0$ so Assumption \ref{sixthass} 
and the assumptions of Proposition \ref{prop:boundarycomparisonnew} 
hold trivially (as they deal with finite-dimensional spaces). 
Clearly $\mathfrak{H}_n=\Dom(\D_L^*)/\Dom(D_L)=L^\perp/L$. If $M$ is even-dimensional, 
$K^{\dim(M)+1}(\C^l)=0$ and so  we find that $\de[(\J\lhd \A,L^2(M,S),\D_L)]=0$ holds trivially. 
If $M$ is odd-dimensional, the boundary class $\de[(\J\lhd \A,L^2(M,S),\D_L)]\in K^{\dim(M)+1}(\C^l)$ 
is represented by the graded $\C^l$-module $L^\perp/L$ graded by $iI$ (where $I$ is the complex 
structure defined by the normal, see Theorem \ref{conicalandnormal}). Since $I$ is a complex 
structure, the odd and even parts of $L^\perp/L$ 
have the same dimension, so $[L^\perp/L]=0$ in $K^0(\C^l)$.

\subsubsection{The boundary of a dimension drop algebra}
For a dimension drop algebra $A$ over a manifold with boundary $M$, with 
$A|_{\partial M}=C(\partial M,B)$, $B\subset M_N(\C)$, we see that we obtain 
the boundary class represented by
$$
(C^\infty(\partial M,B)\wh{\otimes}\Cl_1,L^2(S|_{\partial M}\otimes \C^N), \Dsla_{\partial M}\otimes 1_N)
$$
and $C(\partial M,B)$ is Morita equivalent to $C(\partial M,\C^n)$ for some $1\leq n\leq N$.

\subsubsection{Relative spectral triples and boundary mappings on $\Z/k$-manifolds}
The $\Z/k$-manifolds are manifolds with boundary whose boundary has a very specific structure. 
Let $M$ be a Riemannian manifold with boundary such that a collar neighborhood $U$ of the 
boundary $\de M$ is isometric to $(0,1)\times kN$ (disjoint union of $k$ copies of $N$) where 
$N$ is a closed Riemannian manifold. We identify $\{1\}\times kN$ with the boundary of $M$. 
This situation is well-studied in the literature, see 
\cite{HigsonZk} for an operator theoretic approach and \cites{RobinZk,MelroseFreedZk} 
for index theoretical results. The $C^*$-algebra $C(\ol{M})$ acts as (non-central) multipliers on 
the $C^*$-algebra $C_0((0,1]\times N,M_k(\C))$ via the diagonal embedding 
$$
C(\ol{U})=C([0,1]\times \partial M)\cong C([0,1]\times N,\C^k)\subseteq C([0,1]\times N,M_k(\C)).
$$ 
We consider the $C^*$-algebra
$$
A:=C(\ol{M})+C_0((0,1]\times N,M_k(\C)),
$$
which was constructed as a groupoid $C^*$-algebra in \cite{rosengroupo}, 
see also \cite{RobinZk}*{Section 2.2}. 
The manifold with boundary $M\setminus U$ is again a $\Z/k$-manifold, with boundary $kN$. Consider the 
dimension drop algebra $A_0:=\{f\in C([0,1]\times N,M_k(\C)): f(0)\in \C^k\}$. Despite that $A$ is not a 
dimension drop algebra in the sense of Subsection \ref{dimdropsubex}, 
we can write $A$ as a restricted sum $A=C(M\setminus U)\oplus_{C(N,\C^k)} A_0$.
We define
$$
\A:=C^\infty(\ol{M})+C_0^\infty((0,1\times N,M_k(\C)) \quad\mbox{and}\quad 
\J:=C^\infty_0(M^\circ)+C_0^\infty((0,1)\times N,M_k(\C)).
$$
Clearly, $\A\subseteq A$ is a holomorphically closed $*$-algebra, $\J\lhd \A$ and 
$\J=J\cap \A$ where $J:=C_0(M^\circ)+C_0((0,1)\times N,M_k(\C))$. We have that 
$$
\A/\J=C^\infty(N,M_k(\C))\subseteq C(N,M_k(\C))=A/J.
$$

The relevant Clifford bundles $S\to \ol{M}$ takes the form $S|_U\cong kS_N$ near the boundary, 
for a Clifford bundle $S_N\to (0,1)\times N$. Here $kS_N\to (0,1)\times kN$ is the disjoint union of 
$k$ copies of $S_N\to (0,1)\times N$. We consider a 
Dirac operator $\Dsla$ acting on $C^\infty_c(M^\circ,S)$ respecting 
the $\Z/k$-structure (e.g. when constructed from a 
$\Z/k$-Clifford connection on $S$) and let $\D_{\textnormal{min}}$ 
denote its minimal closure. The algebra $A$ acts on $L^2(M,S)$ via the decomposition 
$$L^2(M,S)=L^2(M\setminus U,S)\oplus L^2(N,S_N\otimes \C^k).$$
We deduce the following result from Proposition \ref{prop:boundarycomparisonnew} using the 
Clifford multiplication by a unit vector normal to the boundary 
that is the same on all components of the boundary.

\begin{prop}
If the Clifford bundle $S$ and $\Dsla$ respects the $\Z/k$-structure, $(\J\lhd \A,L^2(M,S),\D_{\textnormal{min}})$ is a 
relative spectral triple and 
$$\de[(\J\lhd \A,L^2(M,S),\D_{\textnormal{min}})]=[(C^\infty(N)\otimes \Cl_1,L^2(N,S_N),\Dsla_N)],$$
under the isomorphism $K^{\dim(M)+1}(C(N))\cong K^{\dim(M)+1}(C(N,M_k(\C)))$.

\end{prop}

\subsection{Relative spectral triples for crossed products}
\label{crossedproductapproach}

Another way to obtain relative spectral triples from group actions is from the crossed product 
associated with a group action on a manifold with boundary. We will first state some general 
results after which we restrict to compact groups acting on a compact manifold with boundary. 
Let $(M,g)$ denote a Riemannian manifold with boundary and assume that a Lie group $G$ acts 
isometrically on $M$. We tacitly assume all structures, including the group action, to be of product 
type near the boundary of $\ol{M}$. We define $\A$ as the $*$-subalgebra 
$C^\infty_c(\ol{M}\times G)+C_c^\infty(\ol{M})$ of the $C^*$-algebra 
$C_0(\ol{M})+C_0(\ol{M})\rtimes G\subseteq \mathcal{M}(C_0(M)\rtimes G)$.
We define $\J\lhd \A$ by $\J:=C^\infty_c(M^\circ\times G)+C_c^\infty(M^\circ)$. If $S\to \ol{M}$ is a 
$G$-equivariant vector bundle, the algebra $\A$ acts on $L^2(M,S)$ via the integrated representation 
$\pi:\A\to \B(L^2(M,S))$. More precisely, if $a=a_0+a_1$ where $a_0\in C_c^\infty(\ol{M})$ and 
$a_1\in C^\infty_c(\ol{M}\times G)$,
$$\pi(a)f(x):=a_0(x)f(x)+\int_G a_1(x,\gamma)[\gamma.f](x)\mathrm{d}\gamma, \quad f\in L^2(M,S).$$

\begin{prop}
\label{firstproponcrossed}
If $\Dsla$ is a Dirac operator acting on a $G$-equivariant Clifford bundle $S\to \ol{M}$ with minimal 
closure $\D$, the object $(\J\lhd\A,L^2(M,S),\D)$ satisfies the conditions for a relative spectral triple 
if $M$ is compact or once replacing Condition 4. of Definition \ref{defn:alternate} with 
Condition 4'. of Remark \ref{rem:alternatecondition}. If $n_{\de M}$ is a unit 
normal, $c(n_{\de M})$ induces a normal structure on $(\J\lhd\A,L^2(M,S),\D)$. If $\partial M$ is compact, 
$c(n_{\de M})$ induces a second order normal structure $n$ satisfying the conditions appearing in 
Proposition \ref{prop:equivalences} and 
$$\D_\de=\Dsla_{\de M},\quad\mbox{on the image of $C^\infty_c(\ol{M},S)\to 
\mathfrak{H}^{1/2}_n\subseteq L^2(\partial M,S|_{\partial M})$}.$$
\end{prop}

\begin{proof}
It is clear that $\D$ has locally compact resolvent, because $H^1_{c}(M,S)\to L^2(M,S)$ is 
compact. Clearly, for $a\in C^\infty_c(\ol{M})$, the commutator $[\Dsla,\pi(a)]$ is Clifford multiplication 
by $\mathrm{d}a$ and hence bounded. The Dirac operator $\Dsla$ is almost $G$-equivariant 
and for $a\in C^\infty_c(\ol{M}\times G)$ and $f\in C^\infty_c(M,S)$, 
\begin{align*}
[\Dsla,\pi(a)]f(x)&=\int_G \left(\Dsla(a_1(\cdot ,\gamma)[\gamma.f])(x)-
a_1(x ,\gamma)[\gamma.(\Dsla)f])(x)\right)\mathrm{d}\gamma=\\
&=\int_G c(\mathrm{d}a_1(x,\gamma))[\gamma.f](x)\mathrm{d}\gamma+
\int_G a_1(x,\gamma)[(\Dsla-\gamma\Dsla\gamma^{-1})\gamma.f](x)\mathrm{d}\gamma.
\end{align*}
Since $\Dsla-\gamma\Dsla\gamma^{-1}\in C^\infty(\ol{M},\End(S))$, it follows that $[\Dsla,\pi(a)]$ 
extends to a bounded operator on $L^2(M,S)$. 

If $n_{\de M}$ is a unit normal, we can extend $c(n_{\de M})$ to an odd section $n\in C^\infty_b(\ol{M},\End(S))$
as in Example \ref{classicalexampleonmfd}. The group action is of product type near the boundary, so it is clear 
from the construction that $\gamma n\gamma^{-1}-n\in C^\infty_c(M,\End(S))$. 
We deduce that $n$ is a Clifford normal for $(\J\lhd\A,L^2(M,S),\D)$. The conditions appearing in 
Proposition \ref{prop:equivalences} follows if $\de M$ is compact in the usual way.
\end{proof}

\begin{prop}
\label{prop:computewithcrosse}
If $M$, $n$ and $(\J\lhd\A,L^2(M,S),\D)$ are as in Proposition \ref{firstproponcrossed} with $\ol{M}$ compact,
the second order normal structure $n$ satisfies Assumption \ref{sixthass}. Moreover, 
\begin{equation}
\label{boundarycomp}
\de[(\J\lhd\A,L^2(M,S),\D)]=
[(C^\infty(\de M)+C^\infty_c(\de M\times G))\hat{\otimes} \Cl_1,L^2(\de M,S|_{\de M}),\D_{\de M}],
\end{equation}
in $K^{*+1}(C(\de M)+C(\de M)\rtimes G)$.
\end{prop}

\begin{proof}
Recall the computation in Equation \eqref{commwithdira} showing that 
$[\Dsla,n]=n\frac{\de n}{\de u}-2n\Dsla_{n}$ 
as an operator on $C^\infty_c(\ol{M},S)$. 
All assumptions but those in Proposition \ref{prop:boundarycomparisonnew} follow from this fact. 
To prove the assumptions of Proposition \ref{prop:boundarycomparisonnew}, 
we note that all of these statements, as in the classical case, follows from a computation at the level of 
principal symbols and these extend to a group action as the principal symbols 
are equivariant. The identity \eqref{boundarycomp} follows from 
Proposition \ref{prop:boundarycomparisonnew}.
\end{proof}

\begin{eg}
\label{eq:irrat-disc}
Let us consider a concrete example. 
We take a $\theta\in \mathbb{R}\setminus \mathbb{Q}$. 
The associated irrational rotation $R_\theta:S^1\to S^1$ is defined as multiplication 
by $\mathrm{e}^{i\theta}$ in the model 
$S^1\subseteq \C$. The diffeomorphism $R_\theta$ defines a free and
minimal action and $C(\mathbb{T}^2_\theta):=C(S^1)\rtimes _{R_\theta}\Z$ 
is a simple $C^*$-algebra isomorphic to the non-commutative 
torus (see next subsection \ref{sub:theta}). 
We extend $R_\theta$ to an isometry of the disc $\ol{D}:=\{z\in \C: |z|\leq 1\}$. 
We define $I:=C_0(D)\rtimes_{R_\theta}\Z$ and 
$A:=C(\ol{D})\rtimes_{R_\theta} \Z$, so $A/I=C(\mathbb{T}^2_\theta)$. 

Consider the Dirac operator $\Dsla$ on the disc 
$D$ with its euclidean metric. We identify the spinor bundle on 
$D$ with $D\times \C^2$. The operator $\Dsla$ commutes with 
the isometry $R_\theta$. Let $\D_{\rm min}$ 
denote the minimal extension of $\Dsla$. We can from 
Proposition \ref{prop:computewithcrosse} deduce that 
$(C^\infty_c(D)\rtimes^{\rm alg}_{R_\theta} \Z\lhd C^\infty(\ol{D})\rtimes^{\rm alg}_{R_\theta} \Z, L^2(D,\C^2),\D_{\rm min})$ 
is a relative spectral triple for $C_0(D)\rtimes_{R_\theta}\Z\lhd C(\ol{D})\rtimes_{R_\theta} \Z$. The Dirac operator 
$\Dsla_{S^1}=i\frac{\rm d}{{\rm d}x}$ commutes with $R_\theta$, so we obtain a spectral triple 
$(C^\infty(S^1)\rtimes^{\rm alg}_{R_\theta} \Z, L^2(S^1),\Dsla_{S^1})$ for $A/I=C(\mathbb{T}^2_\theta)$.
Proposition \ref{prop:computewithcrosse} shows that 
$$\partial[(C^\infty_c(D)\rtimes^{\rm alg}_{R_\theta} \Z\lhd C^\infty(\ol{D})\rtimes^{\rm alg}_{R_\theta} \Z, L^2(D,\C^2),\D_{\rm min})]=[(C^\infty(S^1)\rtimes^{\rm alg}_{R_\theta} \Z, L^2(S^1),\Dsla_{S^1})],\quad\mbox{in}\;K^1(C(\mathbb{T}^2_\theta)).$$
The class of the right hand side is non-zero because it restricts to the fundamental class of $S^1$ along 
the inclusion $C(S^1)\to C(S^1)\rtimes_{R_\theta}\Z=C(\mathbb{T}^2_\theta)$.
We remark at this point that the inclusion $I\hookrightarrow A$ is null-homotopic, so the exact triangle in $KK$ defined from 
the semi-split exact sequence $0\to I\to A\to C(\mathbb{T}^2_\theta)\to 0$ is isomorphic to the Pimsner-Voiculescu 
triangle defined from the crossed product realisation $C(S^1)\rtimes_{R_\theta}\Z=C(\mathbb{T}^2_\theta)$. 
\end{eg}

\subsection{$\theta$-deformations of relative spectral triples}
\label{sub:theta}

Relative spectral triples and Clifford normals behave well under $\theta$-deformations. 
We fix a twist $\theta\in \textnormal{Hom}(\Z^d\wedge \Z^d,U(1))$. Such a $\theta$ 
is defined from an anti-symmetric matrix $(\theta_{jk})_{j,k=1}^d\in M_d(\R)$. Associated with 
$\theta$, there is a deformation of $C(\mathbb{T}^d)$. We let $C(\mathbb{T}^d)_\theta$ denote 
the $C^*$-algebra generated by $d$ unitaries $U_1,U_2,\ldots, U_d$ subject to the relations 
$$U_jU_k=\mathrm{e}^{i\theta_{jk}}U_kU_j.$$
For $\pmb{k}=(k_1,k_2,\ldots, k_d)\in \Z^d$ we write $U^{\pmb{k}}:=U^{k_1}U^{k_2}\cdots U^{k_d}$.
Let $\mathcal{S}(\Z^d)$ denote the Schwarz space consisting of rapidly decaying sequences.
There is a continuous injection $\mathcal{S}(\Z^d)\to C(\mathbb{T}^d)_\theta$ mapping a 
sequence $(a_{\pmb{k}})_{\pmb{k}\in \Z^d}$ to the element 
$\sum_{\pmb{k}\in \Z^d}a_{\pmb{k}}U^{\pmb{k}}$. We denote the image by 
$C^\infty(\mathbb{T}^d)_\theta$. This is a nuclear Fr\'echet algebra in the topology induced 
from $\mathcal{S}(\Z^d)$. 

If $\mathcal{A}$ is a Fr\'echet algebra with a continuous action of $\mathbb{T}^d$, we define its 
$\theta$-deformation as the invariant subalgebra 
$\A_\theta:=(\A\otimes^{\rm alg} C^\infty(\mathbb{T}^d)_\theta)^{\mathbb{T}^d}$.
Since $\mathbb{T}^d$ is an abelian group, $\A_\theta$ also carries an action of $\mathbb{T}^d$. 
If $A$ is a $\mathbb{T}^d-C^*$-algebra, we define 
$A_\theta:=(A\otimes_{\textnormal{min}} C(\mathbb{T}^d)_\theta)^{\mathbb{T}^d}$.
If $\mathcal{A}\subseteq A$ is a $\mathbb{T}^d$-invariant dense Fr\'echet subalgebra of a $C^*$-algebra $A$, 
$\mathcal{A}_\theta\subseteq A_\theta$ is a  $\mathbb{T}^d$-invariant dense Fr\'echet subalgebra.
The tracial state $\tau_\theta:C(\mathbb{T}^d)_\theta\to \C$, 
$\tau_\theta(\sum_{\pmb{k}\in \Z^d}a_{\pmb{k}}U^{\pmb{k}}):=a_0$ gives 
rise to a GNS-representation $L^2(\mathbb{T}^d)_\theta$ of $C(\mathbb{T}^d)_\theta$. 
For a $\mathbb{T}^d$-representation $\h$ we define its $\theta$-deformation 
$\h_\theta:=(\h\otimes L^2(\mathbb{T}^d)_\theta)^{\mathbb{T}^d}$. 
If a $\mathbb{T}^d-C^*$-algebra $A$ acts equivariantly on $\h$, $A_\theta$ acts equivariantly in $\h_\theta$.
A $\mathbb{T}^d$-equivariant closed operator $\D$ on $\h$, induces a closed operator $\D_\theta$ on $\h_\theta$ 
via restriction of $\D\otimes1_{L^2(\mathbb{T}^d)_\theta}$. It is easily proven that $(\D_\theta)^*=(\D^*)_\theta$ 
using the unitary 
$$U_\theta:\h\to \h_\theta,\quad 
U_\theta:\xi=\sum_{\pmb{k}\in \Z^d}\xi_{\pmb{k}}\mapsto\sum_{\pmb{k}\in \Z^d}\xi_{\pmb{k}}\otimes U^{-\pmb{k}},$$ 
where $\xi=\sum_{\pmb{k}\in \Z^d}\xi_{\pmb{k}}$ denotes the decomposition 
into the homogeneous components for the $\mathbb{T}^d$-action. 
The following proposition follows from the construction.

\begin{prop}
Assume that $(\J\lhd\A,\h,\D)$ is a $\mathbb{T}^d$-equivariant relative spectral triple. 
Then $(\J_\theta\lhd\A_\theta,\h_\theta,\D_\theta)$ is a $\mathbb{T}^d$-equivariant relative spectral triple.
If $n$ is a $\mathbb{T}^d$-invariant Clifford normal for $(\J\lhd\A,\h,\D)$, $n_\theta:=U_\theta nU_\theta^*$ is a Clifford normal for 
$(\J_\theta\lhd\A_\theta,\h_\theta,\D_\theta)$ and Assumption \ref{sixthass} and the hypotheses of Proposition \ref{prop:boundarycomparisonnew} are equivalent for $n$ and $n_\theta$.
\end{prop}

\subsection{Cuntz-Pimsner algebras over manifolds with boundaries}
\label{cpmfdbry}


%
We start with the continuous functions $A=C(\ol{M})$ on a manifold with boundary 
$\ol{M}$ and the module of sections $E=\Gamma(\ol{M},V)$ of 
a complex hermitian vector bundle $V\to M$. 
We equip $E$ with the usual symmetric bimodule structure. One could
instead implement the left multiplication by functions by first composing with
an isometric diffeomorphism, as in \cite{GMR}. This would produce 
a blend between Example \ref{eq:irrat-disc}
and the situation described below. As this is notationally
more intricate, we will not discuss it here.

We form the Cuntz-Pimsner algebra $O_E$ of $E$ viewed as a 
$C^*$-correspondence over $A=C(\ol{M})$. This example is described in detail
in \cite{GMR} for closed manifolds.
The Cuntz-Pimsner algebra $O_E$ 
takes the form $O_E=\Gamma(\ol{M},O_V)$ 
where $O_V\to \ol{M}$ is a locally 
trivial bundle of Cuntz-algebras $O_N$ and 
$N$ is the rank of $V$. 
It is proved in \cite{GMR}*{Theorem 1} that there is an unbounded Kasparov module
$(O_E,\Xi_{\ol{M}},\D_{CP})$  
whose class in $KK_1(O_E,A)$ represents the Pimsner extension, \cite{RRS}, 
$$
0\to \K_A(F_E)\to \mathcal{T}_E\to O_E\to 0.
$$
Here $\K_A(F_E)$ is the $A$-compact endomorphisms on the Fock module of $E$.
Moreover, there are hermitian 
vector bundles $\Xi_{m,k}^V\to \ol{M}$ such that 
$\Xi_{\ol{M}}$ 
decomposes as an orthogonal direct sum 
$$
\Xi_{\ol{M}}=\bigoplus_{m,k\in \Z}\Xi_{m,k}
:=\bigoplus_{m,k\in \Z}\Gamma(\ol{M},\Xi_{m,k}^V).
$$ 
The modules $\Xi_{m,k}$ satisfy 
$\Xi_{m,k}=0$ if $k<0$ or $m+k<0$. Let $p_{m,k}$ 
denote the projections onto $\Xi_{m,k}$. 
We define a function $\psi:\Z\times \mathbb{N}\to \Z$ 
by $\psi(m,k):=m$ if $k=0$ and 
$\psi(m,k):=-|m|-k$ if $k>0$. The 
self-adjoint operator $\D_{CP}$ is the closure of 
$\sum_{m,k}\psi(m,k)p_{m,k}$. 

We use the notation 
$s_\xi\in O_E$ for the operator 
associated with a section 
$\xi\in \Gamma(\ol{M},E)$. If $S\to \ol{M}$ is a Clifford module, and 
$\xi\in \Gamma(\ol{M},E\otimes T^*M)$ 
we let $c(s_\xi)\in O_E\otimes_{C(\ol{M})} \Gamma(\ol{M},\End(S))$ 
denote the associated operator. 
The following construction follows by using the local trivializations 
of $O_E$ and $\Xi_{\ol{M}}$.

\begin{prop}
\label{connprop}
Let $S$ be a Clifford bundle on $\ol{M}$.
A choice of frame for $V$ defines Gra\ss mann connections 
$\nabla_{m,k}:\Gamma^\infty(\ol{M},\Xi_{m,k}^V)\to 
\Gamma^\infty(\ol{M},\Xi_{m,k}^V\otimes \End(S))$ 
such that 
$$
\nabla_\Xi:=\bigoplus_{m,k} \nabla_{m,k}:
\bigoplus ^{\textnormal{alg}}_{m,k}\Gamma^\infty(\ol{M},\Xi_{m,k}^V)\to 
\bigoplus ^{\textnormal{alg}}_{m,k}\Gamma^\infty(\ol{M},\Xi_{m,k}^V\otimes \End(S))
\subseteq\Xi_{\ol{M}}\otimes_{C(\ol{M})} \Gamma(\ol{M},\End(S)),
$$
satisfies 
$\nabla_\Xi(s_\xi\eta)-s_\xi \nabla_\xi(\eta)=c(s_{\mathrm{d}\xi})\cdot \eta\otimes 1_S$ 
for $\xi\in \Gamma^\infty(\ol{M},V)$ 
and $\eta\in \bigoplus ^{\textnormal{alg}}_{m,k}\Gamma^\infty(\ol{M},\Xi_{m,k}^V)$.
\end{prop}

Assume that $\Dsla_M$ is a Dirac operator on 
a Clifford bundle $S\to \ol{M}$ and that $\nabla_\Xi$ 
is a connection as in Proposition \ref{connprop}. 
Following \cite{GMR}, we define an operator $\D_E$ on a Hilbert space $\h_E$. 
For $S$ graded and $\Dsla_M$ odd, 
we define  $\h_E:=\Xi_{\ol{M}}\otimes_{C(\ol{M})}L^2(M,S)$ 
and the operator $\D_E$ as 
the closure of 
\begin{equation}
\label{deconstr}
\Dsla_E:=\D_{CP}\otimes \gamma+1\otimes_{\nabla_\Xi}\Dsla_M,
\end{equation}
defined on $\bigoplus ^{\textnormal{alg}}_{m,k}\Gamma^\infty_c(M^\circ,\Xi_{m,k}^V\otimes S)$. 
Here $\gamma$ denotes the grading operator on $S$. 
If $S$ is ungraded (e.g. if $M$ is odd-dimensional), we define 
the graded Hilbert space 
$\h_E:=\Xi_{\ol{M}}\otimes_{C(\ol{M})}(L^2(M,S)\oplus L^2(M,S))$ 
and the odd operator $\D_E$ as the closure of 
the differential expression
\begin{equation}
\label{deconstr2}
\D_E:=\begin{pmatrix}
0& \D_{CP}\otimes 1-i\otimes_{\nabla_\Xi}\Dsla_M\\
\D_{CP}\otimes 1+i\otimes_{\nabla_\Xi}\Dsla_M& 0 \end{pmatrix},
\end{equation}
defined on 
$\bigoplus ^{\textnormal{alg}}_{m,k}
\Gamma^\infty_c(M^\circ,\Xi_{m,k}^V\otimes (S\oplus S))$. 
In the following we will assume that $S$ is graded, by replacing $S$ by $S\oplus S$
if necessary.
 
We define $\mathcal{O}_E\subseteq O_E$ 
as the dense $*$-subalgebra generated 
by those $s_\xi$ with $\xi\in \Gamma^1(\ol{M},V)$. We also define the ideal 
$\J:=C^1_0(M^\circ)\mathcal{O}_E\lhd \mathcal{O}_E$. 
The closure $J$ of $\J$ coincides with 
$C_0(M^\circ)O_E$. The ideal $J$ fits into the 
$C(\overline{M})$-linear short exact sequence 
of $C(\overline{M})$-$C^*$-algebras
\begin{equation}
\label{sesofcp}
0\to C_0(M^\circ)O_E\to O_E\to O_{E_\partial}\to 0,
\end{equation}
where $O_{E_\partial}$ is constructed from the 
$C(\partial M)$-bimodule $E_\de:=\Gamma(\partial M,V|_{\de M})$. 
We define the dense $*$-subalgebra 
$\mathcal{O}_\de\subseteq O_{E_\partial}$ as the $*$-subalgebra 
generated by those $s_\eta$ with $\eta\in \Gamma^1(\de M,V|_{\de M})$. 
We can define the graded Hilbert space
$\mathfrak{H}_{E} 
:=\Xi_{M_\de}\otimes_{C(\de M)}L^2(\de M,S|_{\de M})$ 
and an operator $\D_{E,\de M}$ constructed 
as in Equation \eqref{deconstr2} from 
$E_\de$ and the boundary operator $\Dsla_{\de M}$. We also write 
$\mathfrak{H}^+_{E}
:=\Xi_{M_\de}\otimes_{C(\de M)}L^2(\de M,S^+|_{\de M})$.

\begin{thm}
\label{theoremforcpbdry}
We consider the Cuntz-Pimsner algebra 
$O_E$ associated with a hermitian vector bundle $V\to \ol{M}$ 
over a compact $d$-dimensional manifold with boundary. 
In the notations of the two previous paragraphs, 
$(\J\lhd\mathcal{O}_E,\h_E,\D_E)$ 
is a relative spectral triple of parity $d+1$ 
and Clifford multiplication by any 
unit normal to $\partial M$ defines a Clifford normal $n$ for 
$(\J\lhd\mathcal{O}_E,\h_E,\D_E)$. 

Moreover, $(\J\lhd\mathcal{O}_E,\h_E,\D_E,n)$ 
satisfies Assumption \ref{sixthass} with $(\D_E)_n=\D_{E,\de M}$, 
$\mathfrak{H}_{n}\cong\mathfrak{H}_{E}$.
\end{thm}

\begin{proof}
We first verify that $(\J\lhd\mathcal{O}_E,\h_E,\D_E)$ 
is a relative spectral triple. We can double the manifold 
$M$ obtaining a closed manifold $2M$ 
and a vector bundle $2V\to 2M$. Set $2E:=\Gamma(2M,2V)$. 
The restriction induces a surjection $\Gamma^1(2M,2V)\to \Gamma^1(M,V)$. 
Let $\mathcal{O}_{2E}\subseteq O_{2E}$ 
denote the $*$-subalgebra generated by $\Gamma^1(2M,2V)$ 
and $\D_{2E}$ the operator constructed 
as in Equation \eqref{deconstr2} on 
$\h_{2E}:=\Xi_{2E}\otimes_{C(2M)}L^2(2M,2S)$. The collection 
$(\mathcal{O}_{2E},\h_{2E},\D_{2E})$ is a spectral triple by \cite{GMR}. 

In particular, any element of $\mathcal{O}_{2E}$ acts 
continuously on $\Dom(\D_{2E})$ in its graph topology. 
We have that $\Dom(\D_E)$ is the closure of 
$\bigoplus_{m,k}\Gamma^1_c(M^\circ, S\otimes \Xi^V_{m,k})$ 
in the graph topology of $\Dom(\D_{2E})$. Any 
$\xi\in \Gamma^1(M,V)$ lifts to a $\tilde{\xi}\in \Gamma^1(2M,2V)$ 
and if $\sigma\in \Dom(\D_E)$ is a limit of 
$(\sigma_j)_{j\in \mathbb{N}}\in 
\bigoplus_{m,k}\Gamma^1_c(M^\circ, S\otimes \Xi^V_{m,k})$ 
in the graph topology, then $s_\xi \sigma$ is the limit of 
$(s_{\tilde{\xi}}\sigma_j)_{j\in \mathbb{N}}
=(s_\xi \sigma_j)_{j\in \mathbb{N}}\in 
\bigoplus_{m,k}\Gamma^1_c(M^\circ, S\otimes \Xi^V_{m,k})$. 
It follows that $\{s_\xi: \xi\in \Gamma^1(M,V)\}$ 
preserves $\Dom(\D_E)$ and 
$[\D_E,s_\xi]= [\D_{2E},s_{\tilde{\xi}}]$ thereon. 
We conclude that $\mathcal{O}_E$ preserves $\Dom(\D_E)$ and has bounded commutators with $\D_E$.


We decompose $\h_E=\bigoplus_{m,k} \h_{m,k}$, 
where $\h_{m,k}:=L^2(M,S\otimes \Xi^V_{m,k})$, and 
$\D_E=\overline{\bigoplus_{m,k} (\psi(m,k)\gamma+\D_{m,k})}$, 
where $\D_{m,k}$ is the minimal closure of 
$1\otimes_{\nabla_{m,k}}\Dsla_M$ and $\gamma$ 
denotes the grading on $S$. Therefore, 
$$
(1+\D_E^*\D_E)^{-1}=\bigoplus_{m,k} (1+\psi(m,k)^2+\D_{m,k}^*\D_{m,k})^{-1}.
$$
Each summand 
$
(1+\psi(m,k)^2+\D_{n,k}^*\D_{m,k})^{-1}:L^2(M,S\otimes \Xi_{m,k}^V)
\to L^2(M,S\otimes \Xi_{m,k}^V)
$ 
is compact since it factors through 
$H^2(M,S\otimes \Xi_{m,k}^V)\cap H^1_0(M,S\otimes \Xi_{m,k}^V)$. Moreover, 
$$
\| (1+\psi(m,k)^2+\D_{m,k}^*\D_{m,k})^{-1}\|_{\B(L^2(M,S\otimes \Xi_{m,k}^V))}
\leq 
(1+\psi(m,k)^2)^{-1}\to 0, \quad\mbox{as $m,k\to \infty$}.
$$
It follows that $(1+\D_E^*\D_E)^{-1}$ is compact.
We can write 
$$
\Dom(\D_E):=\{(f_{m,k})_{m,k}\in \bigoplus_{m,k} H^1_0(M,S\otimes \Xi_{m,k}^V): \; 
\sum_{m,k}\|\psi(m,k)f_{m,k}\|_{\h_{m,k}}^2+\|\D_{m,k}f_{m,k}\|_{\h_{m,k}}^2<\infty\}.
$$
Similarly, we can write 
$$
\Dom(\D_E^*):=\{(f_{m,k})_{m,k}\in 
\bigoplus_{m,k} \Dom\left((\D_{m,k}+\psi(m,k)\gamma_S)^*\right): 
\; \sum_{m,k}\|(\psi(m,k)\gamma_S+\D_{m,k})^*f_{m,k}\|_{\h_{m,k}}^2\!\!<\infty\}.
$$
To prove that $j\Dom(\D_E^*)\subseteq \Dom(\D_E)$ for all $j\in \J$, it suffices to 
prove $j\Dom(\D_E^*)\subseteq \Dom(\D_E)$ for all $j\in C^\infty_c(M)$. Fix $j\in C^\infty_c(M)$. 
For $f=(f_{m,k})_{m,k}\in \Dom(\D_E^*)$ we have that 
$(jf_{m,k})_{m,k}\in \bigoplus_{m,k} H^1_0(M,S\otimes \Xi_{m,k}^V)$. 
Therefore, 
\begin{align*}
\sum_{m,k}\|\psi(m,k)jf_{m,k}\|_{\h_{m,k}}^2+&\|\D_{m,k}jf_{m,k}\|_{\h_{m,k}}^2=
\sum_{m,k}\|(\psi(m,k)\gamma_S+\D_{m,k})^*jf_{m,k}\|_{\h_{m,k}}^2\leq\\
&\leq \|j\|_{\Lip}^2
\left(\|f\|_\h^2+\sum_{m,k}\|(\psi(m,k)\gamma_S+\D_{m,k})^*f_{m,k}\|_{\h_{m,k}}^2\right)
<\infty.
\end{align*}
We conclude that $j\Dom(\D_E^*)\subseteq \Dom(\D_E)$. 
In summary, $(\J\lhd\mathcal{O}_E,\h_E,\D_E)$ 
is a relative spectral triple. 
The symplectic form from Lemma \ref{nondegeprop} takes the form 
$$
\omega_{\D_E^*}(\xi,\eta)=
\sum_{m,k}\int_{\de M}
(\xi_{m,k}|n\eta_{m,k})_{S\otimes \Xi_{m,k}^V}\,\vol_{\de M}, 
\quad\xi,\eta\in \Dom(\D^*_E)/\Dom(\D_E).
$$
The fact that any 
unit normal  to the boundary induces a Clifford normal $n$
in the sense of Definition \ref{defn:normal} 
follows as in the case of a manifold with boundary 
in Example \ref{classicalexampleonmfd}.
The inner product $\Ideal{\cdot,\cdot}_{n}$ 
associated with the Clifford normal $n$ on 
$$
\bigoplus_{m,k} \Gamma^\infty(\de M, S|_{\de M}\otimes  \Xi^V_{m,k}|_{\de M})
\subseteq \mathfrak{H}^{1/2}_n
$$ 
coincides with the inner product on 
$\mathfrak{H}_{E}$. By the density of 
$\bigoplus_{m,k} \Gamma^\infty(\de M, S|_{\de M}\otimes \Xi^V_{m,k}|_{\de M})$ 
inside $\mathfrak{H}_{n}$ 
and $\mathfrak{H}_{E}$, we conclude that
$\mathfrak{H}_{n}=\mathfrak{H}_{E}$. The computation that 
$(\D_E)_n=\D_{E,\de M}$ follows from the classical case 
on each summand $\Gamma^\infty(\de M, S|_{\de M}\otimes \Xi^V_{m,k}|_{\de M})$.
\end{proof}

\begin{question}
Assume that $d$ is odd. 
Is it true that the geometric and $K$-homological boundaries 
coincide in our example, i.e. is it true that 
\begin{equation}
\label{computingboundary}
\de[(\J\lhd\mathcal{O}_E,\h_E,\D_E)]
=[(\mathcal{O}_\de\otimes \Cl_1,\mathfrak{H}_{E},\D_{E,\de M})]
\in K^{d}(O_{E_\partial}\otimes \Cl_1)?
\end{equation}
We reduce our graded boundary spectral triple to an odd spectral triple
as in Lemma \ref{lem:oddbdrytriple}, so that the ungraded Dirac operator takes the form
$\D_{E,\de}
=\D_{CP}\otimes (-in)+1\otimes_{\nabla_{\Xi|_{\de M}}}\Dsla_{\de M}$. 
It is clear that condition 1.\ in Proposition \ref{prop:boundarycomparisonnew} holds. 
Thus, the identity \eqref{computingboundary} 
would follow once $P_C^+-P_{\geq}$ is compact on 
$(\mathfrak{H}^{-1/2}_{n})^+$.
We write 
$\D_{\de M, m,k}=1\otimes_{\nabla_{m,k}|_{\de M}}\Dsla_{\de M}$.
The operators $P_C^+$ and $P_{\geq}$ both decompose over 
$m$ and $k$ and to check 
condition 2.\ in Proposition \ref{prop:boundarycomparisonnew} 
one needs to show that 
as $m,k\to \infty$, 
$$
\left\|(1+\D_{E,\de}^2)^{-1/4}(p_{m,k}\otimes 1)\left(P^+_{C,m,k}-\chi_{[0,\infty)}\left( \psi(m,k)(-in)+\D_{\de M,m,k}\right)\right)(1+\D_{E,\de}^2)^{1/4}\right\|_{L^2\to L^2}\!\!\to 0,
$$
where $P^+_{C,m,k}$ is the Calderon projector for 
$$
\begin{pmatrix} 0 & \psi(m,k)-i\D_{m,k}\\ \psi(m,k)+i\D_{m,k} & 0\end{pmatrix}.
$$
\end{question}

\appendix

\section{Relative $KK$-theory}
\label{subsec:Khom}

We begin by summarising the basic definitions of relative $K$-homology and $KK$-theory, 
generalising the presentation in \cites{BaumDouglas,HigsonRoe} to the bivariant setting.

\begin{defn}
\label{relativefrehdolm}
Let $A$ and $B$ be $\Z/2$-graded $C^*$-algebras, with $A$ separable and $B$ $\sigma$-unital, and let $J\triangleleft A$ denote a graded $G$-invariant closed $*$-ideal.
\begin{itemize}
\item A $G$-equivariant \textbf{even relative Kasparov module} 
$(\rho, X_B,F)$ for $(J\lhd A,B)$ consists of a $\Z/2$-graded $G$-equivariant
representation $\rho:A\rightarrow \End^*_B(X_B)$ on a countably 
generated $\Z/2$-graded $G$-equivariant $B$-Hilbert $C^*$-module $X_B$, 
and an odd operator $F\in \End^*_B(X_B)$ such that 
$\rho(a)(gFg^{-1}-F)$, $[F,\rho(a)]_\pm$, $\rho(j)(F-F^*)$ and $\rho(j)(1-F^2)$ 
are $B$-compact for all $a\in A$, $j\in J$ and $g\in G$. 
\item Assume that $A$ and $B$ are trivially $\Z/2$-graded. A $G$-equivariant 
\textbf{odd relative Kasparov module} for $(J\lhd A,B)$, 
is a triple $(\rho, X_B,F)$ consisting of a
$G$-equivariant representation $\rho:A\rightarrow \End^*_B(X_B)$ on a countably  
generated $G$-equivariant $B$-Hilbert $C^*$-module $X_B$, and an operator 
$F\in \End^*_B(X_B)$ such that 
$gFg^{-1}-F$, $[F,\rho(a)]$, $\rho(j)(F-F^*)$ and $\rho(j)(1-F^2)$ 
are $B$-compact for all $a\in A$, $j\in J$ and $g\in G$. 
\item A relative Kasparov module is \textbf{degenerate} if 
$$
gFg^{-1}-F=[F,\rho(a)]_\pm=\rho(j)(F-F^*)=\rho(j)(1-F^2)=0,\quad\text{ for all }a\in A,\,j\in J, \,g\in G.
$$
\item Whenever $F:[0,1]\to \End^*_B(X_B)$ is a norm-continuous map such that 
$(\rho, X_B,F(t))$ is a 
$G$-equivariant even/odd relative Kasparov module for all $t\in [0,1]$, we say that 
$(\rho, X_B,F(0))$ and  $(\rho, X_B,F(1))$ 
are operator homotopic.
\end{itemize}
We will often omit the notation $\rho$, in which case we write $(X_B,F)$ for a relative Kasparov module.
\end{defn}

Following \cite{BaumDouglas}, if $B=\C$ we also use the terminology relative Fredholm module instead of 
relative Kasparov module. The condition that $\rho(a)(gFg^{-1}-F)$ is $B$-compact for all $a\in A$ and 
$g\in G$ guarantees that a locally compact (for the $A$-action) perturbation of $F$ is $G$-equivariant.

We note that a $G$-equivariant even/odd \textbf{Kasparov module} for $(A,B)$ is nothing but
a $G$-equivariant even/odd relative Kasparov module for $(A\lhd A,B)$.
If $(\rho, X_B,F)$ is a $G$-equivariant relative Kasparov module for $(J\triangleleft A,B)$, it 
is also a $G$-equivariant Kasparov module for $(J,B)$. The converse need not hold.

\begin{rem}
\label{functorrembounded}
Relative Kasparov modules depend contravariantly on $J\triangleleft A$.
That is, if $\varphi:A_1\to A_2$ is an equivariant $*$-homomorphism, $J_2\triangleleft A_2$ a
graded closed $G$-invariant $*$-ideal containing $\varphi(J_1)$ and $(\rho, X_B,F)$ is a 
$G$-equivariant relative Kasparov module for $(J_2\triangleleft A_2,B)$ then 
$\varphi^*(\rho, X_B,F):=(\rho\circ \varphi,X_B,F)$
is a $G$-equivariant relative Kasparov module for $(J_1\triangleleft A_1,B)$. 
In particular, if $I\triangleleft A$ is an ideal with $I\triangleleft J$, any $G$-equivariant 
relative Kasparov module for $(J\triangleleft A,B)$ 
is also a $G$-equivariant relative Kasparov module for $(I\triangleleft A,B)$.
\end{rem}

\begin{defn}
We define the \textbf{relative $KK$-group} 
$KK^0_G(J\lhd A,B)$ to be equivalence classes of $G$-equivariant even 
relative Kasparov modules for $(J\triangleleft A,B)$ under the equivalence relation 
generated by unitary equivalence, 
operator homotopy and the addition of degenerate 
Kasparov modules \cite{KasparovKK,KasparovNovikov}. We set 
$$KK^1_G(J\lhd A,B):=KK^0_G(J\wh{\otimes}\Cl_1\lhd A\wh{\otimes}\Cl_1,B).$$ 
\end{defn}

Here $\Cl_1$ denotes the complex Clifford algebra on one generator, 
and all tensor products are graded. If $G$ is the trivial group, we write $KK^*(J\triangleleft A,B)$ 
instead of $KK^*_G(J\triangleleft A,B)$. The relative $K$-homology groups are defined by
$$
K^*_G(J\triangleleft A):=KK^*_G(J\triangleleft A,\C).
$$
The relative $KK$-groups are indeed groups as a rotation trick shows that 
$-[(\rho, X_B,F)]=[(\wt{\rho},X_B^{\textnormal{op}},-F)]$ 
where $X_B^{\textnormal{op}}:=X_B$ with the opposite grading and
$\wt{\rho}(a^++a^-)=\rho(a^+)-\rho(a^-)$ where $a^+$ is even and $a^-$ is odd.
We note that for trivially graded $A$ and $B$, 
we can also define $KK^1_G(J\lhd A,B)$ to be the set of equivalence classes of 
$G$-equivariant odd relative Kasparov modules 
under the equivalence relation generated by unitary equivalence, 
operator homotopy and the addition of degenerate Kasparov modules \cite{KasparovKK,KasparovNovikov}.

Restricting a $G$-equivariant relative Kasparov module for $J\triangleleft A$ 
along the inclusion $J\triangleleft A$ gives a $G$-equivariant Kasparov module for $J$. 
Hence there is a natural map $KK^*_G(J\lhd A,B)\rightarrow KK^*_G(J,B)$. 
\begin{prop}[Excision]
If $A\to A/J$ is semisplit, e.g. if $A/J$ is nuclear, 
the natural map $KK^*_G(J\lhd A,B)\rightarrow KK^*_G(J,B)$
is an isomorphism.
\end{prop}
The approach of Baum-Douglas to the excision theorem 
\cite{BaumDouglas}*{Theorem 14.24} generalizes to the bivariant setting. 
The proof in \cite{HigsonRoe}*{Theorem 5.4.5} in the case $B=\C$ is 
more direct, but not always applicable. 

For $A$ trivially graded and $A\to A/J$ semisplit, we can use excision to write the six-term exact 
sequence associated with the short exact sequence $0\to J\to A\to A/J\to 0$ as
$$
\xymatrix{KK^0_G(A/J,B)\ar[r]^-{\pi^*}&KK^0_G(A,B)\ar[r]^-{\iota^*}&
KK^0_G(J\lhd A,B)\ar[d]^-\de\\KK_G^1(J\lhd A,B)\ar[u]^-\de&
KK_G^1(A,B)\ar[l]^-{\iota^*}&KK^1_G(A/J,B)\ar[l]^-{\pi^*}}.
$$
The advantage of using the relative $KK$-theory 
$KK^*_G(J\lhd A,B)$ instead of $KK_*(J,B)$ is that the 
boundary map is more ``computable''. 
In \cite{Higson}, Higson writes down Fredholm 
modules representing $\de[(\rho,\h,F)]$ for a relative 
Fredholm module $(\rho,\h,F)$ (see also \cite{HigsonRoe}*{Proposition 8.5.6}). 
Let  $\wt{A}$ denote the unitisation of $A$ in the case 
that $A$ is not unital, and $\wt{A}=A$ when $A$ is unital. 
Higson's expression for the boundary map is 
not entirely constructive as it assumes knowledge of the 
completely positive splitting (see \cite{HigsonRoe}*{Definition 5.3.6}) 
$\sigma:\wt{A/J}\rightarrow\wt{A}$ and that it is already in dilation form 
(see \cite{HigsonRoe}*{Definition 8.5.5}). The precise Stinespring dilation
that splits $\wt{A}\to \wt{A/J}$ can be very difficult to construct in examples.

In the case that $A$ and $B$ are trivially $\Z/2$-graded, the even-to-odd boundary map 
$\de:KK^0_G(J\lhd A,B)\rightarrow KK^1_G(A/J,B)$ 
can also be described using extensions. The following can be 
found in \cite{HigsonRoe}*{p. 39 ff}, \cite{KasparovKK}*{\S7}. The equivariant theory of extensions 
is well presented in \cite{thomsenequi}, where the case of compact groups and 
proper actions is considered in \cite{thomsenequi}*{Section 9}.
In \cite{KasparovKK}, Kasparov defines the extension groups 
in greater generality than we do here.  
When discussing extensions, we restrict to trivially $\Z/2$-graded $C^*$-algebras.

Let $\K_G$ denote the $G$-$C^*$-algebra of compact operators on $L^2(G)\otimes \h$ for 
a separable infinite-dimensional Hilbert space $\h$. 
Recall that an extension of a trivially $\Z/2$-graded $C^*$-algebra $A$ by an equivariantly stable $C^*$-algebra 
$B\otimes \K_G$ is a $G$-equivariant short exact sequence
$$
0\to B\otimes \K_G\to E\to A\to 0.
$$
Since $B\otimes \K_G$ is an invariant ideal in $E$, there is an equivariant $*$-homomorphism 
$i_E:E\to \mathcal{M}(B\otimes \K_G)$.
The Busby invariant $\beta:A\to \mathcal{Q}(B\otimes \K_G):=\mathcal{M}(B\otimes \K_G)/B\otimes \K_G$ 
of an extension is an equivariant $*$-homomorphism given by choosing a linear splitting of $t:A\to E$ 
and composing with $i_E$ and the quotient mapping 
$\pi:\mathcal{M}(B\otimes \K_G)\to  \mathcal{Q}(B\otimes \K_G)$.
An extension is determined by its Busby invariant up to isomorphism. To be precise,
an equivariant $*$-homomorphism $\beta:A\to \mathcal{Q}(B\otimes \K_G)$ 
defines an equivariant extension $E_\beta$ by
$$
E_\beta:=\{(a,x)\in A\oplus \mathcal{M}(B\otimes \K_G): \beta(a)=\pi(x)\}.
$$
When $\beta$ comes from an extension $E$, the two extensions $E$ and $E_\beta$ are
stably isomorphic, and isomorphic when $\beta$ is injective.

The isomorphism classes of equivariant extensions form an 
abelian semigroup $E_G(A,B)$ under direct sum of 
Busby invariants after choosing an equivariant $*$-monomorphism $\K_G\oplus \K_G\to \K_G$.
We say that an extension with Busby invariant $\beta:A\rightarrow \mathcal{Q}(B\otimes \K_G)$ 
is \textbf{split} if there is an equivariant $*$-homomorphism 
$\wt{\beta}:A\rightarrow \mathcal{M}(B\otimes \K_G)$ 
that lifts $\beta$, i.e. $\beta=\pi\circ \wt{\beta}$. We say that an extension $E$ is semisplit if 
$E\to A$ admits a completely positive splitting. Averaging over the compact group $G$ guarantees that 
semisplit extensions admit equivariant completely positive splittings. The set of isomorphism classes of 
split extensions forms a sub-semigroup $D_G(A,B)\subseteq E_G(A,B)$, and so does the set of 
isomorphism classes of semisplit extensions.

\begin{defn}
The \textbf{extension group} $\Ext_G(A,B)$ is the semigroup quotient $E_G(A,B)/D_G(A,B)$. 
We define $\Ext^{-1}_G(A,B)$ as the sub-semigroup generated by semisplit extensions.
\end{defn}

For simplicity, if $G$ is the trivial group, we drop $G$ from the notations. 
For $A$ and $B$ trivially $\Z/2$-graded, 
there is a mapping $KK^1_G(A,B)\to \Ext_G(A,B)$ defined at the level of cycles as follows. 
Given an equivariant odd Kasparov module $(\rho,X_B,F)$ for $(A,B)$, we consider the mapping
$$
\beta_F:A\to \End^*_B(X_B)/\K_B(X_B), \quad \beta_F(a):=\pi\left(\frac{F+1}{2}\rho(a)\right),
$$
where $\pi:\End^*_B(X_B)\to \End^*_B(X_B)/\K_B(X_B)$ denotes the quotient map.
The conditions, for $a\in A$, $a(g^{-1}Fg-F)$, $[F,\rho(a)]$, $\rho(a)(F-F^*)$, $\rho(a)(F^2-1)\in \K_B(X_B)$ 
ensure the $\beta_F$ is an equivariant $\ast$-homomorphism. After stabilization, we can assume that 
$X_B=B\otimes \h\otimes L^2(G)$ and then the Calkin algebra is
$\End^*_B(X_B)/\K_B(X_B)=\mathcal{Q}(B\otimes \K_G)$.
The map $(\rho,X_B,F)\mapsto\beta_F$ induces a well-defined mapping 
$KK_G^1(A,B)\rightarrow\Ext_G(A,B)$. In fact
 this mapping is an isomorphism 
onto $\Ext^{-1}_G(A,B)$, \cite{thomsenequi}*{Theorem 9.2}. 

When $A$ is nuclear and $G$ is the trivial group, $\Ext^{-1}(A,B)=\Ext(A,B)$, and $\Ext(A,B)$ is a group. 
Unfortunately, the inverse mapping $\Ext^{-1}(A,B)\rightarrow KK_1(A,B)$ is not defined at the 
level of cycles and involves a choice of completely positive splitting. 

One advantage of the extension picture of odd 
$K$-homology is that the boundary mapping 
$KK_G^0(J\lhd A,B)\rightarrow KK_G^1(A,B)$ becomes easy to 
describe, and we do this now. Let $(\rho, X_B,F)$ be an equivariant even relative 
Kasparov module for $(J\lhd A,B)$, where $A$ and $B$ are
trivially $\Z/2$-graded, $A$ is separable, $B$ is $\sigma$-unital, and 
$A\to A/J$ is semisplit. 
Assume that $F$ is equivariant, self-adjoint and that $F^2$ is a projection.
Let $X_B=X_B^+\oplus X_B^-$ 
be the $\Z/2$-grading of $X_B$, and let 
$F_{\pm}:X_B^\pm\rightarrow X_B^\mp$ be the 
even-to-odd part and the odd-to-even part, respectively, of $F$. 
Our assumptions on $F$ 
guarantee that $F^+$ and $(F^+)^*$ are equivariant
partial isometries. We define $\wt{V}^+:=F^+$ and $\wt{V}^-:=(F^+)^*$. 
The modules $\ker(\wt{V}_\pm)$ are $G$-invariant complemented submodules of 
$X_B$, and we let $P_{\ker( \wt{V}_\pm)}$ 
be  the orthogonal projection onto $\ker( \wt{V}_\pm)$. 
Define extensions $\alpha_\pm$, of $A/J$ by
\begin{align*}
&\alpha_\pm:A/J\rightarrow \mathcal{Q}_B(\ker (\wt{V}_\pm)),\quad \alpha^\pm(a)
=\pi(P_{\ker( \wt{V}_\pm)}\wt{a}P_{\ker( \wt{V}_\pm)}),
\end{align*}
where $\wt{a}\in A$ is any lift of $a\in A/J$,  and 
$\pi:\End^*_B(\ker(\wt{V}_\pm))\rightarrow \mathcal{Q}_B(\ker( \wt{V}_\pm))$ is the quotient map. 
Recall our assumption that $A\to A/J$ is semisplit. 
By composing the equivariant completely positive splitting 
$A/J\to A$ with the equivariant completely positive mapping 
$A\to \End^*_B(\ker(\wt{V}_\pm))$, 
$b\mapsto P_{\ker( \wt{V}_\pm)}bP_{\ker( \wt{V}_\pm)}$, we obtain equivariant
completely positive mappings $\tau_\pm:A/J\to \End^*_B(\ker(\wt{V}_\pm))$ such that 
$\alpha_\pm=\pi\circ \tau_\pm$. Therefore $[\alpha_+],[\alpha_-]\in \Ext^{-1}_G(A/J,B)$. 
The following proposition follows immediately from \cite{cumero}*{Proposition 12.15}.

\begin{prop}
\label{eq:bdrymap}
Assume that $0\to J\to A\to A/J\to 0$ is a $G$-equivariant semisplit short exact sequence of trivially 
$\Z/2$-graded $C^*$-algebras. Let $x\in KK^0_G(J\triangleleft A,B)$ 
and assume that $(\rho,X_B,F)$ is 
a representative of $x$ by an equivariant relative 
Kasparov module where $F$ is equivariant, self-adjoint and $F^2$ 
is a projection. The image of $x$ under the boundary mapping 
$\partial: KK^0_G(J\triangleleft A,B)\to KK^1_G(A/J,B)$ is under the 
isomorphism $\Ext^{-1}_G(A/J,B)\cong KK^1_G(A/J,B)$ given 
by 
$$
\partial x=[\alpha_+]-[\alpha_-].
$$
\end{prop}

The case $B=\C$ was considered in \cite{BDT}*{p. 784} and \cite{HigsonRoe}*{Remark 8.5.7}.

The bounded transform of a relative unbounded Kasparov module will not generally square to a projection. 
The next couple of results show that we can nevertheless obtain such a module from 
a relative unbounded Kasparov module under an additional assumption, and hence write 
down the boundary of the relative unbounded Kasparov module in terms of extensions. 
These results give an alternative approach to that using the phase in Proposition \ref{prop:phasetriple}.

\begin{lemma}
\label{lem:doublinguparelativefredholmmodule}
Let $(\rho,X_B,F)$ be an even relative Kasparov module for $(J\lhd A,B)$. Let $T:X_B^+\rightarrow X_B^-$ 
be the even-to-odd part of $F$, and suppose that
\begin{itemize}
\item[1)]
$\|T\|\leq 1$,
\item[2)]
$\ker(T)$ and $\ker(T^*)$ are complemented submodules of $X_B^+$ and $X_B^-$ respectively, and 
$\rho(a)^+(1-T^*T)^{1/2}(1-P_{\ker(T)})$ and $\rho(a)^-(1-TT^*)^{1/2}(1-P_{\ker(T^*)})$ are $B$-compact 
for all $a\in A$, where $\rho(a)^j:X_B^j\rightarrow X_B$ is the decomposition of $\rho(a)$ with respect to 
the $\Z/2$-grading of $X_B$.
\end{itemize}
Let $\wt{F}=\left(\begin{smallmatrix}0&T^*\\T&0\end{smallmatrix}\right)\in \End_B^*(X_B)$, and let 
$$H=
\begin{pmatrix}
\wt{F}&(1-\wt{F}^2)^{1/2}(1-P_{\ker(\wt{F})})\\
(1-\wt{F}^2)^{1/2}(1-P_{\ker(\wt{F})})&-\wt{F}
\end{pmatrix}.$$
Then $(\rho\oplus0,(X\oplus X^{\text{op}})_B,H)$ is an even relative Kasparov module for 
$(J\lhd A,B)$, representing the same class as $(\rho,X_B,F)$.
\end{lemma}

\begin{proof}
The idea is similar to that of \cite{Blackadar}*{\S 17.6}. 
We note first that $[0,1]\ni t\mapsto tF+(1-t)\wt{F}$ 
is an operator homotopy of relative Kasparov modules, 
and thus $(\rho,X_B,F)$ and $(\rho,X_B,\wt{F})$ are 
relative Kasparov modules representing the same class. 
We then observe that the relative Fredholm module 
$(0,X_B^{\text{op}},-\wt{F})$ is degenerate and so 
$\left(\rho\oplus0,(X\oplus X^{\text{op}})_B,
\left(\begin{smallmatrix}\wt{F}&0\\0&-\wt{F}\end{smallmatrix}\right)\right)$ 
is a relative Kasparov module representing the 
same class as $(\rho,X_B,\wt{F})$. Since $\|\wt{F}\|\leq1$, 
the operator $H$ is well-defined, and we see that
\begin{align*}
(\rho\oplus0)(a)\left(H-\begin{pmatrix}\wt{F}&0\\0&-\wt{F}\end{pmatrix}\right)
&=\begin{pmatrix}0&\rho(a)(1-\wt{F}^2)^{1/2}(1-P_{\ker(\wt{F})})\\0&0\end{pmatrix},
\end{align*}
which is $B$-compact by 2), and thus $H$ is a ``relatively compact'' 
perturbation of $\left(\begin{smallmatrix}\wt{F}&0\\0&-\wt{F}\end{smallmatrix}\right)$, 
and so the straight line homotopy between $H$ and 
$\left(\begin{smallmatrix}\wt{F}&0\\0&-\wt{F}\end{smallmatrix}\right)$ 
is a homotopy of relative Kasparov modules.
\end{proof}
\begin{prop}\label{prop:extbdrymap}
Let $(\rho,X_B,F)$ be as in Lemma \ref{lem:doublinguparelativefredholmmodule}, 
and assume in addition that $A$ and $B$ are trivially $\Z/2$-graded, 
that $A\rightarrow A/J$ is semisplit, and that $F$ is equivariant. 
Let $T:X_B^+\rightarrow X_B^-$ be the even-to-odd part of $F$, 
and define extensions $\beta_j$, $j=0,1$ of $A/J$ by
\begin{align*}
&\beta_0:A/J\rightarrow\mathcal{Q}_B(\ker(T)),\quad \beta_0(a)=\pi(P_{\ker( T)}\wt{a}P_{\ker( T)}),\\
&\beta_1:A/J\rightarrow\mathcal{Q}_B(\ker( T^*)),\quad\beta_1(a)=\pi(P_{\ker(T^*)}\wt{a}P_{\ker (T^*)}),
\end{align*}
where $\wt{a}\in A$ is any lift of $A/J$ and $\pi:\End^*_B(\ker(T))\rightarrow\mathcal{Q}_B(\ker(T)):=
\End_B^*(\ker(T))/\K_B(\ker(T))$ is the quotient map. Then
\begin{align*}
\de[(\rho,X_B, F)]=[\beta_0]-[\beta_1]\in\Ext_G^{-1}(A/J)\cong KK^1_B(A/J,B).
\end{align*}
\end{prop}
\begin{proof}
The relative Kasparov module $(\rho,X_B,F)$ is equivalent to $(\rho\oplus0,(X\oplus X^{\text{op}})_B,H)$, 
where $H$ is as in Lemma \ref{lem:doublinguparelativefredholmmodule}. Since $H=H^*$ and 
$H^2=\left(\begin{smallmatrix}1-P_{\ker(\wt{F})}&0\\0&1-P_{\ker(\wt{F})}\end{smallmatrix}\right)$ 
is a projection, we can apply Equation \eqref{eq:bdrymap}. Since $\ker(H)=\ker(\wt{F})\oplus\ker(\wt{F})$,  
$\wt{F}^+=T$, $\wt{F}^-=T^*$, and the representation of $A$ is $\rho\oplus0$, the extensions 
$\alpha_0:A/J\rightarrow\mathcal{Q}_B(\ker(H^+))$ and $\alpha_1:A/J\rightarrow\mathcal{Q}_B(\ker(H^-))$ 
are equivalent to $\beta_0$ and $\beta_1$ respectively.
\end{proof}

Although the boundary map has a description 
in terms of extensions, getting from an extension back to a 
Kasparov module is not straightforward. Ideally we would 
want not just a Kasparov module representing the 
extension class, but an unbounded Kasparov module, e.g. a spectral triple. 
Spectral triples, and their generalizations unbounded 
Kasparov modules, carry geometric information as well as $K$-homological information,
which can often facilitate the computation of Kasparov products. 
One would like to be able to compute the boundary map in the unbounded setting,
an often difficult case of the product.

Since we know that the boundary map is (in principle) computable 
from relative Fredholm modules, the first step of computing the 
boundary map in the unbounded setting is to find 
unbounded representatives of relative Fredholm modules, 
just as spectral triples are unbounded representatives of 
Fredholm modules. These unbounded representatives are relative spectral triples.

\section{The proof of Lemma \ref{lem:adjointbound} and Theorem \ref{thm:alternate}}
\label{sec:technicalproof}

Recall that $A$ and $B$ are $\Z/2$-graded $G$-$C^*$-algebras, 
with $A$ separable and $B$ $\sigma$-unital, 
that $J$ is a $G$-invariant graded ideal in $A$, and that $(\J\lhd\A,X_B,\D)$ is 
a relative unbounded Kasparov module for $(J\lhd A,B)$.

\begin{proof}[Proof of Lemma \ref{lem:adjointbound}]
We first address some domain issues. Since 
$(1+\lambda+\D^*\D)(1+\lambda+\D^*\D)^{-1}=1$,
it follows that
$
(1+\lambda+\D^*\D)^{-1}:\Dom(\D)\rightarrow\{\zeta\in\Dom(\D^*\D):(1+\lambda+\D^*\D)\zeta\in\Dom(\D)\}.
$
Let 
$
\mu\in\{\zeta\in\Dom(\D^*\D):(1+\lambda+\D^*\D)\zeta\in\Dom(\D)\}$,
and let $\eta=(1+\lambda+\D^*\D)\mu$. Since $\Dom(\D^*\D)\subset\Dom(\D)$, $(1+\lambda)\mu\in\Dom(\D)$, and hence
$
\D^*\D\mu=\eta-(1+\lambda)\mu\in\Dom(\D).
$
That is,
\begin{align}
\label{eqn:domainissues}
\D^*\D(1+\lambda+\D^*\D)^{-1}\cdot\Dom(\D)\subset\Dom(\D).
\end{align}
Hence
\begin{align*}
&j\D(1+\lambda+\D^*\D)^{-1}-(1+\lambda+\D_e^*\D_e)^{-1}j\D\\
&=\Big(j\D-(1+\lambda+\D_e^*\D_e)^{-1}j\D(1+\lambda+\D^*\D)\Big)(1+\lambda+\D^*\D)^{-1}\\
&=\Big(j\D-(1+\lambda)(1+\lambda+\D_e^*\D_e)^{-1}j\D-(1+\lambda+\D_e^*\D_e)^{-1}j\D\D^*\D\Big)
(1+\lambda+\D^*\D)^{-1}\\
&\qquad \text{(this is well-defined by Equation \eqref{eqn:domainissues})}\\
&=\Big(\D_e^*\D_e(1+\lambda+\D_e^*\D_e)^{-1}j\D
-(1+\lambda+\D_e^*\D_e)^{-1}j\D\D^*\D\Big)(1+\lambda+\D^*\D)^{-1},
\end{align*}
since $1-(1+\lambda)(1+\lambda+x)^{-1}=x(1+\lambda+x)^{-1}$. By Lemma \ref{lem:domains},
$$
\D_e(1+\lambda+\D_e^*\D_e)^{-1}|_{\Dom(\D)}=(1+\lambda+\D_e\D_e^*)^{-1}\D^*|_{\Dom(\D)}
$$
since $\D_e\subset\D^*$. Since $j\cdot\Dom(\D^*)\subset\Dom(\D)$, 
\begin{align*}
&j\D(1+\lambda+\D^*\D)^{-1}-(1+\lambda+\D_e^*\D_e)^{-1}j\D=\\
&\big(\D_e^*(1+\lambda+\D_e\D_e^*)^{-1}\D^*j\D
-(1+\lambda+\D_e^*\D_e)^{-1}j\D\D^*\D\big)(1+\lambda+\D^*\D)^{-1}.
\end{align*}
Similarly, Lemma \ref{lem:domains} and the facts that $\D\subset\D_e^*$ and $j\cdot\Dom(\D)\subset\Dom(\D)$ imply that 
$$
\D_e^*(1+\lambda+\D_e\D_e^*)^{-1}j|_{\Dom(\D)}=(1+\lambda+\D_e^*\D_e)^{-1}\D j|_{\Dom(\D)}.
$$
Hence
\begin{align*}
&j\D(1+\lambda+\D^*\D)^{-1}-(1+\lambda+\D_e^*\D_e)^{-1}j\D\\
&=\big(\D_e^*(1+\lambda+\D_e\D_e^*)^{-1}\D^*j\D
-(-1)^{\deg j}\D_e^*(1+\lambda+\D_e\D_e^*)^{-1}j\D^*\D\\
&\qquad\quad+(-1)^{\deg j}(1+\lambda+\D_e^*\D_e)^{-1}\D j\D^*\D
-(1+\lambda+\D_e^*\D_e)^{-1}j\D\D^*\D\big)
(1+\lambda+\D^*\D)^{-1}\\
&=\D_e^*(1+\lambda+\D_e\D_e^*)^{-1}[\D^*,j]_\pm\D(1+\lambda+\D^*\D)^{-1}\\
&\qquad\quad+(-1)^{\deg j}(1+\lambda+\D_e^*\D_e)^{-1}[\D,j]_\pm\D^*\D(1+\lambda+\D^*\D)^{-1}.
\end{align*}
Thus
\begin{align*}
\ol{j\D(1+\lambda+\D^*\D)^{-1}-(1+\lambda+\D_e^*\D_e)^{-1}j\D}&
=\D^*_e(1+\lambda+\D_e\D_e^*)^{-1}\ol{[\D,j]_\pm}\D(1+\lambda+\D^*\D)^{-1}\\
&\hspace{-14pt}+(-1)^{\deg j}(1+\lambda+\D_e^*\D_e)^{-1}\ol{[\D,j]_\pm}\D^*\D(1+\lambda+\D^*\D)^{-1},
\end{align*}
where we used Lemma \ref{lem:commutator} 2) for $\ol{[\D^*,j]_\pm}=\ol{[\D,j]_\pm}$, and so
the estimates in parts 3) and 4) of Lemma \ref{lem:domains} yield
\begin{align*}
&\|\ol{j\D(1+\lambda+\D^*\D)^{-1}-(1+\lambda+\D_e^*\D_e)^{-1}j\D}\|\leq\frac{2\|\ol{[\D,j]_\pm}\|}{1+\lambda}.&&\qedhere
\end{align*}
\end{proof}

We break the proof of Theorem \ref{thm:alternate} into several lemmas. 
Recall that $F$ is the bounded transform $F=\D(1+\D^*\D)^{-1/2}$.
The main tool used to prove that the bounded transform of a relative unbounded Kasparov module is a 
relative Kasparov module is the integral formula for fractional powers, \cite{Pedersen}*{p. 8},
which we use in the form \cite{CP1}
\begin{align}\label{eq:iffp}
(1+\D^*\D)^{-1/2}=\frac{1}{\pi}\int_0^\infty\lambda^{-1/2}(1+\lambda+\D^*\D)^{-1}\,d\lambda.
\end{align}
We would like to be able to take terms such as $\D[(1+\D^*\D)^{-1/2},a]_\pm$ and use \eqref{eq:iffp} to write
$$
\D[(1+\D^*\D)^{-1/2},a]_\pm
=\frac{1}{\pi}\int_0^\infty\lambda^{-1/2}\D[(1+\lambda+\D^*\D)^{-1},a]\,d\lambda.
$$
For this expression to be well-defined requires that the 
integral converges in operator norm, and this is why we need the norm estimates in parts 3) and 4)
of Lemma \ref{lem:domains}.

\begin{lemma}
\label{lem:compactcommutator}
 Let $(\J\lhd\A,X_B,\D)$ be an
equivariant relative unbounded Kasparov module for $(J\lhd A,B)$, and let 
$F=\D(1+\D^*\D)^{-1/2}$ be the bounded transform 
of $\D$. The graded commutator $[F,a]_\pm$ is $B$-compact for all $a\in A$.
\end{lemma}

\begin{proof}
By density it is enough to show that $[F,ab]_\pm$ is $B$-compact 
for all $a,b\in\A$. 
For $a,b\in\A$ of homogeneous degree, we write
$$[F,ab]_\pm=[F,a]_\pm b+(-1)^{\deg a}a[F,b]_\pm.$$
The first term on the right can be expressed as
$$[F,a]_\pm b=[\D,a]_\pm(1+\D^*\D)^{-1/2}b+\D[(1+\D^*\D)^{-1/2},a] b.$$
The first term is $B$-compact, and using the integral formula for fractional powers 
together with Lemma \ref{lem:commutator} and the estimates in 
parts 3) and 4) of Lemma \ref{lem:domains}, the second term
\begin{align*}
\D[(1+\D^*\D)^{-1/2},a]b&=-\frac{1}{\pi}\int_0^\infty\lambda^{-1/2}\big(\D\D^*(1+\lambda+\D\D^*)^{-1}[\D,a]_\pm(1+\lambda+\D^*\D)^{-1}b\\
&\qquad\qquad+(-1)^{\deg a}\D(1+\lambda+\D^*\D)^{-1}[\D^*,a]_\pm\D(1+\lambda+\D^*\D)^{-1}b\big)\,d\lambda
\end{align*}
is a norm convergent integral, which is $B$-compact. 
Thus $[F,ab]_\pm$ is $B$-compact if and only if 
$a[F,b]_\pm$ is $B$-compact. Taking adjoints and 
using the integral formula for fractional powers, we see that
\begin{align*}
[F^*,b^*]_\pm a^*&=\frac{1}{\pi}\int_0^\infty\lambda^{-1/2}[\D^*(1+\lambda+\D\D^*)^{-1},b^*]_\pm a^*\,d\lambda\\
&=\frac{1}{\pi}\int_0^\infty\lambda^{-1/2}\big(-\D^*\D(1+\lambda+\D^*\D)^{-1}[\D^*,b^*]_\pm(1+\lambda+\D\D^*)^{-1}a^*\\
&\qquad\qquad-(-1)^{\deg b}\D^*(1+\lambda+\D\D^*)^{-1}[\D,b^*]_\pm\D^*(1+\lambda+\D\D^*)^{-1}a^*\\
&\qquad\qquad+[\D^*,b^*]_\pm(1+\lambda+\D\D^*)^{-1}a^*\big)\,d\lambda
\end{align*}
For $\xi\in\Dom(\D)$, and $a,b\in\A$ of homogeneous degree, and using Lemma \ref{lem:domains} (2),
\begin{align*}
&[\D^*(1+\lambda+\D\D^*)^{-1},a]_\pm b\xi=[(1+\lambda+\D^*\D)^{-1}\D,a]_\pm b\xi\\
&=\big((1+\lambda+\D^*\D)^{-1}\D a(1+\lambda+\D^*\D)\\
&\qquad-(-1)^{\deg a}a(1+\lambda+\D^*\D)^{-1}\D(1+\lambda+\D^*\D)\big)(1+\lambda+\D^*\D)^{-1}b\xi\\
&=\big(\D a-\D^*\D(1+\lambda+\D^*\D)^{-1}\D a+(1+\lambda+\D^*\D)^{-1}\D a\D^*\D\\
&\qquad-(-1)^{\deg a}a\D+(-1)^{\deg a}a\D^*\D(1+\lambda+\D^*\D)^{-1}\D\\
&\qquad-(-1)^{\deg a}a(1+\lambda+\D^*\D)^{-1}\D\D^*\D\big)(1+\lambda+\D^*\D)^{-1}b\xi\\
&\text{(since $(1+\lambda)(1+\lambda+x)^{-1}=1-x(1+\lambda+x)^{-1}$)}\\
&=\big([\D,a]_\pm-\D^*\D(1+\lambda+\D^*\D)^{-1}[\D,a]_\pm-(-1)^{\deg a}\D^*\D(1+\lambda+\D^*\D)^{-1}a\D\\
&\qquad+(1+\lambda+\D^*\D)^{-1}\D a\D^*\D+(-1)^{\deg a}a\D^*\D(1+\lambda+\D^*\D)^{-1}\D\\
&\qquad-(-1)^{\deg a}a(1+\lambda+\D^*\D)^{-1}\D\D^*\D\big)(1+\lambda+\D^*\D)^{-1}b\xi.
\end{align*}
Using $\D^*(1+\lambda+\D\D^*)^{-1/2}|_{\Dom(\D)}=(1+\lambda+\D^*\D)^{-1/2}\D$, 
we can rearrange the right-hand side to find a formula which is valid on all of $X_B$, 
which must also agree with the left-hand side on this larger domain since all the terms are bounded. 
Thus we conclude that
\begin{align*}
&[\D^*(1+\lambda+\D\D^*)^{-1},a]_\pm b\\
&=\big([\D,a]_\pm-\D^*\D(1+\lambda+\D^*\D)^{-1}[\D,a]_\pm-(-1)^{\deg a}\D^*\D(1+\lambda+\D^*\D)^{-1}a\D\\
&\qquad+\D^*(1+\lambda+\D\D^*)^{-1}a\D^*\D+(-1)^{\deg a}a\D^*\D(1+\lambda+\D^*\D)^{-1}\D\\
&\qquad-(-1)^{\deg a}a(1+\lambda+\D^*\D)^{-1/2}\D^*(1+\lambda+\D\D^*)^{-1/2}\D^*\D\big)(1+\lambda+\D^*\D)^{-1}b.
\end{align*}
The $B$-compactness of every term on the right-hand side except the fourth 
follows from the $B$-compactness of $(1+\lambda+\D^*\D)^{-1/2}b$ and 
$a(1+\lambda+\D^*\D)^{-1/2}$. So the left-hand side is $B$-compact if and only if
\begin{align}\label{eq:isthiscompact}
\D^*(1+\lambda+\D\D^*)^{-1}a\D^*\D(1+\lambda+\D^*\D)^{-1}b
\end{align}
is $B$-compact. \eqref{eq:isthiscompact} is $B$-compact if $\D^*(1+\lambda+\D\D^*)^{-1}a$ is $B$-compact, 
which in turn is $B$-compact if
$(1-P_{\ker(\D^*)})(1+\lambda+\D\D^*)^{-1/2}a$ is $B$-compact. 
Hence if $(1-P_{\ker(\D^*)})(1+\lambda+\D\D^*)^{-1/2}a$ is $B$-compact for all $a\in\A$, then the integrand of
$$[F^*,a]_\pm b=\frac{1}{\pi}\int_0^\infty\lambda^{-1/2}[\D^*(1+\lambda+\D\D^*)^{-1},a]_\pm b\,d\lambda$$
is $B$-compact for all $a,b\in\A$, and thus $[F,ab]_\pm$ is $B$-compact for all $a,b\in\A$. 
The $B$-compactness of $(1-P_{\ker(\D^*)})(1+\lambda+\D\D^*)^{-1/2}a$ follows from the $B$-compactness
of $(1-P_{\ker(\D^*)})(1+\D\D^*)^{-1/2}a$ and the first resolvent formula, and so the lemma is proven.
\end{proof}

\begin{lemma}
\label{lem:compactcommutatoralt}
 Let $(\J\lhd\A,X_B,\D)$ be a triple satisfying all conditions to be a relative spectral triple for $(J\lhd A,B)$, 
except possibly Condition 4., and suppose this triple satisfies Condition 4'. of Remark \ref{rem:alternatecondition}. Let 
$F=\D(1+\D^*\D)^{-1/2}$ be the bounded transform of $\D$. 
Then the graded commutator $[F,a]_\pm$ is $B$-compact for all $a\in A$.
\end{lemma}

\begin{proof}
Since $\A$ is dense in $A$ it is enough to show that $[F,a]_\pm$ 
is $B$-compact for all $a\in\A$. Let $a\in\A$ be of homogeneous degree. We can write
\begin{align}
\label{eq:splitup}
[F,a]_\pm=[\D,a]_\pm(1+\D^*\D)^{-1/2}+\D[(1+\D^*\D)^{-1/2},a].
\end{align}
Let $(\phi_k)_{k=1}^\infty\subset\A$ 
be a sequence such that $\phi_k[\D,a]_\pm$ and $[\D,a]_\pm\phi_k$
converge to $[\D,a]_\pm$ in operator norm for all $a\in\A$.  Then
$$
[\D,a]_\pm(1+\D^*\D)^{-1/2}=\lim_{k\rightarrow\infty}\ol{[\D,a]_\pm}\phi_k(1+\D^*\D)^{-1/2}
$$
and so the first term of \eqref{eq:splitup} is $B$-compact. 
By Lemma \ref{lem:commutator} 3) and the integral 
formula for fractional powers \eqref{eq:iffp}, the second term of \eqref{eq:splitup} is
\begin{align*}
\D[(1+\D^*\D)^{-1/2},a]
&=-\frac{1}{\pi}\int_0^\infty\lambda^{-1/2}
\Big(\D\D^*(1+\lambda+\D\D^*)^{-1}[\D,a]_\pm(1+\lambda+\D^*\D)^{-1}\\
&+(-1)^{\deg a}\D(1+\lambda+\D^*\D)^{-1}[\D^*,a]_\pm\D(1+\lambda+\D^*\D)^{-1}\Big)\,d\lambda.
\end{align*}
which is an operator norm convergent integral by Lemma \ref{lem:commutator} 2) and the
estimates in parts 3) and 4) of Lemma \ref{lem:domains}.
The integrand is $B$-compact, 
since we have the operator norm limits of $B$-compact operators
\begin{align*}
&\D\D^*(1+\lambda+\D\D^*)^{-1}[\D,a]_\pm(1+\lambda+\D^*\D)^{-1}
\!=\!\lim_{k\rightarrow\infty}\!\D\D^*(1+\lambda+\D\D^*)^{-1}\ol{[\D,a]_\pm}\phi_k(1+\lambda+\D^*\D)^{-1},
\end{align*}
and
\begin{align*}
&\D(1+\lambda+\D^*\D)^{-1}[\D^*,a]_\pm\D(1+\lambda+\D^*\D)^{-1}
=\lim_{k\rightarrow\infty}\D(1+\lambda+\D^*\D)^{-1}\phi_k\ol{[\D,a]_\pm}\D(1+\lambda+\D^*\D)^{-1}.
\end{align*}
Since the integral converges in operator norm, the integral is 
$B$-compact as well. Hence $[F,a]_\pm$ is $B$-compact for all $a\in\A$.
\end{proof}

\begin{lemma}
\label{lem:theotherstuffiscompacttoo}
Let $(\J\lhd\A,X_B,\D)$ be an
equivariant relative unbounded Kasparov module for $(J\lhd A,B)$.
The operators $j(F-F^*)$ and $j(1-F^2)$ are $B$-compact
 for all $j\in J$. This also holds without Condition 4. of Definition \ref{defn:alternate}.
\end{lemma}

For a proof see \cite{Hilsum}*{Theorem 3.2}. This result does not require either  Condition 4. or 4'.

Theorem \ref{thm:alternate} now follows, also if Condition 4. of Definition \ref{defn:alternate} is replaced by 
Condition 4'. of Remark \ref{rem:alternatecondition}, by noting that the equivariance of $\D$ implies 
the equivariance of the bounded transform $F$, and combining Lemmas \ref{lem:compactcommutator}, 
\ref{lem:compactcommutatoralt} and \ref{lem:theotherstuffiscompacttoo}.

\end{document}